%% file: manuscript_inchworm.tex
\documentclass[3p,final]{elsarticle}

\biboptions{sort&compress}
\makeatletter
\def\ps@pprintTitle{%
 \let\@oddhead\@empty
 \let\@evenhead\@empty
 \def\@oddfoot{}%
 \let\@evenfoot\@oddfoot}
\makeatother

\usepackage{physics, amssymb, amsthm}
\usepackage{mathrsfs}
\usepackage{xcolor}
\usepackage{subcaption}
\usepackage{tikz,pgfplots}
\usetikzlibrary{patterns,snakes}
\usepackage{scalerel,stackengine}
\stackMath
\newcommand\reallywidehat[1]{%
\savestack{\tmpbox}{\stretchto{%
  \scaleto{%
    \scalerel*[\widthof{\ensuremath{#1}}]{\kern-.6pt\bigwedge\kern-.6pt}%
    {\rule[-\textheight/2]{1ex}{\textheight}}
  }{\textheight}%
}{0.5ex}}%
\stackon[1pt]{#1}{\tmpbox}%
}
\usepackage{hyperref}

\usepackage{cancel}

\newtheorem{theorem}{Theorem}
\newtheorem{lemma}{Lemma}
\newtheorem{proposition}[theorem]{Proposition}
{\theoremstyle{remark} \newtheorem{remark}{Remark}}

\def\C{\mathbb{C}}
\def\Z{\mathbb{Z}}

\def\ii{\mathrm{i}}
\def\gb{\boldsymbol{g}}
\def\eb{\boldsymbol{e}}
\def\yb{\boldsymbol{y}}

\def\sb{\boldsymbol{s}}
\def\xib{\boldsymbol{\xi}}

\def\etab{\boldsymbol{\eta}}

\newcommand\TT{\mathrm{T}}
\newcommand\ee{\mathrm{e}}
\newcommand\FF{\mathrm{F}}
\newcommand{\mc}[1]{\mathcal{#1}}

\def\E{\mathbb{E}}
\def\Is{\mathcal{I}}
\def\Js{\mathcal{J}}
\def\Ks{\mathcal{K}}
\def\Ls{\mathcal{L}}
\def\Gs{\mathcal{G}}
\def\Ds{\mathcal{D}}
\def\Fs{\mathcal{F}}
\def\Us{\mathcal{U}}
\def\Hs{\mathcal{H}}

\def\Ns{\mathcal{N}}
\newcommand\mQ{\mathcal{Q}}
\newcommand{\mf}[1]{\mathfrak{#1}}

\def\bdG{\mathscr{G}}
\def\bdL{\mathscr{L}}
\def\bdH{\mathscr{H}}
\def\bdW{\mathscr{W}}

\newcommand{\redtext}[1]{{\color{red}#1}}

\newcommand\Sarr{s_{\uparrow}}
\def\sa{s_{\mathrm{f}}}
\def\si{s_{\mathrm{i}}}
\def\sf{s_{\mathrm{f}}}
\def\sgn{\text{sgn}}
\def\std{\text{std}}
\def\bG{\bar{G}}
\def\tG{\widetilde{G}}
\def\dG{\Delta G}
\def\bg{\bar{\gb}}

\def\tg{\widetilde{\gb}}
\def\hg{\vec{\gb}}
\def\htg{\vec{\tg}}
\def\bK{\bar{K}}
\def\tK{\widetilde{K}}
\def\bk{\bar{k}}
\def\tk{\widetilde{k}}
\def\Ge{G_{\text{e}}}
\def\ge{\gb^{\text{e}}}
\def\ges{\gb^{\text{e}*}}
\def\tu{\widetilde{u}}
\def\bu{\bar{u}}
\def\rkn{|u_{\mathrm{RK}}|_{\mathrm{std}}}

\DeclareMathOperator\Var{Var}

\begin{document}

\begin{frontmatter}

\title{Numerical analysis for inchworm Monte Carlo method: Sign problem and error growth}

\author[]{Zhenning Cai\fnref{fn1}}
\ead{matcz@nus.edu.sg}

\author[]{Jianfeng Lu\fnref{fn2}}
\ead{jianfeng@math.duke.edu}

\author[]{Siyao Yang\fnref{fn1}}
\ead{siyao_yang@u.nus.edu}

\fntext[fn1]{Department of Mathematics, National University of
  Singapore, Level 4, Block S17, 10 Lower Kent Ridge Road, Singapore 119076.}
\fntext[fn2]{Department of Mathematics, Department of Physics, and
  Department of Chemistry, Duke University, Box 90320, Durham NC 27708, USA.}

\begin{abstract}
We consider the numerical analysis of the inchworm Monte Carlo method, which is proposed recently to tackle the numerical sign problem for open quantum systems. We focus on the growth of the numerical error with respect to the simulation time, for which the inchworm Monte Carlo method shows a flatter curve than the direct application of Monte Carlo method to the classical Dyson series. To better understand the underlying mechanism of the inchworm Monte Carlo method, we distinguish two types of exponential error growth, which are known as the numerical sign problem and the error amplification. The former is due to the fast growth of variance in the stochastic method, which can be observed from the Dyson series, and the latter comes from the evolution of the numerical solution. Our analysis demonstrates that the technique of partial resummation can be considered as a tool to balance these two types of error, and the inchworm Monte Carlo method is a successful case where the numerical sign problem is effectively suppressed by such means. We first demonstrate our idea in the context of ordinary differential equations, and then provide complete analysis for the inchworm Monte Carlo method. Several numerical experiments are carried out to verify our theoretical results.
\end{abstract}

\begin{keyword}
 Open quantum system \sep inchworm Monte Carlo method \sep numerical sign problem \sep error growth 
\end{keyword}

\end{frontmatter}

\section{Introduction}
In quantum mechanics, an open quantum system refers to a quantum system interacting with the environment. In reality, no quantum system is absolutely isolated, and therefore the theory of open quantum systems has wide applications including quantum thermodynamics \cite{Esposito2009}, quantum information science \cite{Shor1995}, and quantum biology \cite{Asano2016}. Due to interaction with the environment, the quantum system is irreversible \cite{Manicino2018}, and its master equation can be obtained by Nakajima-Zwanzig projection technique \cite{Nakajima1958, Zwanzig1960}, which is an integro-differential equation showing that the dynamics is non-Markovian. When the coupling between the quantum system and the environment is weak, Markovian approximation can be used to simplify the simulation \cite{Lindblad1976}, while for non-Markovian simulations, one needs to apply more expensive methods such as QuAPI (quasi-adiabatic propagator path integral) \cite{Makri1992,Makri1993} and HEOM (hierarchical equations of motion) \cite{Ishizaki2005}. In this paper, we are interested in the numerical analysis for the inchworm algorithm \cite{Chen2017, Cai2020}, which is a recently proposed diagrammatic Monte Carlo method for open quantum system. The inchworm algorithm was originally proposed in \cite{Cohen2015} for impurity models. In \cite{Cai2020}, the method is recast in a continuous form as an integro-differential equation, so that classical numerical techniques can be applied.

In the integro-differential equation formulation of the inchworm method, the time derivative of the propagator is written as an expression involving an infinite series and high-dimensional integrals. Therefore, the numerical method involves both a Runge-Kutta part for time marching  and a Monte Carlo part to deal with the series and integrals. In abstract form, we write the equation as
\begin{equation} \label{eq:general equation}
\frac{\dd u}{\dd t} = \mathit{RHS} = \E_X R(X),
\end{equation}
where $\E_X$ denotes the expectation with respect to the random variable $X$, and both $\mathit{RHS}$ and $R(X)$ are functions of the solution $u$. While we motivate \eqref{eq:general equation} using the inchworm method, such type of equations also arises in many other contexts, and thus our analysis applies in a wider context.

Consider using the forward Euler method as the time integrator, combined with Monte Carlo estimate of the $\mathit{RHS}$, the numerical scheme is
\begin{equation} \label{eq:forward Euler}
u_{n+1} = u_n + \frac{h}{N_s} \sum_{i=1}^{N_s} R(X_i^{(n)}),
\end{equation}
where $h$ is the time step, and $X_i^{(n)}$ are random variables drawn from the probability distribution of $X$. Such a scheme is highly related to a number of existing methods such as the Direct simulation Monte Carlo \cite{Bird1963, Bird1994}, stochastic gradient descent method \cite{Zhang2004}, and the random batch method \cite{Jin2020}. The qDRIFT method proposed in \cite{Campbell2019} is also a variant of \eqref{eq:forward Euler} by replacing the forward Euler method with an exact solver in the context of Hamiltonian simulation. The scheme \eqref{eq:forward Euler} can be easily extended to general Runge-Kutta methods, which is found in \cite{Cai2020} to be useful in the simulation of open quantum systems. The numerical analysis of such a method has been carried out for differential-type equations in several cases \cite{Li2017, Hu2019, Jin2020}. When such methods are applied to systems with dissipation \cite{Li2017,Li2020}, the numerical error can be well controlled by the intrinsic property of the system for long-time simulations. However, in quantum mechanics, where the propagators remain unitary for any $t$, it is often seen that the error grows rapidly with respect to time in the real-time simulations \cite{Cai2018,MacKernan2002}, as is known as the ``numerical sign problem'', or more specifically the ``dynamical sign problem'' in the context of open quantum systems \cite{Muhlbacher2008,Werner2009,Schiro2010}.

The purpose of the inchworm Monte Carlo method is to mitigate the numerical sign problem when simulating the open quantum system by Dyson series expansion \cite{Chen2017,Cai2020}. The numerical sign problem, which will be further elaborated in \S~\ref{nsp}, refers to the stochastic error when applying Monte Carlo method to estimate the sum or the integral of highly oscillatory, high-dimensional functions. This is an intrinsic and notorious difficulty for simulating many-body quantum systems, such as in condensed matter physics \cite{Loh1990} and lattice field theory \cite{Cristoforetti2012}. For open quantum systems, the numerical sign problem 
becomes more severe as the simulation time gets longer 
\cite{Muhlbacher2008,Werner2009,Schiro2010}. Specifically, the average numerical error is proportional of the exponential of $t^2$, as introduces great difficulty for long time simulations. Inchworm Monte Carlo method adopts the idea of ``partial resummation'', and has successfully reduced the numerical error in a number of applications \cite{Dong2017,Ridley2018,Eidelstein2020}. However, as mentioned in \cite{Chen2017b}, it is not totally clear how the inchworm Monte Carlo method mitigates the numerical sign problem, despite some intuition coming from the idea of partial resummation. In this work, our aim is to demystify such mechanism by a deep look into the evolution of the numerical error. We find that the error of the inchworm Monte Carlo method grows as the exponential of a polynomial of $t$. However, the source of such error growth is not the numerical sign problem. The reason of the fast growth mainly comes from the amplification of the error at previous time steps, which is more similar to the error amplification in Runge-Kutta methods for ordinary differential equations. By separating these two types of error growth--- numerical sign problem and error amplification, we find that partial resummation can be regarded as a tool to trade-off the two types of error, as may help flatten the error growth curve in certain cases. We hope this also helps understand a more general class of iterative numerical methods for computing summations \cite{Makri1995a,Prokof'ev2007,Makri2017,Li2019}.

In fact, such understanding of error balance can be already revealed in the context of ODEs, which we will first focus on to save the involved notations in the inchworm Monte Carlo method. Thus, in Section~\ref{sec: diff eq}, we carry out the error analysis of differential equations for general Runge-Kutta methods with Monte Carlo evaluation of the right-hand side. The results are of independent interests, and the analysis also serves as a simple context to understand how partial resummation transforms the mechanism of error growth from numerical sign problem to error amplification. Afterwards, a detailed analysis for the inchworm Monte Carlo method will be given, which reveals the behavior of the error growth in the inchworm Monte Carlo method, and explains whether/how it relaxes the numerical sign problem. In Section \ref{sec: inchworm intro}, we introduce the inchworm Monte Carlo method based on an integro-differential equation formulation, and we present the corresponding main results of numerical analysis and their implication in Section \ref{sec: int diff eq}. Our analytical results are verified by several numerical tests in Section \ref{sec: numer exp}, showing the agreement between the theory and the experiments.
The rigorous proofs for the error analysis of differential equation and inchworm Monte Carlo equation are later given in Section \ref{sec: proof diff eq} and Section \ref{sec: proof} respectively. Finally, some concluding remarks are given in Section \ref{sec: conclusion}.

\section{A stochastic numerical method for differential equations}
\label{sec: diff eq}
To demonstrate the methodology of numerical analysis for the equations with the form \eqref{eq:general equation}, we first consider the simple case of an ordinary differential equation:
\begin{equation}\label{eq: diff eq}
\frac{\dd u}{\dd t} = f(t,u(t)), \quad t \in [0,T],
\end{equation}
where $u:[0,T] \rightarrow \C^{d}$ and the right-hand side $f$ is $(p+1)$-times continuously differentiable. A general $s$-stage explicit Runge-Kutta method of order $p$ reads 
\begin{equation}\label{def: scheme diff rk}
\begin{aligned}
& u_{n+1} = u_n + h \sum^s_{i=1} b_i k_i, \\
& k_i = f\Big(t_n + c_i h,u_n + h\sum^{i-1}_{j= 1}a_{ij}k_j \Big), \quad i = 1,\cdots,s.
\end{aligned}
\end{equation}
For simplicity, we assume that the time step $h = T/N$ is smaller than $1$, and $u_0$ is given by the initial condition $u_0 = u(0)$. The error estimation of Runge-Kutta methods is standard and can be found in textbooks such as \cite{Hairer:1993:SOD:153158}.

As mentioned in the introduction, we now consider a special scenario where the right-hand side of this equation can be represented as the expectation of a stochastic variable:
\begin{equation} 
\label{def: g}
f(t,u) = \E_X [g(t,u,X)], \qquad \forall t, u
\end{equation}
with $X$ being a random variable subject to a given distribution. Inspired by the Runge-Kutta method, we consider the following numerical scheme:
\begin{equation}\label{def: scheme diff mc}
\tu_{n+1} = \tu_n + h \sum^s_{i=1} b_i \tk_i,~0 \le n \le N
\end{equation}
where 
\begin{equation}
\tk_i = \frac{1}{N_s} \sum^{N_s}_{l = 1} g\Big(t_n + c_i h, \tu_n + h \sum^{i-1}_{j = 1} a_{ij} \tk_j,X^{(i)}_l\Big)
\end{equation}
with the initial condition $\tu_0 = u(0)$. Here $X^{(i)}_l$ are independent samples generated from the probability distribution of $X$.  

In this section, we will look for the gap between these two numerical methods \eqref{def: scheme diff rk} and  \eqref{def: scheme diff mc}. In specific, we aim to bound the \emph{bias} $\|\E(u_N-\tu_N)\|_2$ and the \emph{numerical error} $[\E(\|u_N-\tu_N\|_2^2)]^{1/2}$. By combining these errors with the error estimation of Runge-Kutta methods, the final error bounds can be obtained simply by triangle inequality:
\begin{align*}
\|\E(u(T)-\tu_N)\|_2 &\le
\|u(T) - u_N\|_2 + \|\E(u_N - \tu_N)\|_2, \\
[\E(\|u(T)-\tu_N\|^2_2)]^{1/2} &\le
\|u(T) - u_N\|_2 + [\E(\|(u_N - \tu_N)\|^2_2)]^{1/2},
\end{align*}
where $\|u(T) - u_N\|_2$ is the numerical error for the standard Runge-Kutta method, whose analysis can be found in a number of textbooks.

\subsection{Main results for differential equations}
\label{sec: diff results}
In this section, we will list the main results of our error analysis. The results are based on the following working hypothesis:
\begin{equation}
\label{assump: bd} \|\nabla_u g^{(m)}(t,u,X)\|_2 \le M', \quad \|\nabla_u f^{(m)}(t,u) \|_2 \le M', \quad \|\nabla_{u}^2 f^{(m)}(t,u)\|_{\FF} \le M'',
\end{equation}
where $f^{(m)}$ denotes the $m$th component of $f$, and $M$, $M'$ and $M''$ are constants independent of $t$, $u$ and $X$. For the Runge-Kutta solution, we define
\begin{gather*}
|u^{(m)}_{\text{RK}}|_{\text{std}}  = \max_n \max_{i=1,\cdots,s} \sqrt{\Var g^{(m)}\Big(t_n + c_i h, u_n + h \sum_{j=1}^{i-1} a_{ij} k_j, X \Big)} \redtext{,} \\
\rkn = \left( \sum_{m=1}^d |u^{(m)}_{\text{RK}}|^2_{\text{std}} \right)^{\frac{1}{2}} . 
\end{gather*}
For simplicity, below we use $R$ to denote the upper bound of all the coefficients appearing in the Runge-Kutta method \eqref{def: scheme diff rk}. Precisely, we assume that
\begin{equation}\label{assump: rk bd}
|a_{ij}|,|b_i|,|c_i|\le R \text{~for all~}i,j.
\end{equation}
Thus, the recurrence relations of both the bias and the numerical error can be established as:
\begin{proposition}\label{thm: diff recurrence relations}
Given a sufficiently small time step length $h$ and a sufficiently large number of samples at each step $N_s$. If the boundedness assumptions \eqref{assump: bd} hold, we have the recurrence relations 
\begin{equation}\label{eq: diff recurrence 1}
\|\E(u_{n+1}-\tu_{n+1})\|_2 \le \ (1+\alpha h)\|\E(u_{n}-\tu_{n})\|_2 + \alpha h \Big( \E\big(\|u_n -\tu_n\|_2^2\big) + \frac{h^2}{N_s} \rkn^2 \Big)
\end{equation}
and
\begin{equation}\label{eq: diff recurrence 2}
\E(\|u_{n+1}-\tu_{n+1}\|_2^2) \le (1+\beta h )  \E(\|u_{n}-\tu_{n}\|_2^2) + \beta \Big( \frac{h^2}{N_s} \rkn^2 + \frac{\alpha^2 h^5}{s^2 R^2 N_s^2} \rkn^4 \Big),
\end{equation}
where $\alpha = 2^s sR\sqrt{d} \max(M', 2 s M'', 2^{s+2} s^3 M'^2 M'' R^2, 2^{s+1} s^2M'' R^2)$\\ and $\beta = \max(4+2^{s+3}M'^2 d R^2 s^3, 2^{s+1} R^2 s^2)$.
\end{proposition}

Next, we apply the two recurrence relations above and accumulate the two errors step by step. We will reach the following estimates:
\begin{theorem}\label{thm: diff bounds}
Under the settings in Proposition \ref{thm: diff recurrence relations}, we have 

Bias estimation
\begin{equation}\label{diff 1st error upper bound}
\|\E(u_{N}-\tu_{N})\|_2 \le \left[ \frac{h^2}{N_s}\big( e^{\alpha T}-1\big) + \alpha T \Big( \frac{h}{N_s} + \frac{\alpha^2 h^4}{s^2 R^2 N_s^2} \rkn^2 \Big) \big(e^{\max(\alpha,\beta) T} -1\big) \right]  \rkn^2.
 \end{equation}

Numerical error estimation
\begin{equation}\label{diff 2ed error upper bound}
\E(\|u_N-\tu_N\|_2^2) \le   \big( e^{\beta T}-1\big) \Big( \frac{h}{N_s} + \frac{\alpha^2 h^4}{s^2 R^2 N_s^2} \rkn^2 \Big) \rkn^2.
\end{equation}
\end{theorem}

This theorem shows that although the stochastic scheme is biased, the bias
plays a minor role in the numerical simulation since the stochastic noise, estimated by the square root of \eqref{diff 2ed error upper bound}, is significantly larger. It is also worth mentioning that the constant $\beta$ depends only on the Runge-Kutta scheme and the bound of the first-order derivatives, while the constant $\alpha$ depends also on the bounds of the second-order derivatives. However, in the estimate \eqref{diff 2ed error upper bound}, the constant $\alpha$ only appears in a term significantly smaller than $h/N_s$. Thus, it is expected that the second-order derivatives have less effect on the numerical error. This is also the case for the inchworm Monte Carlo method to be analyzed in Section \ref{sec: proof}, and accordingly we will give less details for the estimate involving  second-order derivatives.

We will defer the proofs of these theorems. For now, let us discuss the implication of these theorems and how this is related to the numerical sign problem in the quantum Monte Carlo method.

 
\subsection{Discussion on the relation between the numerical error and the numerical sign problem} \label{nsp}

The above error estimation shows exponential growth of the numerical error with respect to time. Such exponential growth is due to the amplification of the error at previous time steps, as is well known in the numerical analysis for ordinary differential equations, which is often estimated by the discrete Gr\"onwall inequality. 

There exists another kind of exponential growth of error, which is typically encountered in the stochastic simulation of quantum mechanical systems, called the ``numerical sign problem'' \cite{Loh1990}. Such a problem occurs when using the Monte Carlo method to evaluate the integral or sum of a strongly oscillatory high-dimensional function. To understand how the numerical sign problem causes the exponential growth of the numerical error, we consider the case where $f(t,u) = -\ii H(t) u$ and $g(t,u,X) = -\ii A(t,X)u$. As an analog of quantum mechanics, we assume that $H(t)$ is a Hermitian matrix, so that $\|u(t)\|_2^2 = \|u(0)\|_2^2$ for any $t$. The solution of this system of ordinary differential equations can be expressed by the Dyson series:
\begin{equation} \label{dyson}
\begin{split}
u(T) = u(0) + \sum_{M=1}^{+\infty}
  \int_0^T \int_0^{t_M} \cdots \int_0^{t_2}
& (-\ii)^M \big( \E_{X_M} A(t_M,X_M) \big) 
    \big( \E_{X_{M-1}} A(t_{M-1},X_{M-1}) \big) \\
& \cdots \big( \E_{X_1} A(t_1,X_1) \big) u(0)
  \,\dd t_1 \cdots \,\dd t_{M-1} \,\dd t_M.
\end{split}
\end{equation}
Thus $u(T)$ can be evaluated directly using the Monte Carlo method to approximate the integral, where $M, t_1, \cdots, t_M$ and $X_1, \cdots, X_M$ are all treated as random variables. While different methods to draw samples exist \cite{Cai2018}, here we only consider the simplest approach, in which $M$ follows the Poisson distribution with the parameter $\lambda = M'T$ where $M'$ is the bound of the first-order derivatives defined in \eqref{assump: bd}, the time points $(t_1, t_2, \cdots, t_M)$ are uniformly distributed in the $M$-dimensional simplex. Let $u^{(\mathrm{num})}(T)$ be the numerical solution obtained using this method. Then the standard error estimation of the Monte Carlo method yields
\begin{equation} \label{eq:variance}
\begin{split}
& \E \|u^{(\mathrm{num})}(T) - u(T)\|^2 \\
\le{} &
  \sum_{M=0}^{+\infty} \frac{\ee^{M'T}}{(M')^M} \int_0^T \int_0^{t_M} 
  \cdots \int_0^{t_2}
    \E \| A(t_M,X_M) A(t_{M-1},X_{M-1}) 
    \cdots A(t_1,X_1) u(0) \|_2^2
  \,\dd t_1 \cdots  \,\dd t_M - \|u(T)\|_2^2 \\
\le{} & [\exp (  (d  +1) M' T ) - 1] \|u(0)\|_2^2,
\end{split}
\end{equation}
where $d$ is the number of dimensions of $u$. It can be observed that the numerical error again grows exponentially in time. However, such exponential growth of the error is not due to the error amplification in the Gr\"onwall inequality. It comes from the growing integral domain on the right-hand side of \eqref{dyson}. Note that although the integral domain expands as $T$ increases, the magnitude of the infinite sum does not grow over time ($\|u(T)\|_2 = \|u(0)\|_2$). This indicates that when $T$ is large, strong oscillation exists in the integrand of \eqref{dyson}, resulting in significant cancellation when taking the sum, which leads to huge variance for Monte Carlo estimation. This intrinsic difficulty of stochastic methods is known as the numerical sign problem. 

As we will see later, if not carefully dealt with, the numerical sign problem can cause even faster growth of the numerical error. One possible approach to mitigating the numerical sign problem is to use the method of \emph{partial resummation}, which only takes part of the summation (instead of the whole integral), and use the result to find other parts of the sum. For example, suppose we want to compute the infinite sum:
\begin{displaymath}
s = 1+a+a^2+a^3+\cdots,
\end{displaymath}
we can choose to take the sum directly using the Monte Carlo method. Alternatively, we can also first take the partial sum $s_1 = 1+a$, and then use the result of $s_1$ to compute another partial sum $s_3 = (1 + a^2) s_1$. Afterwards, $s_3$ can be used to compute $s_7 = (1 + a^4) s_3$, and so forth. It can be seen that the error in the computation of $s_1$ will be amplified when computing $s_3$, and the error of $s_3$ will be amplified in the computation of $s_7$. This illustrates the idea of partial resummation, which partly transfers the sign problem to error amplification. Below we would like to demonstrate that sometimes the Runge-Kutta method can also be considered as the partial resummation of \eqref{dyson}, which changes the underlying mechanism of the error growth.

Consider applying the forward Euler method to the equation
\begin{equation} \label{eq: toy model}
\frac{\dd u}{\dd t} = -\ii H(t)u, \qquad H(t) = \E_X A(t,X),
\end{equation}
so that the numerical scheme is 
\begin{displaymath}
\tu_{n+1} = \tu_n -\ii \frac{h}{N_s} \sum_{l=1}^{N_s} A(t_n, X_l^{(n)}) \tu_n.
\end{displaymath}
When $n = 0$, the scheme gives
\begin{equation} \label{eq: tu1}
\tu_1 = u(0) -\ii \underline{\frac{h}{N_s} \sum_{l=1}^{N_s} A(0, X_l^{(0)}) u(0)}.
\end{equation}
If we view the underlined term as a special Monte Carlo method to evaluate the integral
\begin{displaymath}
\int_0^h \E A(t_1, X) u(0) \,\dd t_1,
\end{displaymath}
for which we only take one sample of $t_1$ locating at $t_1=0$,
then $\tu_1$ turns out to be part of the right-hand side of \eqref{dyson} by reducing the integral domain from $T$ to $h$ and only considering $M=1$. Next, this partial sum $\tu_1$ is used to compute $\tu_2$:
\begin{displaymath}
\begin{split}
\tu_2 &= \tu_1 -\mathrm{i} \frac{h}{N_s} \sum_{l=1}^{N_s} A(t_1, X_l^{(1)}) \tu_1 \\
&= u(0) -\mathrm{i} \frac{2h}{N_s} \sum_{l=1}^{N_s} \frac{A(0, X_l^{(0)}) + A(h,X_l^{(1)})}{2} u(0) + (-\ii)^2 \frac{h^2}{N_s^2} \left( \sum_{l=1}^{N_s} A(h, X_l^{(1)}) \right) \left( \sum_{l=1}^{N_s} A(0, X_l^{(0)}) \right) u(0) \\
& \approx u(0) + \int_0^{2h} (-\ii) \E A(t_1, X) u(0) \,\dd t_1
+ \frac{1}{2} \int_0^{2h} \int_0^{t_2} (-\ii)^2 (\E A(t_2, X))(\E A(t_1, X)) u(0) \,\dd t_1 \,\dd t_2,
\end{split}
\end{displaymath}
which can again be considered as the partial sum of \eqref{dyson}. For further time steps, this can also be verified. As is well known, the error of the forward Euler method may accumulate as the solution evolves, and therefore this example again shows the change of the mechanism for the growth of the error.

However, our error estimate in Theorem \ref{thm: diff bounds} seems to
suggest that shifting the numerical sign problem to error amplification does not flatten the error curve, which still
grows exponentially. The reason is that our error estimation does not make
any assumption on the stability of the Runge-Kutta method, as leads to the exponential growth of the error regardless of the scheme and the problem. Again, let us take \eqref{eq: toy model} as an example and apply the second-order Heun's method. Then the deterministic scheme and the stochastic scheme are, respectively,
\begin{equation}\label{eq:two schemes}
u_{n+1} = (I - h \mathcal{L}) u_n \quad \text{and} \quad
\tilde{u}_{n+1} = (I - h \mathcal{A}) \tilde{u}_n,
\end{equation}
where
\begin{align*}
\mathcal{L} &= \frac{1}{2} \left[
  \ii \Big( H(t_n) + H(t_{n+1}) \Big) + h H(t_{n+1}) H(t_n)
\right], \\
\mathcal{A} &= \frac{1}{2} \left[
  \frac{\ii}{N_s} \sum_{l=1}^{N_s} \Big( A(t_n, X_l^{(1)}) + A(t_{n+1}, X_l^{(2)}) \Big)
  + h \left( \frac{1}{N_s} \sum_{l=1}^{N_s} A(t_{n+1}, X_l^{(2)})\right)
      \left( \frac{1}{N_s} \sum_{l=1}^{N_s} A(t_n, X_l^{(1)})\right)
\right].
\end{align*}
By straightforward calculation, we can find that
\begin{displaymath}
\E \|u_{n+1} - \tilde{u}_{n+1}\|_2^2 \le
  \|I-h\mathcal{L}\|_2^2 \, \E\|u_n - \tilde{u}_n\|_2^2
  + h^2 \E\|(\mathcal{L} - \mathcal{A}) \tilde{u}_n\|_2^2.
\end{displaymath}
On the right-hand side, the second term can be bounded by the standard Monte
Carlo error estimate. For the first term, if we assume that $H(t_{n+1})$ and
$H(t_n)$ are both Hermitian matrices, and $H(t_{n+1}) - H(t_n) = O(h)$, then
$\|I-h\mathcal{L}\|_2^2 = 1 + O(h^4)$. Therefore when $h$ is small, the
exponential growth of the error can be well suppressed, since
\begin{displaymath}
\lim_{h\rightarrow 0^+} (1 + Ch^4)^{T/h} = 1
\end{displaymath}
for any positive constants $C$ and $T$. However, for large time
steps such that $\|I - h\mathcal{L}\|_2$ is significantly larger than $1$, the error still grows exponentially. In general, for a stable Runge-Kutta scheme, the constant $\beta$ in the coefficient $(1+\beta h)$ appearing in \eqref{eq: diff recurrence 2} can be negative or a positive $o(1)$ quantity. In this case, partial resummation indeed helps reduce the error growth. One example of applications is the method of qDRIFT proposed in \cite{Campbell2019}, where the total Hamiltonian is also computed using a stochastic method. A symplectic time integrator is utilized therein so that the error growth is also well suppressed. Note that the paper \cite{Campbell2019} provides only an estimate of the bias for qDRIFT, while has not considered the full numerical error. As we have discussed, the bias is in fact not the major part of the error. 

The analysis of such a simple ODE sketches the idea how the numerical sign
problem can be mitigated in the algorithms with partial resummation. For open quantum systems, the inchworm Monte Carlo method, which has been claimed to have the capability of taming the numerical sign problem \cite{Cohen2015}, is one option to apply partial resummation to the corresponding Dyson series. However, due to the existence of the heat bath, the evolution of the quantum state is non-Markovian, and the equation that the inchworm Monte Carlo method solves can only be formulated as an integro-differential equation, so that one cannot simply apply a symplectic scheme to suppress the error growth. For this reason, the situation of the inchworm Monte Carlo method is much more complicated due to the nontrivial behavior of the error amplification. A detailed discussion can be found in the following.


\section{Introduction to inchworm Monte Carlo method}
\label{sec: inchworm intro}

We study an open quantum system described by the von Neumann equation for the density matrix $\rho(t)$   
\begin{equation} \label{eq:vonNeumann}
\ii \frac{\dd \rho}{\dd t} = [H, \rho]:= H\rho - \rho H,
\end{equation}
where the Schr\"odinger picture Hamiltonian $H$ is a Hermitian operator on the Hilbert space $\mc{H}_s \otimes \mc{H}_b$, with $\mc{H}_s$ and $\mc{H}_b$ representing respectively the Hilbert spaces associated with the system and the bath of the open quantum system. The operator $H$ consists of the Hamiltonians of the system and the bath, as well as the coupling terms describing the interaction of the system and the bath. Assuming that the coupling term has the tensor-product form, we have
\begin{displaymath}
 H =  H_s \otimes \mathrm{Id}_b + \mathrm{Id}_s \otimes H_b +  W_s \otimes W_b,
\end{displaymath}
where $H_s$ and $W_s$ are Hermitian operators on $\mc{H}_s$, $H_b$ and $W_b$ are Hermitian operators on $\mc{H}_b$, and $\mathrm{Id}_s,\mathrm{Id}_b$ are the identity operators for the system and the bath, respectively. In our paper, we take the common assumption that the bath is modeled by a larger number of harmonic oscillators, and we only consider the simplest system modeled by a single spin. Such a problem is often used as benchmarks, as it exhibits most difficulties in the treatment of the system-bath coupling, and is known as the spin-boson model to be introduced below.

\subsection{Spin-boson model}
As one fundamental example of open quantum systems \cite{Wang2000, Kernan2002, Duan2017}, the spin-boson model assumes that 
\begin{displaymath}
\mc{H}_s = \text{span}\{ \ket{1}, \ket{2} \}, \qquad
\mc{H}_b = \bigotimes_{l=1}^L \left( L^2(\mathbb{R}^3) \right),
\end{displaymath}
where $L$ is the number of harmonic oscillators in the bath. The corresponding Hamiltonians are
\begin{displaymath}
H_s = \epsilon \hat{\sigma}_z + \Delta \hat{\sigma}_x, \qquad
H_b = \sum_{l=1}^L \frac{1}{2} (\hat{p}_l^2 + \omega_l^2 \hat{q}_l^2)
\end{displaymath}
Here $\hat{\sigma}_x$, $\hat{\sigma}_z$ in $H_s$ are Pauli matrices satisfying
  $\hat{\sigma}_x \ket{1} = \ket{2}$, $\hat{\sigma}_x \ket{2} = \ket{1}$,
  $\hat{\sigma}_z \ket{1} = \ket{1}$, $\hat{\sigma}_z \ket{2} = -\ket{2}$, and the parameters $\epsilon$, $\Delta$ are respectively the energy difference between two spin states and frequency of the spin flipping. In the bath Hamiltonian $H_b$, the notations $\hat{p}_l$, $\hat{q}_l$ and $\omega_l$ are respectively the momentum operator, the position operator and the frequency of the $l$th harmonic oscillator. The coupling operators are given by  
\begin{displaymath}
 W_s = \hat{\sigma}_z ,\qquad W_b = \sum_{l=1}^L c_l \hat{q}_l,
\end{displaymath}
where $c_l$ is the coupling intensity between the $l$th harmonic oscillator and
the spin.

The density matrix solving \eqref{eq:vonNeumann} can be formally written as $\rho(t) =  \ee^{-\ii t H} \rho(0) \ee^{\ii t H}$, and we assume its initial value has the separable
form $\rho(0) = \rho_s \otimes \rho_b$ where the bath $\rho_b$ commutes with the Hamiltonian $H_b$. We are interested in the evolution of the expectation for a given Hermitian observable $O = O_s \otimes \mathrm{Id}_b$ acting only on the system part, defined by
\begin{equation} \label{eq:O(t)}
\langle O(t) \rangle := \tr(O \rho(t))
  = \tr(O \ee^{-\ii t H} \rho(0) \ee^{\ii t H})=  \tr(\rho_s \otimes \rho_b \ee^{\ii t H} O \ee^{-\ii t H} ) = \tr_s (\rho_s \Ge(2t,0))
\end{equation}
where we need to evaluate the \emph{full propagator} $\Ge(2t,0):=\tr_b(\rho_b \ee^{\ii t H} O \ee^{-\ii t H}) \in \C^{2\times 2}$. The subscript ``e" stands for ``exact" which is used to distinguish $\Ge$ from its numerical solutions later. The time parameters $0$ and $2t$ are defined on the ``Keldysh contour'' \cite{Keldysh1965}, where we consider $\ee^{\ii t H} O \ee^{-\ii t H}$ as the operator evolving the quantum state for time $t$ using the Hamiltonian $H$, applying the observable operator $O$, and then evolving the quantum state for another period of time $t$ using the Hamiltonian $-H$. Thus the total time of evolution is $2t$, and the parameters $0$ and $2t$ in $\Ge(2t,0)$ denote the initial and final times of such evolution, respectively. This notation can also be generalized to the evolution restricted on a section of the Keldysh contour from $\si$ to $\sf$ as $\Ge: \mathcal{T} \rightarrow \C^{2\times 2}$, where $\mathcal{T} = \{(\sf, \si) \mid 2t \geq \sf \geq \si \geq 0\}$. The precise definition of $\Ge(\sf,\si)$ is
\begin{displaymath}
\Ge(\sf,\si) = \begin{cases}
\tr_b(\rho_b \ee^{\ii \sf H_b} \ee^{-\ii(\sf-\si)H} \ee^{-\ii \si H_b}), & \text{if } \si \leq \sf < t, \\
\tr_b(\rho_b \ee^{\ii (2t-\sf) H_b} \ee^{-\ii(\si-\sf)H} \ee^{-\ii (2t-\si) H_b}), & \text{if } t \leq \si \leq \sf, \\
\tr_b(\rho_b \ee^{\ii (2t-\sf) H_b} \ee^{\ii(\sf-t)H}  O \ee^{-\ii(t-\si)H} \ee^{-\ii \si H_b}), & \text{if } \si < t \leq \sf.
\end{cases}
\end{displaymath}
It can be seen that in the case $\si < t \leq \sf$, if we replace $H_b$ with $H$ in the definition, the result becomes identical to the definition of $G_e(2t,0)$. Since $H_b$ is the bath part of $H$, the full propagator $\Ge(\sf,\si)$ can be considered as a partial summation of $\Ge(2t,0)$ (see \cite{Chen2017, Cai2020} for details). The function $\Ge(\sf, \si)$ is discontinuous at $\sf = t$ and $\si = t$ when the operator $O$ is applied. While we only consider the spin-boson model in this work, our analysis can be extended to higher-dimensional Hilbert space $\mc{H}_s$ without difficulties.

Due to the high dimensionality of the space $\mc{H}_b$, it is impractical to solve $\ee^{\pm \ii t H}$ directly. One feasible approach is to apply the method of quantum Monte Carlo to approximate $\langle O(t) \rangle$ numerically based on the Dyson series of propagator $\Ge(2t,0)$ which will be introduced subsequently.

\subsection{Dyson series}
 It is well known that $\Ge(\sf,\si)$ can be expanded into the following \emph{Dyson series} (for derivation, see \cite{Cai2020}):
\begin{equation}\label{eq:Ge_Dyson}
    \Ge(\sf, \si)  =  G_s^{(0)}(\sf,\si) + \sum_{M=1 }^{+\infty}
  \ii^M \int_{\sf > \vec{\sb} > \si} 
(-1)^{\#\{\vec{\sb} < t\}}  \mathcal{U}^{(0)}(\sf, \vec{\sb} , \si) \cdot
    \mathcal{L}^{(0)}(\vec{\sb}) \ \dd \vec{\sb}, \quad \text{~for~} (\sf,\si) \in \mathcal{T}.
 \end{equation}
Here $\vec{\sb} = (s_M, s_{M-1}, \cdots, s_1)$ is an $M$-dimensional vector denoting a decreasing time sequence, and $\#\{\vec{\sb} < t\}$ denotes the number of components in $\vec{\sb}$ that are less than $t$. The integral with respect to $\vec{\sb}$ is interpreted as
\begin{displaymath}
\int_{\sf > \vec{\sb} > \si} \psi(\vec{\sb}) \,\dd \vec{\sb}  : = \int^{\sf}_{\si}  \int^{s_M}_{\si}  \cdots \int^{s_{2}}_{\si}    \psi(\vec{\sb}) \, \dd s_1 \cdots  \dd s_{M-1} \, \dd s_M  .
\end{displaymath}
In the integrand of \eqref{eq:Ge_Dyson}, the functional $\mathcal{U}^{(0)}$ is associated with the system part, defined by
\begin{equation}\label{def U0}
 \mathcal{U}^{(0)}(\sf,\vec{\sb},\si)  = G_s^{(0)}(\sf, s_M) W_s G_s^{(0)}(s_{M}, s_{M-1}) W_s
  \cdots W_s G_s^{(0)}(s_2, s_1) W_s G_s^{(0)}(s_1, \si),
\end{equation}
where
\begin{equation}\label{Gs0}
  G_s^{(0)}(s_{k+1}, s_k) =
  \begin{cases}
    \ee^{-\ii (s_{k+1} - s_k) H_s},
    & \text{if } s_k \le s_{k+1} < t, \\
    \ee^{-\ii (s_k - s_{k+1}) H_s},
    & \text{if } t \le s_k \le s_{k+1}, \\
    \ee^{-\ii (t - s_{k+1}) H_s} O_s \ee^{-\ii (t - s_k) H_s},
    & \text{if } s_k < t \le s_{k+1}.
  \end{cases}
\end{equation}
Such a form is related to the operator $W_s$ in the interaction picture. The function $\mc{L}^{(0)}$ comes from the contribution of the bath, which is computed using the Wick's theorem \cite{Negele1988}. It has the form
\begin{equation} \label{eq:L all pair}
   \Ls^{(0)}(s_M,\cdots,s_1) =  \left\{   \begin{array}{l l}
   0, & \text{if $M$ is odd}; \\ 
   \displaystyle \sum_{\mf{q} \in \mQ(\vec{\sb})} \prod_{(s_k,s_j) \in \mf{q}} B(s_k,s_j), &  \text{if $M$ is even},
    \end{array} \right.
\end{equation}
where $B: \{(\tau_2,\tau_1) \mid  \tau_1 \leq \tau_2 \} \rightarrow \mathbb{C}$ is the \emph{two-point bath correlation} whose general definition is given in \cite{Chen2017}. Later in the numerical experiments, we will specify the formula of $B(\cdot,\cdot)$ for the spin-boson model. The set $\mQ(\vec{\sb})$ is given by:
\begin{equation} \label{eq:all linking pairs}
  \begin{split}
  \mQ(s_M,\cdots,s_1) =
  \Big\{ \{(s_{k_1}, s_{j_1}), \cdots, (s_{k_{M/2}}, s_{j_{M/2}})\} \,\Big\vert\,  \{j_1, \cdots, j_{M/2}, k_1, \cdots, k_{M/2}\} = \{1,\cdots,M\}, \\
  j_l < k_l \text{ for any } l = 1,\cdots,M/2
  \Big\},
  \end{split}
\end{equation} 
which includes all possible pairings of the set $\{s_1, s_2, \cdots, s_M\}$. For example, when $M = 4$, we have
\begin{displaymath}
  \mQ(s_4, s_3, s_2, s_1) = \Big\{ \{(s_4, s_3), (s_2, s_1)\}, \quad \{(s_4, s_2), (s_3, s_1)\}, \quad \{(s_4, s_1), (s_3, s_2)\},\Big\},
\end{displaymath}
and thus
\begin{equation} \label{eq:all pairs example}
   \Ls^{(0)}(s_4,s_3,s_2,s_1) = B(s_2,s_1) B(s_4,s_3) + B(s_3,s_1) B(s_4,s_2) + B(s_4,s_1) B(s_3,s_2).
\end{equation}
For convenience, this is often expressed by the following diagrammatic equation:
\begin{equation} \label{eq:all linking pair diagram example}
\Ls^{(0)}(s_4,s_3,s_2,s_1)
=
\begin{tikzpicture}
\draw[-] (0,0)--(1.5,0);\draw plot[only marks,mark =*, mark options={color=black, scale=0.5}]coordinates {(0,0) (0.5,0) (1,0)(1.5,0)};
\draw[-] (0,0) to[bend left=75] (0.5,0);
\draw[-] (1,0) to[bend left=75] (1.5,0);
 \end{tikzpicture}
 +
 \begin{tikzpicture}
\draw[-] (0,0)--(1.5,0);\draw plot[only marks,mark =*, mark options={color=black, scale=0.5}]coordinates {(0,0) (0.5,0) (1,0)(1.5,0)};
\draw[-] (0,0) to[bend left=75] (1,0);
\draw[-] (0.5,0) to[bend left=75] (1.5,0);
 \end{tikzpicture}
 +
  \begin{tikzpicture}
\draw[-] (0,0)--(1.5,0);\draw plot[only marks,mark =*, mark options={color=black, scale=0.5}]coordinates {(0,0) (0.5,0) (1,0)(1.5,0)};
\draw[-] (0,0) to[bend left=60] (1.5,0);
\draw[-] (0.5,0) to[bend left=75] (1,0);
 \end{tikzpicture}~,
\end{equation}
where each diagram refers to a product $B(\cdot,\cdot)B(\cdot,\cdot)$ and each arc connecting a pair of bullets denotes the corresponding two-point correlation. For general $M$, the value of the corresponding bath influence functional is the sum of all possible pairings, and the number of these diagrams is $(M-1)!!$.

To evaluate $\Ge(\sf,\si)$, one may truncate the Dyson series at a sufficiently large integer $\bar{M}$ and evaluate those high-dimensional integrals on the right-hand side using Monte Carlo integration, resulting in the bare dQMC. More specifically, we replace each integral in \eqref{eq:Ge_Dyson} by the average of $N_s$ samples of time sequences, and the Dyson series is approximated by 
\begin{equation}\label{bare dqmc}
  \Ge(\sf, \si)  \approx  G_s^{(0)}(\sf,\si) + \sum_{\substack{M=2 \\ M \text{~is even}  } }^{\bar{M}} \frac{1}{N_s} \  \sum^{N_s}_{i = 1} \   \frac{(\sf - \si)^M}{M!} \cdot  \ii^{M} 
(-1)^{\#\{\sb^{(i)}_M < t\}}  \mathcal{U}^{(0)}(\sf, \vec{\sb}^{(i)}_M , \si) 
    \mathcal{L}^{(0)}(\vec{\sb}^{(i)}_M)
\end{equation}
where each time sequence $\vec{\sb}^{(i)}_M := (s^{(i)}_1,\cdots,s^{(i)}_{M})$ is drawn independently from the uniform distribution $U([\si,\sf]^M)$ and then sorted such that $s^{(i)}_1<s^{(i)}_2< \cdots < s^{(i)}_{M}$. Note that the bath influence functional vanishes when $M$ is odd, which is why the right-hand side of \eqref{bare dqmc} only sums over terms with even $M$. For the same reason, the truncation $\bar{M}$ should also be chosen as an even integer.

The numerical solution obtained via bare dQMC will also encounter the dynamical sign problem: for any given even integer $M$, the function $\Ls^{(0)}$ is the sum of $(M-1)!!$ terms, and each is bounded by $\bdL^{M/2}$ upon assuming $|B(\cdot,\cdot)|$ has the uniform upper bound $\bdL$. Similar to our analysis for \eqref{eq:variance}, 
the numerical sign problem of \eqref{eq:Ge_Dyson} can  be quantified by the following bound of variance \cite[Section 5]{Cai2020}:
\begin{equation}\label{dyson sign problem}
\sum_{\substack{M=2\\M\text{ is even}}}^{+\infty} \frac{(\sf-\si)^M}{M!} (M-1)!! \left(\|W_s\|^M \bdL^{M/2}\right)^2 = \exp \left( \frac{\|W_s\|^4 \bdL^2 (\sf-\si)^2}{2} \right) - 1.
\end{equation}
One can see from the bound above that the variance increases exponentially with the square of the length of the time interval, making long time simulations extremely difficult.
To mitigate this fast error growth, the inchworm Monte Carlo method was introduced in \cite{Chen2017, Cai2020}, which will be introduced in the subsequent section.

\subsection{Inchworm Monte Carlo method}
 \subsubsection{Integro-differential equation}
 \label{sec: int diff eq intro}
In \cite[Section 4]{Cai2020}, the full propagator has been proved to satisfy the following integro-differential equation:
\begin{equation}\label{eq: inchworm equation 1}
    \frac{\partial \Ge(\sf,\si)}{\partial \sf} =   \sgn(\sf-t)  \left[ \ii H_s \Ge(\sf,\si) +  \sum^{+ \infty}_{\substack{M=1\\ M \text{~is odd~}}} \ii^{M+1} \int_{\sf > \vec{\sb} > \si} \dd \vec{\sb}   (-1)^{\#\{\vec{\sb} < t\}}  W_s \mc{U}(\sf,\vec{\sb},\si) \Ls(\sf,\vec{\sb}) \right] .
\end{equation} 
The equation holds for the initial time point $\si\in[0,2t]\backslash \{t\}$ and the final time point $\sa \in [\si,2t]\backslash \{t\}$, and satisfies the following ``jump conditions'' due to the discontinuity of $\Ge(\sf,\si)$:
\begin{equation}  \label{eq: jump condition}
    \begin{split}
&\lim_{\sa \rightarrow t^+} \Ge(\sa,\si) = O_s \lim_{\sa \rightarrow t^-}\Ge(\sa,\si),  \\
&\lim_{\si \rightarrow t^-} \Ge(\sa,\si) =  \lim_{\si \rightarrow t^+}\Ge(\sa,\si)O_s.
  \end{split}
\end{equation}
In addition, \eqref{eq: inchworm equation 1} also satisfies the boundary condition $\Ge(s',s') = \text{Id}$ for all $0 \le s' \le 2t$. 

In the integrand of \eqref{eq: inchworm equation 1}, $\mc{U}$ is defined similarly to $\mc{U}^{(0)}$ in the Dyson series with the bare propagator $G^{(0)}_s(\cdot, \cdot)$ replaced by the full propagator $\Ge(\cdot,\cdot)$:
\begin{equation}\label{U}
\mc{U}(\sf, \vec{\sb}, \si) = \Ge(\sf, s_M) W_s
  \Ge(s_{M}, s_{M-1}) W_s \cdots W_s \Ge(s_2, s_1) W_s \Ge(s_1, \si).
\end{equation}
The definition of the bath part $\Ls$ is also similar to $\Ls^{(0)}$:
\begin{equation} \label{eq:L inchworm}
    \Ls(\sf, s_M, \cdots,s_{1})   =  \sum_{\mf{q} \in \mQ^c(\sf,\vec{\sb})} \prod_{(s_k,s_j) \in \mf{q}} B(s_k,s_j),
\end{equation}
where $\mQ^c$ is a subset of $\mQ$ appearing in $\Ls^{(0)}$ which only includes ``linked" pairings, which means in its diagrammatic representation all points are connected with each other using arcs as ``bridges". For example when $M=3$, $\Ls(\sf,s_3,s_2,s_1)$ only contains one linked diagram in \eqref{eq:all linking pair diagram example}:
\begin{equation} \label{Lbc m3}
\Ls(\sf,s_3,s_2,s_1)
= \ 
 \begin{tikzpicture}
\draw[-] (0,0)--(1.5,0);\draw plot[only marks,mark =*, mark options={color=black, scale=0.5}]coordinates {(0,0) (0.5,0) (1,0)(1.5,0)};
\draw[-] (0,0) to[bend left=75] (1,0);
\draw[-] (0.5,0) to[bend left=75] (1.5,0);
 \end{tikzpicture} = B(s_3,s_1) B(\sf,s_2).
\end{equation}
Another example for $M=5$ is given by 
\begin{equation}\label{linked pairs example}
  \begin{split}
    \Ls(\sf,s_5,s_4,s_3,s_2,s_1)
  ={} & 
\begin{tikzpicture}
\draw[-] (0,0)--(2.5,0);\draw plot[only marks,mark =*, mark options={color=black, scale=0.5}]coordinates {(0,0) (0.5,0) (1,0)(1.5,0)(2,0)(2.5,0)};
\draw[-] (0,0) to[bend left=60] (1,0);
\draw[-] (0.5,0) to[bend left=60] (2,0);
\draw[-] (1.5,0) to[bend left=60] (2.5,0);
 \end{tikzpicture} 
 +
\begin{tikzpicture}
\draw[-] (0,0)--(2.5,0);\draw plot[only marks,mark =*, mark options={color=black, scale=0.5}]coordinates {(0,0) (0.5,0) (1,0)(1.5,0)(2,0)(2.5,0)};
\draw[-] (0,0) to[bend left=60] (1.5,0);
\draw[-] (0.5,0) to[bend left=60] (2,0);
\draw[-] (1,0) to[bend left=60] (2.5,0);
 \end{tikzpicture}  
 +
 \begin{tikzpicture}
\draw[-] (0,0)--(2.5,0);\draw plot[only marks,mark =*, mark options={color=black, scale=0.5}]coordinates {(0,0) (0.5,0) (1,0)(1.5,0)(2,0)(2.5,0)};
\draw[-] (0,0) to[bend left=60] (1.5,0);
\draw[-] (0.5,0) to[bend left=60] (2.5,0);
\draw[-] (1,0) to[bend left=60] (2,0);
 \end{tikzpicture} 
 +
  \begin{tikzpicture}
\draw[-] (0,0)--(2.5,0);\draw plot[only marks,mark =*, mark options={color=black, scale=0.5}]coordinates {(0,0) (0.5,0) (1,0)(1.5,0)(2,0)(2.5,0)};
\draw[-] (0,0) to[bend left=60] (2,0);
\draw[-] (0.5,0) to[bend left=60] (1.5,0);
\draw[-] (1,0) to[bend left=60] (2.5,0);
 \end{tikzpicture} \\
 ={} & B(s_3,s_1)B(s_5,s_2)B(\sf,s_4)+ B(s_4,s_1)B(s_5,s_2)B(\sf,s_3)\\
 & \hspace{60pt}  +B(s_4,s_1)B(\sf,s_2)B(s_5,s_3)+B(s_5,s_1)B(s_4,s_2)B(\sf,s_3) 
    \end{split}
\end{equation}
which does not include the unlinked terms in the bath influence functional $\Ls^{(0)}(\sf,s_5,s_4,s_3,s_2,s_1)$ such as 
\begin{equation} \label{unlinked diagrams} 
  \begin{split}
&\begin{tikzpicture}
\draw[-] (0,0)--(2.5,0);\draw plot[only marks,mark =*, mark options={color=black, scale=0.5}]coordinates {(0,0) (0.5,0) (1,0)(1.5,0)(2,0)(2.5,0)};
\draw[-,red] (0,0) to[bend left=75] (0.5,0);
\draw[-] (1,0) to[bend left=75] (2,0);
\draw[-] (1.5,0) to[bend left=75] (2.5,0);
 \end{tikzpicture}:= \redtext{B(s_2,s_1)}B(s_5,s_3)B(\sf,s_4), \quad
 \begin{tikzpicture}
\draw[-] (0,0)--(2.5,0);\draw plot[only marks,mark =*, mark options={color=black, scale=0.5}]coordinates {(0,0) (0.5,0) (1,0)(1.5,0)(2,0)(2.5,0)};
\draw[-] (0,0) to[bend left=75] (1,0);
\draw[-] (0.5,0) to[bend left=60] (2.5,0);
\draw[-,red] (1.5,0) to[bend left=75] (2,0);
 \end{tikzpicture} :=B(s_3,s_1)B(\sf,s_2)\redtext{B(s_5,s_4)} , \quad \cdots
  \end{split}
\end{equation}
where the pairs marked in red do not link to the rest part of the diagrams via the arc bridges. Despite the smaller number of the diagrams included in $\Ls(s_m,\cdots,s_1)$ compared to $\Ls^{(0)}(s_m,\cdots,s_1)$ in Dyson series, this number also grows asymptotically as the double factorial $\ee^{-1}(m-1)!!$ \cite{Stein1978b}. Therefore, the modulus of $\Ls(\sf,s_M,\cdots,s_1)$ for given odd $M$ is similarly assumed to be bounded by $M!! \bdL^{\frac{M+1}{2}}$. 

The relation between the equation \eqref{eq: inchworm equation 1} and the original Dyson series \eqref{eq:Ge_Dyson} is similar to the relation between \eqref{eq: diff eq} and \eqref{dyson}. However, due to the non-Markovian nature of the propagators, partial resummation cannot reduce the series \eqref{eq:Ge_Dyson} to a differential equation, and therefore the equation holds an integro-differential form. Nevertheless, by partial resummation, the series in integro-differential equation \eqref{eq: inchworm equation 1} converges much faster than Dyson series \eqref{eq:Ge_Dyson}, allowing us to truncate it to get a reasonable approximation. In this paper, we assume that the series in \eqref{eq: inchworm equation 1} is truncated up to a finite $\bar{M}$, which can provide sufficiently good approximation to the original equation. For convenience, we write the integro-differential equation with truncated series as
\begin{equation}\label{eq: inchworm equation}
 \frac{\partial \Ge(\sa,\si)}{\partial \sa} =   \sgn(\sa - t) \ii H_s \Ge(\sa,\si) + \Hs(\sa,\Ge,\si),
 \end{equation}
where
\begin{equation}\label{eq: calH}
    \Hs(\sa,\Ge,\si) := \sgn(\sa - t)\sum^{\bar{M}}_{\substack{M=1\\ M \text{~is odd~}}} \ii^{M+1} \int_{\sa > \vec{\sb} > \si } (-1)^{\#\{\vec{\sb} < t\}}W_s \Us(\sa,\vec{\sb},\si) \Ls(\sa,\vec{\sb}) \,\dd\vec{\sb}.
\end{equation}

\begin{remark}
     The original inchworm algorithm proposed in \cite[Section 3]{Chen2017} was not introduced using the integro-differential equation. Instead, it writes the full propagator $G(\sf,\si)$ in a form similar to the Dyson series:
\begin{equation} \label{eq:inchworm}
\begin{split}
G(\sf, \si) &= \mc{G}_{\Sarr}(\sf,\si) +
\sum_{\substack{M=2\\[2pt] M \text{ is even}}}^{+\infty}
  \int_{\sf > \vec{\sb} > \si} (-1)^{\#\{\vec{\sb} < t\}} \ii^M \times\\
  & \times     \mc{G}_{\Sarr}(\sf, s_M) W_s \mc{G}_{\Sarr}(s_M, s_{M-1}) W_s
  \cdots W_s \mc{G}_{\Sarr}(s_2, s_1) W_s \mc{G}_{\Sarr}(s_1, \si) \mathcal{L}_{\mathrm{ip}}(\vec{\sb}) 
  \,\dd \vec{\sb},
\end{split}
\end{equation}
where $\Sarr \in (\si, \sf)$ and
\begin{displaymath}
\mc{G}_{\Sarr}(\sf, \si) = \left\{ \begin{array}{ll}
  G(\sf,\si), & \text{if } \si \le \sf \le \Sarr, \\
  G_s^{(0)}(\sf,\si), & \text{if } \Sarr < \si \le \sf, \\
  G_s^{(0)}(\sf,\Sarr) \, G(\Sarr,\si),
    & \text{if } \si \le \Sarr < \sf.
\end{array} \right.
\end{displaymath}
The bath part $\mc{L}_{\mathrm{ip}}(\vec{\sb})$ in \eqref{eq:inchworm} is the sum of all the ``inchworm proper'' diagrams with nodes $\vec{\sb}$. The set of inchworm proper diagrams includes all linked diagrams, and we refer to the readers to \cite{Chen2017} for its precise definition. The integro-differential equation \eqref{eq: inchworm equation 1} can be derived from \eqref{eq:inchworm} by setting $\sf - \Sarr$ to be infinitesimal.
\end{remark}

 \subsubsection{Numerical method}
Similar to the case of the differential equation, we may use general explicit time integrator to solve \eqref{eq: inchworm equation} numerically. In this work, we focus on the numerical method proposed in \citep{Cai2020}, which is inspired by the second-order Heun's method:
\begin{equation}\label{def: scheme 1}
 \begin{split}
&G^*_{n+1,m} = (I+\sgn(t_n - t) \ii H_s h)G_{n,m} +  K_1 h,  \\
&G_{n+1,m} = (I + \frac{1}{2}\sgn(t_n - t)\ii H_s h  )G_{n,m} +\frac{1}{2}\sgn(t_{n+1} - t)\ii  H_s h  G^*_{n+1,m} +  \frac{1}{2}(K_1+K_2) h, \quad 0 \le m\le n\le 2N,
   \end{split}
\end{equation}
where $h=t/N$ (we again require $h \le 1$) is the time step length, and $G_{n,m}$ denotes the numerical approximation of the solution $\Ge(nh, mh)$. Different from the standard Heun's method for ODEs, the slope $K_1$ has to be computed based on a number of previous numerical solutions
\begin{equation}\label{eq: gbnm}
\gb_{n,m}:=(G_{m+1,m};G_{m+2,m+1},G_{m+2,m};\cdots;G_{n,n-1},\cdots,G_{n,m}).
\end{equation}
The explicit expression for $K_1$ is given by 
\begin{multline}\label{def: scheme 1 k1}
  K_1 =  F_1(  \gb_{n,m}) := \sgn(t_n - t)
 \sum^{\bar{M}}_{\substack{M=1 \\M \text{~is odd~}}}
\ii^{M+1}  \int_{t_n > \vec{\sb} > t_m } (-1)^{\#\{\vec{\sb} < t\}}W_s  I_h G(t_n,s_M) W_s  \cdots W_s \times  \\ \times I_h G(s_1,t_m) \Ls(t_n,\vec{\sb}) \dd \vec{\sb}, 
\end{multline}
where $I_h G(\cdot,\cdot)$ is obtained by piecewise linear interpolation on the triangular mesh shown in Figure \ref{fig:mesh and order} such that $I_h G(t_j,t_k) = G_{j,k}$ for all integers $m\le k\le j\le n$. Similarly, $K_2$ is given by 
\begin{multline}\label{def: scheme 1 k2}
K_2 = F_2(\gb^*_{n,m}) := \sgn(t_{n+1} - t)  \sum^{\bar{M}}_{\substack{M=1 \\ M \text{~is odd~}}} \ii^{M+1}  \int_{t_{n+1} > \vec{\sb} > t_m } (-1)^{\#\{\vec{\sb} < t\}}W_s I^*_h G(t_{n+1},s_M) W_s  \cdots W_s \times \\ \times I^*_h G(s_1,t_m) \Ls(t_{n+1},\vec{\sb}) \dd\vec{\sb} 
\end{multline}
where $\gb^*_{n,m}:=(\gb_{n,m};G_{n+1,n},\cdots,G_{n+1,m+1},G^*_{n+1,m})$ and $I^*_h G(\cdot,\cdot)$ is the linear interpolation such that
\begin{equation*}
I^*_h G(t_j,t_k) = 
\begin{cases}
  G_{j,k}, & \text{if } (j,k) \neq (n+1,m), \\
  G^*_{n+1,m}, & \text{if } (j,k) = (n+1,m).
\end{cases}
\end{equation*}
To implement this scheme, we compute each $G_{j,k}$ in the order illustrated in Figure \ref{fig:mesh and order}. Specifically, we calculate the propagators column by column from left to right, and for each column we start from the boundary value $G_{i,i} = \text{Id}$ (red ``\textcolor{red}{$\bullet$}'') locating on the diagonal and compute from top to bottom.

Due to the jump conditions \eqref{eq: jump condition}, we need a special treatment for the two discontinuities at $G_{N,k}$ (green ``\textcolor{green}{$\bullet$}'') and $G_{j,N}$ (blue ``\textcolor{blue}{$\bullet$}'') to achieve second-order convergence. In the numerical scheme, we keep two copies of  $G_{N,k}$ or $G_{j,N}$ representing the left- and right-limits when $s \to t^{\pm}$:
\begin{equation*}
(G_{N^+,k},G_{N^-,k}) \text{~and~} 
(G_{j,N^+},G_{j,N^-}) \text{~for~} 0\le k \le N-1,N+1 \le j \le 2N.
\end{equation*}
Here $G_{N^{\pm},k}$ and $G_{j,N^{\pm}}$ are, respectively, the approximation of $\lim\limits_{s \rightarrow t^{\pm}} G(s,
kh)$ and $\lim\limits_{s \rightarrow t^{\pm}} G(j h, s)$. The relation $G_{N^+,k} = O_s G_{N^-,k}$ and $G_{j,N^-} = G_{j,N^+} O_s$ are immediately derived from the jump conditions. Moreover, we note that the boundary value on the discontinuities are given by: $G_{N^+,N^+}=G_{N^-,N^-} =\text{Id}$ and $G_{N^+,N^-} = O_s$.

In the implementation, we need to follow the rules below while evolving the scheme \eqref{def: scheme 1} near the discontinuities:
\begin{itemize}
\item[(R1)] When $n = N-1$, the quantities $G_{n+1,m}^*$ and $G_{n+1,m}$ are regarded as $G_{N^-,m}^*$ and $G_{N^-,m}$, respectively, and $\sgn(t_{n+1}-t)$ takes the value $-1$. When $n = N$, the propagator $G_{n,m}$ is regarded as $G_{N^+,m}^*$, and $t_n-t$ takes the value $1$.

\item[(R2)] The value of $G_{n+1,N^+}$ is set to be $O_s G_{N^-,m}$; the value of $G_{n+1,N^-}$ is set to be $G_{n+1,N^+} O_s$. 

\item[(R3)] $\sgn(t_{N^-} - t) = -1$, $\sgn(t_{N^+} - t) = 1$.

\item[(R4)] The interpolation of $I_h G$ and $I_h^* G$ should respect such
  discontinuities. For example, the interpolating operator $I_h$ should satisfy
  \begin{gather*}
  \lim_{s\rightarrow t^{\pm}} I_h G(t_j, s) = G_{j,N^{\pm}}, \qquad
  \lim_{s\rightarrow t^{\pm}} I_h G(s, t_k) = G_{N^{\pm},k}, \\
  \lim_{s\rightarrow t^+}
    \lim_{\tilde{s} \rightarrow t^+} I_h G(\tilde{s}, s) = 
  \lim_{s\rightarrow t^-}
    \lim_{\tilde{s} \rightarrow t^-} I_h G(s, \tilde{s}) = \text{Id}, \qquad
  \lim_{s\rightarrow t^+}
    \lim_{\tilde{s} \rightarrow t^-} I_h G(s, \tilde{s}) = O_s.
  \end{gather*}
  The conditions for $I_h^* G$ are similar.
\end{itemize}

\input{images/fig_mesh_and_order.tex}

As we will prove later, the above numerical method guarantees a second-order approximation of the solution. However, the computation cost is not affordable when $M$ is large since the degrees of freedom for calculating the integral with respect to $\vec{\sb}$ will grow exponentially w.r.t $M$. Therefore, we take advantage of Monte Carlo integration and replace the integrals by the averages of Monte Carlo samples, resulting in the following \emph{inchworm Monte Carlo method}: 
\begin{equation}\label{def: scheme 2}
 \begin{split}
&\tG^*_{n+1,m} = (I+\sgn(t_n - t) \ii H_s h)\tG_{n,m} +  \tK_1 h,  \\
&\tG_{n+1,m} = (I + \frac{1}{2}\sgn(t_n - t)\ii H_s h  )\tG_{n,m} +\frac{1}{2}\sgn(t_{n+1} - t)\ii  H_s h  \tG^*_{n+1,m} +  \frac{1}{2}(\tK_1+\tK_2) h, \quad 0 \le m\le n\le 2N,
   \end{split}
\end{equation}
with
\begin{align*}
 \tK_1 & = \frac{1}{N_s}\sum_{i=1}^{N_s} \widetilde{F}_1( \tg_{n,m};\vec{\sb}^{(i)})\\ & 
 := \frac{\sgn(t_n - t)}{N_s} \sum_{i=1}^{N_s}\sum^{\bar{M}}_{\substack{M=1 \\ M \text{~is odd~}}} 
\ii^{M+1}  \frac{(t_n-t_m)^M}{M!} (-1)^{\#\{\vec{\sb}^{(i)}_M < t\}}W_s I_h\tG(t_n,s^{(i)}_M) W_s  \cdots W_s \times \\ 
  & \hspace{150pt} \times I_h\tG(s^{(i)}_1,t_m) \Ls(t_n,\vec{\sb}^{(i)}_M),
\end{align*}
where $N_s$ denotes the number of samples and each time sequence $\vec{\sb}^{(i)}_M=(s_1^{(i)},s_2^{(i)},\cdots,s_M^{(i)})$ is drawn from the uniform distribution $U([t_m,t_n]^M)$ and then sorted such that $s_1^{(i)} < s_2^{(i)} < \cdots < s_M^{(i)}$. Similarly, we define
\begin{align*}
\tK_2 & =   \frac{1}{N_s}\sum_{i=1}^{N_s}  \widetilde{F}_2(\tg^*_{n,m};{\vec{\sb}}^{*(i)}) \\
& := \frac{\sgn(t_{n+1} - t)}{N_s} \sum_{i=1}^{N_s} \sum^{\bar{M}}_{\substack{M=1 \\ M \text{~is odd~}}} \ii^{M+1}  \frac{(t_{n+1}-t_m)^M}{M!} (-1)^{\#\{\vec{\sb}^{*(i)}_M < t\}}W_s I^*_h \tG(t_{n+1},s_M^{*(i)}) W_s  \cdots W_s \times \\
& \hspace{150pt} \times I^*_h \tG(s_1^{*(i)},t_m) \Ls(t_{n+1},\vec{\sb}^{*(i)}_M)
\end{align*}
with the time sequences $\vec{\sb}^{*(i)}_M = (s_1^{*(i)},s_2^{*(i)},\cdots,s_M^{*(i)})$ drawn from $U([t_m, t_{n+1}]^M)$ and then sorted as $s_1^{*(i)} < s_2^{*(i)} < \cdots < s_M^{*(i)}$.

Our goal in this paper is to study the error growth of the inchworm Monte Carlo method \eqref{def: scheme 2} by comparing it with the numerical sign problem for the classic Dyson series as discussed in \eqref{dyson sign problem}. While it is claimed in \cite{Chen2017,Chen2017b} that the inchworm Monte Carlo method can effectively mitigate such sign problem, a detailed argument on this mitigation has not been provided. Therefore, we aim at a rigorous numerical analysis for the scheme, which will begin from the next section.

\section{Numerical analysis for inchworm Monte Carlo method}
\label{sec: int diff eq}


\subsection{Notation and Assumptions}
To facilitate our analysis, we first introduce some notations and assumptions in this section.
 
\subsubsection{Vectorization and norms}
For a sequence of matrices defined as $\yb:=(Y_{1},Y_{2},\cdots,Y_{\ell}) \in \C^{2\times 2\ell}$, we define its vectorization $\vec{\yb}$ by
\begin{equation}
\vec{\yb}=(Y^{(11)}_{1},Y^{(21)}_1,Y^{(12)}_1,Y^{(22)}_1,\cdots,Y^{(11)}_{\ell},Y^{(21)}_\ell,Y^{(12)}_\ell,Y^{(22)}_\ell)^{\TT} \in \C^{4\ell},
\end{equation}
which reshapes the matrix into a column vector. The same notation applies to a single matrix. For example, for $Y = (Y^{(ij)})_{2\times 2}$, we have $\vec{Y} = (Y^{(11)}, Y^{(21)}, Y^{(12)}, Y^{(22)})^{\TT}$. For any function $F(\boldsymbol{y})$, its gradient $\nabla F(\boldsymbol{y})$ is defined as a $4\ell$-dimensional vector:
\begin{equation*}
\nabla F(\yb) = \left(
\frac{\partial F}{\partial Y_1^{(11)}},
\frac{\partial F}{\partial Y_1^{(21)}},
\frac{\partial F}{\partial Y_1^{(12)}},
\frac{\partial F}{\partial Y_1^{(22)}}, \cdots,
\frac{\partial F}{\partial Y_{\ell}^{(11)}},
\frac{\partial F}{\partial Y_{\ell}^{(21)}},
\frac{\partial F}{\partial Y_{\ell}^{(12)}},
\frac{\partial F}{\partial Y_{\ell}^{(22)}}
\right)^{\TT},
\end{equation*}
so that the mean value theorem is denoted by
\begin{equation*}
F(\yb_2) - F(\yb_1) = \nabla F\Big( (1-\xi) \yb_1 + \xi \yb_2 \Big)^{\TT} (\vec{\yb}_2 - \vec{\yb}_1), \qquad \text{for some } \xi \in [0,1].
\end{equation*}
The Hessian matrix $\nabla^2 F(\yb)$ is similarly defined as a $4\ell \times 4\ell$ matrix.

Let $\Is$ be an index set and $\gb = (G_{\alpha})_{\alpha \in \Is}$ be a collection of random matrices with each $G_{\alpha} \in \C^{2\times 2}$. For any $\Ds \subset \Is$, we define 
\begin{equation}
\|\gb\|_{\Ds} := \max_{\alpha \in\Ds}\{\|  G_{\alpha}  \|_\FF\},
\end{equation}
where $\|\cdot\|_\FF$ denotes the Frobenius norm. When $\Ds = \Is$, the subscript $\Ds$ will be omitted: $\|\gb\| = \|\gb\|_{\Is}$. It is clear that $\|\cdot\|$ defines a norm, and $\|\cdot\|_{\Ds}$ is a seminorm if $\Ds \subset \Is$. In particular, for any $2\times 2$ matrix $G$, we define $\|G\| = \|G\|_{\FF}$.  In our analysis, the index $\alpha$ is always a 2D multi-index. For example, if $n > m$ and
\begin{equation*}
\Is = \{ (m+1,m) \} \cup \{ (m+2,m+1), (m+2,m) \} \cup \cdots \cup \{ (n,n-1), \cdots, (n,m)\},
\end{equation*}
then $\gb$ equals $\gb_{n,m}$ defined in \eqref{eq: gbnm}. Similarly, we define
\begin{equation}
\Ns^{(\std)}_{\Ds}(\gb) = \max_{\alpha \in \Ds} \left\{ \left[ \E( \|G_\alpha \|_{\FF}^2) \right]^{1/2} \right\}
\end{equation}
which will be often used throughout our analysis.


\subsubsection{Boundedness assumptions}
We will need the following assumptions for our analysis:
\begin{enumerate}

\item[(H1)] \label{assump: H1}
The exact solution of \eqref{eq: inchworm equation} $\Ge$, numerical solution $G$ solved by the deterministic method \eqref{def: scheme 1} and numerical solution $\tG$ solved by inchworm Monte Carlo method \eqref{def: scheme 2} are all bounded by a $O(1)$ constant:
\begin{equation}
\begin{gathered}
\|\Ge(\sa,\si)\| \le \ \bdG \text{~for any~} 0 \le \si \le \sa \le 2t; \\
\|\tG_{j,k}\|,\|G_{j,k}
\| \le \bdG\text{~for any~}  j,k = 0,1,\cdots,N-1,N^-,N^+,N+1,\cdots 2N-1,2N.
\end{gathered}
\end{equation}

\item[(H2)] \label{assump: H2}
Each $rs-$entry ($r,s=1,2$) of the full propagator $\Ge^{(rs)}(\sa,\si)$ is of class $C^3$ on the domain $\si\in[0,2t]\backslash \{t\},\ \sa \in [\si,2t]\backslash \{t\}$ and we define the following upper bounds:
\begin{align*}
\left| \frac{\partial^\alpha \Ge^{(rs)}}{\partial \sa^{\alpha_1} \partial \si^{\alpha_2}}(x_1,x_2) \right|  \le 
\begin{cases}
  \bdG'', &\text{for~} \alpha = \alpha_1 + \alpha_2 = 2, \\
  \bdG''', &\text{for~}\alpha = \alpha_1 + \alpha_2 = 3, 
\end{cases}
\end{align*}
for any $x_2 \in[0,2t]\backslash \{t\},\ x_1 \in [x_2,2t]\backslash \{t\}$.

\item[(H3)] \label{assump: H3}
We further assume that system Hamiltonian $H_s$ and system perturbation $W_s$ can also be bounded by $O(1)$ constants:
\begin{equation}
 \|H_s\| \le \bdH, \ \|W_s\| \le \bdW. 
\end{equation}
As mentioned in Section \ref{sec: int diff eq intro}, the two-point correlation is assumed to be bounded by $|B(\cdot,\cdot)| \le \bdL$ for some $O(1)$ constant $\bdL$. According to the definition \eqref{eq:L inchworm} for $\Ls$, we bound  
 \begin{equation}\label{eq: L bound}
|\Ls(s_m,s_{m-1},\cdots,s_1)| \le  (m-1)!! \bdL^{\frac{m}{2}}.
\end{equation}

\end{enumerate}

\begin{remark}
Recall that \eqref{eq: L bound} is used as the upper bound for the bath influence functional $|\Ls^{(0)}(s_m,\cdots,s_1)|$ in Dyson series when analyzing the sign problem \eqref{dyson sign problem}. Since $\Ls(s_m,\cdots,s_1)$ includes fewer diagrams, $|\Ls(s_m,\cdots,s_1)|$ can actually be controlled by a lower bound as $\alpha(m) (m-1)!! \bdL^{\frac{m}{2}}$ with a coefficient $\alpha(m) \in (0,1)$. In fact, the factor $\alpha(m)$ is the reason why the series \eqref{eq: calH} has faster convergence than the Dyson series \eqref{eq:Ge_Dyson}. Here we are mainly interested in the case of a finite $\bar{M}$ (the upper bound of $M$), and therefore the looser bound \eqref{eq: L bound} does not change the order of the error and its general growth rate.
\end{remark}

\subsection{Main results and discussions on the error growth} \label{sec: inchworm results}

In this section, we will provide our main results for the error estimation of the inchworm Monte Carlo method, and compare the results with the error growth of Dyson series expansion. The following theorem gives the difference between the inchworm Monte Carlo method \eqref{def: scheme 2} and the deterministic scheme \eqref{def: scheme 1}:
\begin{theorem} \label{thm: bounds}
Let $\dG_{n,m}= \tG_{n,m} - G_{n,m}$. Given a sufficiently small time step length $h$ and a sufficiently large $N_s$, if the assumptions (H1) and (H3) hold, the difference between the deterministic solution and the Monte Carlo solution can be estimated by
\begin{itemize}
\item Bias estimation
\begin{equation}\label{1st error upper bound}
\|\E(\Delta G_{n+1,m})\| \le 4\theta_2^2 \bar{\alpha}(t_{n-m+1})\bar{\gamma}(t_{n-m+1})\left(\ee^{3 \theta_1 \sqrt{P_1(t_{n-m+1})} t_{n-m+1}}\right)\cdot \frac{h}{N_s}
\end{equation}
\item Numerical error estimation
\begin{equation}\label{2ed error upper bound}
    [\E(\|\dG_{n+1,m}\|^2)]^{1/2} \le  \theta_2 \sqrt{\bar{\gamma}(t_{n-m+1})}\left(\ee^{ \theta_1 \sqrt{P_1(t_{n-m+1})} t_{n-m+1}}\right) \cdot \sqrt{\frac{h}{N_s}} 
\end{equation}
\end{itemize}
Here
\begin{align} \label{eq: alpha gamma}
&\bar{\alpha}(t)= 16P_2(t)\cdot (10t+16t^2 + 5t^3 + \frac{1}{4}t^4), \qquad
\bar{\gamma}(t) =  2\bdW \bdG \bdL^{1/2}   \sum^{\bar{M}}_{\substack{M=1 \\ M \text{~is odd~}}}  \frac{(\bdW \bdG \bdL^{1/2} t)^M}{(M-1)!!}, \\
\label{eq: P1}
& P_1(t) = 2 \bdW^2 \bdG \bdL +  3\bdW^3 \bdG^2 \bdL^{\frac{3}{2}} (1+t)  \sum^{\bar{M}}_{\substack{M=3 \\ M ~\text{is odd}}} \frac{(M-1)M}{(M-3)!!}(2\bdW \bdG \bdL^{1/2}t)^{M-2}
\end{align}
and the function $P_2(\cdot)$ is a polynomial of degree $\bar{M}-1$ and are independent of $h$. 
The constants $\theta_1$ and $\theta_2$ are given by $\theta_1 = 353$ and $\theta_2 = \sqrt{34}$.
\end{theorem}
The difference between the results of the inchworm Monte Carlo method and the exact solution is given by
\begin{theorem}\label{thm: final estimates}
Under the same settings as Theorem \ref{thm: bounds}, if we further assume that (H2) holds, then the difference between the inchworm Monte Carlo solution and the exact solution can be estimated by
\begin{itemize}
\item Bias estimation
\begin{equation}\label{final 1st error upper bound}
   \begin{split}
& \left\|\E\left(  G_{\emph{e}}(t_{n+1},t_m) - \tG_{n+1,m}   \right)\right\| \le  P^{\emph{e}}(t_{n-m+1})\left(\ee^{\theta_1 \sqrt{P_1(t_{n-m+1})} t_{n-m+1}}\right)\cdot h^2  \\
& \qquad + 4\theta_2^2 \bar{\alpha}(t_{n-m+1})\bar{\gamma}(t_{n-m+1})\left(\ee^{3 \theta_1 \sqrt{P_1(t_{n-m+1})} t_{n-m+1}}\right)\cdot \frac{h}{N_s},
   \end{split}
\end{equation}
\item Numerical error estimation
\begin{equation}\label{final 2ed error upper bound}
  \begin{split}
&\left[\E\left( \| G_{\emph{e}}(t_{n+1},t_m) - \tG_{n+1,m}\|^2 \right) \right]^{1/2} \le    P^{\emph{e}}(t_{n-m+1})\left(\ee^{\theta_1 \sqrt{P_1(t_{n-m+1})} t_{n-m+1}}\right)\cdot h^2  \\
&\hspace{120pt}+  \theta_2 \sqrt{\bar{\gamma}(t_{n-m+1})}\left(\ee^{ \theta_1 \sqrt{P_1(t_{n-m+1})} t_{n-m+1}}\right) \cdot \sqrt{\frac{h}{N_s}} .
  \end{split}
\end{equation}
\end{itemize}
Here the function $P^{\emph{e}}(t)$ is defined by
\begin{equation} \label{eq: Ce}
       P^{\emph{e}}(t) =  \left( \frac{1}{4} \bdH + 8P_1(t) \right)\bdG'' + \frac{5}{12} \bdG'''  +  \bdW \bdG'' \bdL^{1/2}  \left( \sum^{\bar{M}}_{\substack{M=1 \\ M \text{~is odd~}}} \frac{M+1}{(M-1)!!}(2\bdW \bdG \bdL^{1/2} t)^M \right).
\end{equation}
\end{theorem}
The above result indicates the following properties of the inchworm scheme, which are similar to the case of the differential equations:
\begin{itemize}
\item The bias again has only a small contribution to the numerical error, which is often hardly observable in the numerical experiments. 
\item The error consists of two parts. The first part is second-order in $h$, and the second part is half-order in the total number of samples.
\end{itemize}

The growth of the numerical error over time is more complicated compared to the ODE case. Since the function $P_1(t)$ is a polynomial of degree $\bar{M}-1$, the growth of the error is on the order of $\exp(C t^{(\bar{M}+1)/2})$. Clearly the growth rate depends on the choice of $\bar{M}$. In the numerical examples shown in \cite{Chen2017b,Cai2020}, only $\bar{M} = 1$ and $\bar{M} = 3$ are used in the applications. Regarding the behavior of the error growth for different $\bar{M}$, we remark that 
\begin{itemize}
\item When $\bar{M}=1$, the final error estimation \eqref{final 2ed error upper bound} shows that there exists constants $C_1$ and $C_2$ such that
\begin{displaymath}
\left[\E\left( \| \Ge(t_{n+1},t_m) - \tG_{n+1,m}\|^2 \right) \right]^{1/2} \le C_1 \ee^{C_2 t_{n-m+1}} \left( h^2 + \sqrt{\frac{h}{N_s}} \right),
\end{displaymath}
showing that the error grows exponentially with respect to the time difference in the propagator, which is slower than the method using Dyson series, where the error grows exponentially with respect to the square of the time difference. In this case, the numerical error is successfully mitigated.
\item When $\bar{M}=3$, there exist constants $C_1$ and $C_2$ such that
\begin{displaymath}
\left[\E\left( \| \Ge(t_{n+1},t_m) - \tG_{n+1,m}\|^2 \right) \right]^{1/2} \le C_1 \ee^{C_2 (t_{n-m+1}+t_{n-m+1}^2)} \left( h^2 + \sqrt{\frac{h}{N_s}} \right).
\end{displaymath}
In this case, the growth rate is exponential in $t^2$, which is the same as the Dyson series. Thus which method has greater error depends on the coefficient in front of $t^2$. Instead of a detailed analysis, we would just comment that the inchworm Monte Carlo method is likely to have a smaller coefficient due to the effect of partial resummation, which leads to less terms in  \eqref{eq: calH} than the original Dyson series.
\item For larger $\bar{M}$, the error growth $\ee^{\theta_1 \sqrt{P_1(t)}t}$ of the inchworm Monte Carlo method in general is faster than quadratic exponential for Dyson series. However when $t$ is small, the difference between $\sqrt{P_1(t)}t$ for $\bar{M}>3$ and $\bar{M}=3$ will not be too large (see Figure \ref{fig:P1}). Since the coefficient of $t^k$ is smaller when $k$ is larger, we may again expect the inchworm Monte Carlo method has slower overall error growth for short-time simulations. Here we would also like to mention that in practice, the departure of the curve for $\bar{M}=5$ from the curve for $\bar{M}=3$ is often observed to be much later than that indicated in Figure \ref{fig:P1}, since our estimation of the constants in $P_1(t)$ is based on the worst case and may not be optimal. It has been observed in the literature that the fast error growth may stay unnoticeable for a significant amount of time despite our error analysis \cite{Chen2017b}. 
\begin{figure}[h!]
    \centering
    \includegraphics[width=0.5\textwidth]{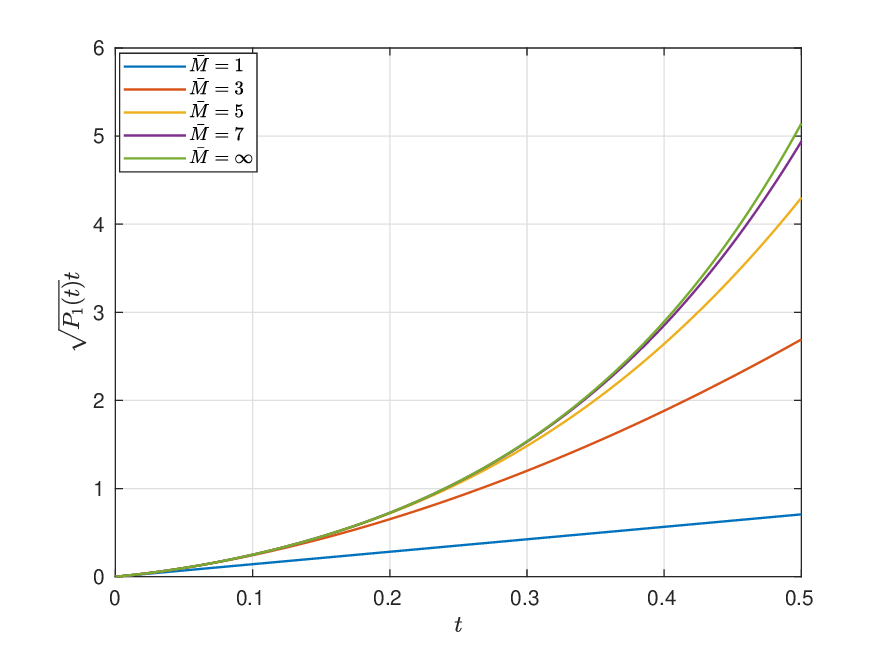}
    \caption{Graphs of $\sqrt{P_1(t)}\cdot t$ for various $\bar{M}$ with $\bdW=\bdG=\bdL = 1$. The curves for $\bar{M}=3$ and $\bar{M}=+\infty$ are almost identical before $t=0.15$.}
    \label{fig:P1}
\end{figure}
\item When $\bar{M} \rightarrow +\infty$, we have
  \begin{align*}
  \lim_{\bar{M}\rightarrow +\infty} \bar{\gamma}(t) &= 2\bdW^2 \bdG^2 \bdL t \ee^{ \frac{\bdW^2 \bdG^2 \bdL t^2}{2} }, \\
  \lim_{\bar{M}\rightarrow +\infty} P_1(t) &= 2\bdW^2 \bdG \bdL + 3\bdW^3 \bdG^2 \bdL^{\frac{3}{2}} (1+t) P(2\bdW \bdG \bdL^{1/2}t) \cdot \ee^{2\bdW^2 \bdG^2 \bdL t^2},
  \end{align*}
  where $P(x) = x^5 + 7x^3 + 6x$. Although these quantities are still finite, the error bound \eqref{final 2ed error upper bound} grows double exponentially with respect to $t^2$, which is undesired in applications.
\end{itemize}

The numerical experiments in \cite{Chen2017b,Cai2020} show that in certain regimes where the constant $\bdL$ is relatively small, the contribution from $M = 1$ is dominant in the series \eqref{eq: calH}. In this case, the inchworm Monte Carlo method can well suppress the numerical sign problem and achieve an exponential error growth in these applications.

\subsection{Outline of the proof}
We will postpone the details of the proof while provide an outline here. The results are obtained in the following steps:
\begin{itemize}
\item Estimate the derivatives of the right-hand sides (Propositions \ref{thm: first order derivative} and \ref{thm: second order derivative}).
\item Derive recurrence relations for the numerical error (Proposition \ref{thm: recurrence relations}).
\item Apply the recurrence relations to derive the error bounds (Theorem \ref{thm: bounds}).
\item Estimate the error of the deterministic method (Proposition \ref{thm: Runge Kutta error}).
\item Use the triangle inequality to derive the final error bounds (Theorem \ref{thm: final estimates}).
\end{itemize}
Some more details of these steps are given by a number of propositions below.

We first define some sets of 2D multi-indices that will be used.
\begin{align} 
\begin{array}{l l}
\Omega_{n,m}=\{(j,k) \in \Z^2 \ | \  m \le k < j \le n     \}, & \Omega^*_{n,m}= \Omega_{n+1,m}; \\
\partial \Omega_{n,m}=\{(j,k)\in \Omega_{n,m}\ | \ j=n\text{~or~}k=m\}, &\partial \Omega^*_{n,m}=\partial \Omega_{n+1,m}; \\
\mathring{\Omega}_{n,m}= \Omega_{n,m}\backslash \partial \Omega_{n,m}, &\mathring{\Omega}^{\ast}_{n,m}=\mathring{\Omega}_{n+1,m};\\
\Gamma_{n,m}(i) = \{(j,k)\in \Omega_{n,m}\ | \ j-k= i\}, & \Gamma^*_{n,m}(i)=\Gamma_{n+1,m}(i).
\end{array}
\end{align}
One may refer to Fig.~\ref{fig:sets} to visualize these definitions. Note that $\Omega_{n,m}$ and $\Omega^*_{n,m}$ respectively contain indices of the numerical solutions in $\gb_{n,m}$ and $\gb^*_{n,m}$ and thus give the locations of all nodes that $K_1$ and $K_2$ depend on. 
In addition, since $G^*_{n+1,m}$ is calculated completely based on the rest of $\gb^*_{n,m}$, we define $\bar{\Omega}_{n,m}=\Omega^*_{n,m}\backslash \{(n+1,m)\}$ to represent the indices of all full propagators that we actually use in order to obtain $G_{n+1,m}$.

\input{images/fig_sets.tex} 

For the analysis of ODEs, it has been assumed in \eqref{assump: bd} the boundedness of the first- and second-order derivatives of the rigth-hand side. Correspondingly, our first step is to estimate the derivatives of the functional $F_1$ and $F_2$ given by \eqref{def: scheme 1 k1} and \eqref{def: scheme 1 k2}. For $F_1$, the results are given by the following two propositions for first- and second-order derivatives respectively.

\begin{proposition}
Assume (H1)(H3)(R4) hold. Given the time step length $h$ and any $\xib_{n,m}$ being a convex combination of exact solution $\gb^{\mathrm{e}}_{n,m}$ given by \eqref{gbe}, numerical solution of scheme \eqref{def: scheme 1} $\gb_{n,m}$ and numerical solution of inchworm Monte Carlo method \eqref{def: scheme 2}, the first-order derivative of $F_1(\xib_{n,m})$ defined by \eqref{def: scheme 1 k1} w.r.t. the $pq-$entry ($p,q=1,2$) of $G_{k,\ell}$ is bounded by
\begin{equation} \label{eq: 1st order bound}
\left\| \frac{\partial F_1 (\xib_{n,m})}{\partial G_{k,\ell}^{(pq)}} \right\| \le
\begin{cases}
   P_1(t_{n-m}) h, &\text{for~} (k,\ell) \in \partial \Omega_{n,m} , \\
   P_1(t_{n-m}) h^2, &\text{for~}(k,\ell) \in \mathring{\Omega}_{n,m}, 
\end{cases}
\end{equation}
where $P_1(t)$ is defined in \eqref{eq: P1}.
\label{thm: first order derivative}
\end{proposition}

\begin{proposition}
Assume (H1)(H3)(R4) hold. Given the time step length $h$, the second-order derivative of $F_1(\xib_{n,m})$ defined in \eqref{def: scheme 1 k1} w.r.t. the $p_1 q_1-$entry of $G_{k_1,\ell_1}$ and the $p_2 q_2-$entry of $G_{k_2,\ell_2}$ ($p_i,q_i = 1,2$) is bounded by:

\begin{itemize}
\item If $(k_1,\ell_1) \times (k_2,\ell_2) \in \partial \Omega_{n,m} \times \partial \Omega_{n,m}$,
\begin{equation}\label{eq:2ed deriv bdxbd}
 \begin{split}
\left\| \frac{\partial^2 F_1 (\xib_{n,m})}{\partial G_{k_1,\ell_1}^{(p_1 q_1)}\partial G_{k_2,\ell_2}^{(p_2 q_2)}} \right\| \le 
\begin{cases}
 P_2(t_{n-m}) h, &\text{if one of the conditions \textbf{(a)}-\textbf{(d)} holds}, \\
  P_2(t_{n-m}) h^2, &\text{otherwise},
\end{cases}
   \end{split}
\end{equation}
where the conditions \textbf{(a)}-\textbf{(d)} are given by 
\begin{align*}
    &\textbf{(a)} \ k_1 = k_2 = n, \ (\ell_1,\ell_2) \in \{(n-1,m),(m,n-1)\}, \\
    &\textbf{(b)} \ \ell_1 = \ell_2 = m, \ (k_1,k_2) \in \{(m+1,n),(n,m+1)\}, \\
     &\textbf{(c)} \ k_1 = n \text{~and~} \ell_2 = m, \  (k_2,\ell_1) \in \big\{ m \le \ell_1 \le n-1,m+1 \le k_2 \le n \, \big\vert \, |\ell_1 - k_2|\le 1 \big\}, \\
    &\textbf{(d)} \ k_2 = n \text{~and~} \ell_1 = m, \  (k_1,\ell_2) \in \big\{ m \le \ell_2 \le n-1,m+1 \le k_1 \le n  \, \big\vert \, |\ell_2 - k_1|\le 1 \big\} .
\end{align*}

\item If $(k_1,\ell_1)\times (k_2,\ell_2) \in \partial \Omega_{n,m} \times  \mathring{\Omega}_{n,m}$, 
\begin{equation}\label{eq:2ed deriv bdxint}
 \begin{split}
\left\| \frac{\partial^2 F_1 (\xib_{n,m})}{\partial G_{k_1,\ell_1}^{(p_1 q_1)}\partial G_{k_2,\ell_2}^{(p_2 q_2)}} \right\| \le 
\begin{cases}
 P_2(t_{n-m}) h^2, &\text{for~}k_1 =n,|k_2-\ell_1|\le 1\text{~or~}\ell_1=m,|k_1 -\ell_2|\le 1 , \\
  P_2(t_{n-m}) h^3, &\text{otherwise}. 
\end{cases}
   \end{split}
\end{equation}

\item If $(k_1,\ell_1)\times (k_2,\ell_2) \in   \mathring{\Omega}_{n,m} \times \partial \Omega_{n,m}$, 
\begin{equation}\label{eq:2ed deriv bdxint 2}
 \begin{split}
\left\| \frac{\partial^2 F_1 (\xib_{n,m})}{\partial G_{k_1,\ell_1}^{(p_1 q_1)}\partial G_{k_2,\ell_2}^{(p_2 q_2)}} \right\| \le 
\begin{cases}
 P_2(t_{n-m}) h^2, &\text{for~}k_2 =n,|k_1-\ell_2|\le 1\text{~or~}\ell_2=m,|k_2 -\ell_1|\le 1 , \\
  P_2(t_{n-m}) h^3, &\text{otherwise}. 
\end{cases}
   \end{split}
\end{equation}

\item If $(k_1,\ell_1)\times(k_2,\ell_2) \in  \mathring{\Omega}_{n,m} \times  \mathring{\Omega}_{n,m} $,
\begin{equation}\label{eq:2ed deriv intxint}
  \begin{split}
\left\| \frac{\partial^2 F_1 (\xib_{n,m})}{\partial G_{k_1,\ell_1}^{(p_1 q_1)}\partial G_{k_2,\ell_2}^{(p_2 q_2)}} \right\| \le 
\begin{cases}
 P_2(t_{n-m}) h^3, &\text{for~}|k_1-\ell_2|\le1\text{~or~}|k_2-\ell_1|\le 1 , \\
 P_2(t_{n-m}) h^4, &\text{otherwise},
\end{cases}
  \end{split}
\end{equation}
\end{itemize}
\label{thm: second order derivative} 
where $P_2(t)$ is a polynomial of degree $\bar{M}-1$ independent of $h$.
\end{proposition}

In these two propositions, the functions $P_1(\cdot)$ and $P_2(\cdot)$ are the same as the corresponding functions in Theorems \ref{thm: bounds} and \ref{thm: final estimates}, respectively. The proofs of these two propositions are quite technical and tedious and thus we move them to \ref{sec: derivatives}. Unlike the case of differential equations, the partial derivative of $F_1(\cdot)$ involves a number of previous numerical solutions (all red nodes in Figure \ref{fig:sets}), and the magnitudes depend on the locations of the nodes, as forms different cases in the above propositions. Similar results for the derivatives of $F_2(\cdot)$ with the same functions $P_1(t),P_2(t)$ can be proven, where all the indices $n$ should be changed to $n+1$.

With the above estimates for the derivatives, we can establish recurrence relations for the bias and the numerical error:
\begin{proposition}
\label{thm: recurrence relations}
Let $\Delta \gb_{n,m}= \tg_{n,m} - \gb_{n,m}$ and $\Delta \gb^*_{n,m}= \tg^*_{n,m} - \gb^*_{n,m}$. Given a sufficiently small time step length $h$ and a sufficiently large $N_s$, if the assumptions (H1) and (H3) hold, the difference between the deterministic solution and the Monte Carlo solution can be estimated by
\begin{itemize}
\item Bias estimation:
\begin{equation}\label{eq: recurrence 1}
  \begin{split}
&\|\E(\dG_{n+1,m})\|\\
\le&  \   (1+\frac{1}{8}\bdH^4 h^4)\|\E(\dG_{n,m})\| 
+ 22P_1(t_{n-m+1}) h^2  \sum_{i = 1}^{n-m} \left( 2 + (n-m+1-i)h  \right)  \left\| \E ( \Delta \gb^*_{n,m}) \right\|_{\Gamma^*_{n,m}(i)}\\
& + \left(\frac{7}{2}\bar{\alpha}(t_{n-m+1})h\left[ \Ns^{(\emph{std})}_{\bar{\Omega}_{n,m}}(\Delta \gb^*_{n,m}) \right]^2 + 8 \bar{\alpha}(t_{n-m+1})\bar{\gamma}(t_{n-m+1})\cdot \frac{h^3}{N_s} \right). 
  \end{split}
\end{equation}

\item Numerical error estimation: 
\begin{equation}\label{eq: recurrence 2 1/2}
    \begin{split}
     &[\E(\|\dG_{n+1,m}\|^2)]^{1/2}   \le   (1+\frac{1}{8}\bdH^4 h^4)[\E(\|\dG_{n,m}\|^2)]^{1/2} \\
&\hspace{10pt}+ 22 P_1(t_{n-m+1}) h^2  \sum_{i = 1}^{n-m} \big( 2 + (n-m+1-i)h  \big) \Ns^{(\emph{std})}_{\Gamma^*_{n,m}(i)}(\Delta \gb^*_{n,m}) + \frac{7}{2}  \sqrt{\bar{\gamma}(t_{n-m+1})}\cdot \frac{h}{\sqrt{N_s}},
    \end{split}
\end{equation}
and   
\begin{equation}\label{eq: recurrence 2}
    \begin{split}
   & \E(\|\dG_{n+1,m}\|^2)
   \le  ( 1 + \frac{1}{4}\bdH^4 h^4 )\cdot\E(\|\dG_{n,m}\|^2) + (1+ \frac{1}{8}\bdH^4 h^4)h \left[\E(\|\dG_{n,m}\|^2)\right]^{1/2} \times\\ 
     & \hspace{-20pt} \left\{  44 P_1(t_{n-m+1}) h  \sum_{i = 1}^{n-m} \big( 2 + (n-m+1-i)h  \big) \Ns^{(\emph{std})}_{\Gamma^*_{n,m}(i)}(\Delta \gb^*_{n,m})  +  4 \bar{\alpha}(t_{n-m+1})\bar{\gamma}(t_{n-m+1})\cdot \frac{h^2}{N_s} \right\}\\
     & \hspace{-10pt}+912 P_1^2(t_{n-m+1}) h^4 \left\{  \sum_{i = 1}^{n-m} \big( 2 + (n-m+1-i)h  \big) \Ns^{(\emph{std})}_{\Gamma^*_{n,m}(i)}(\Delta \gb^*_{n,m})  \right\}^2+ 17\bar{\gamma}(t_{n-m+1})\cdot \frac{h^2}{N_s},
   \end{split}
\end{equation}
\end{itemize}
where the functions $\bar{\alpha}$ and $\bar{\gamma}$ are defined in \eqref{eq: alpha gamma}.
\end{proposition}
In the above proposition, two different recurrence relations are given for the numerical error. Note that \eqref{eq: recurrence 2} is not a simple square of the estimation \eqref{eq: recurrence 2 1/2}. The main reason lies in the term $4 \bar{\alpha}(t_{n-m+1})\bar{\gamma}(t_{n-m+1})\cdot h^2/N_s$ located at the end of the second line in \eqref{eq: recurrence 2}. Below we are going to use a simple analog to help the readers understand the difference. Consider the two recurrence relations
\begin{align} \label{eq: analog1}
e_{n+1} &\le e_n + \frac{h}{\sqrt{N_s}}, \\
\label{eq: analog2}
e_{n+1}^2 &\le e_n^2 + \frac{2h^2}{N_s} e_n + \frac{h^2}{N_s}.
\end{align}
The relation between these two recurrence relations are analogous to the relation between \eqref{eq: recurrence 2 1/2} and \eqref{eq: recurrence 2}. The square of the first recurrence relation \eqref{eq: analog1} is
\begin{displaymath}
e_{n+1}^2 \le e_n^2 + \frac{2h}{\sqrt{N_s}} e_n + \frac{h^2}{N_s},
\end{displaymath}
where the cross-term is different from \eqref{eq: analog2}. However, the relation \eqref{eq: analog2} provides a higher numerical order than \eqref{eq: analog1}, since by Cauchy-Schwarz inequality, we can derive from \eqref{eq: analog2} that
\begin{equation} \label{eq: cs}
e_{n+1}^2 \le \left(1 + \frac{h^2}{N_s}\right) e_n^2 + \frac{2h^2}{N_s},
\end{equation}
indicating that $e_n \sim O(\sqrt{h/N_s})$, while the recurrence relation \eqref{eq: analog1} can only give us $e_n \sim O(\sqrt{1/N_s})$. Besides, in order to study the error growth rate with respect to time, we are not allowed to use the Cauchy-Schwarz inequality to simplify equation \eqref{eq: recurrence 2} like \eqref{eq: cs}. Later in our proof, the simpler equation \eqref{eq: recurrence 2 1/2} will be used the determine the growth rate of the numerical error, while the more complicated version \eqref{eq: recurrence 2} is responsible for the final error estimation. Theorem \ref{thm: bounds} is obtained from the recurrence relations stated in Proposition \ref{thm: recurrence relations}.

To obtain the final estimates (Theorem \ref{thm: final estimates}), we need to estimate the error of the deterministic scheme \eqref{def: scheme 1}, which is given by
\begin{proposition}\label{thm: Runge Kutta error}
We define the deterministic error $E_{n+1,m} = \Ge(t_{n+1}, t_m) - G_{n+1,m}$. If the assumptions (H1)(H2)(H3) hold, then for a sufficiently small time step length $h$ and a sufficiently large number of samples at each step $N_s$, we have
\begin{equation}\label{eq: Runge Kutta error estimate global}
    \|E_{n+1,m}\| \le  P^{\emph{e}}(t_{n-m+1})\left(\ee^{\theta_1 \sqrt{P_1(t_{n-m+1})} t_{n-m+1}}\right)\cdot h^2
\end{equation}
where $P^{\emph{e}}(t)$ is defined in \eqref{eq: Ce}, and the constant $\theta_1$ is the same as the one in Theorem \ref{thm: final estimates}.
\end{proposition}

It is easy to see that our final conclusions in Theorem \ref{thm: final estimates} are a straightforward combination of Theorem \ref{thm: bounds} and Proposition \ref{thm: Runge Kutta error} by the triangle inequality.

\section{Numerical experiments}\label{sec: numer exp}
In this section, we will verify the above statements using numerical experiments. The following two subsections will be devoted, respectively, to the case of differential equations and the inchworm Monte Carlo method.

\subsection{Numerical experiments for ordinary differential equations} \label{sec: numer exp diff}
We consider an example as the following ordinary differential equation:
\begin{equation}\label{numexp:diff eq}
  \begin{split}
&\frac{\dd u}{\dd t} = -\frac{\ii}{2} K u(t) = \E_X \big(R(u,X)\big),\ t\in[0,T], \\
& R(u,X) = - \ii Xu
   \end{split}
\end{equation}
with the initial condition $u(0) = 1$ and the random variable $X \sim U(0,K)$.

We apply the two schemes proposed in \eqref{eq:two schemes} to get the numerical solutions $u_n$ and $\tu_n$ with uniform time step length $h = T/N$. For the stochastic $\tu_n$, we carry out the experiments independently for $N_{\exp} = 100N N_s$ times to obtain $\tu^{(1)}_n,\tu^{(2)}_n,\cdots,\tu^{(N_{\exp})}_n$ and we approximate the numerical error by
\begin{equation}\label{def: second moment diff}
\E(|u_n-\tu_n|_2^2) \approx \mu_n :=\frac{1}{N_{\exp}}\sum^{N_{\exp}}_{i=1} |u_n -\tu_n^{(i)}|^2, \qquad \text{for } n = 0,1,\cdots,N.
\end{equation}
Based on these settings, we now focus on the numerical order of the scheme and the growth of the numerical error with respect to $t$. For given time step $h$, we define the error function $e(\cdot)$ by $e(nh) = \mu_n$.

We first set $K=3$ and $K=10$ in \eqref{numexp:diff eq} and $T = 3$. Figure \ref{fig:evolution diff} shows the evolution of the numerical error $e(t)$ for $h = \frac{1}{4}$ and various numbers of samples $N_s$. For $K = 10$, the left panel of Figure \ref{fig:evolution diff} shows that the error grows exponentially over time as predicted in Theorem \ref{thm: diff bounds}, while for smaller $K$, the stability of the method takes effect, and the error grows only linearly up to $T=3$ as exhibited in the right panel of Figure \ref{fig:evolution diff}. This verifies that the exponential growth can be well controlled if appropriate Runge-Kutta schemes and sufficiently small time steps are adopted, which avoids the numerical sign problem in the Monte Carlo method that directly calculates \eqref{dyson}.

\begin{figure}[h]
    \centering
    \includegraphics[width=\textwidth]{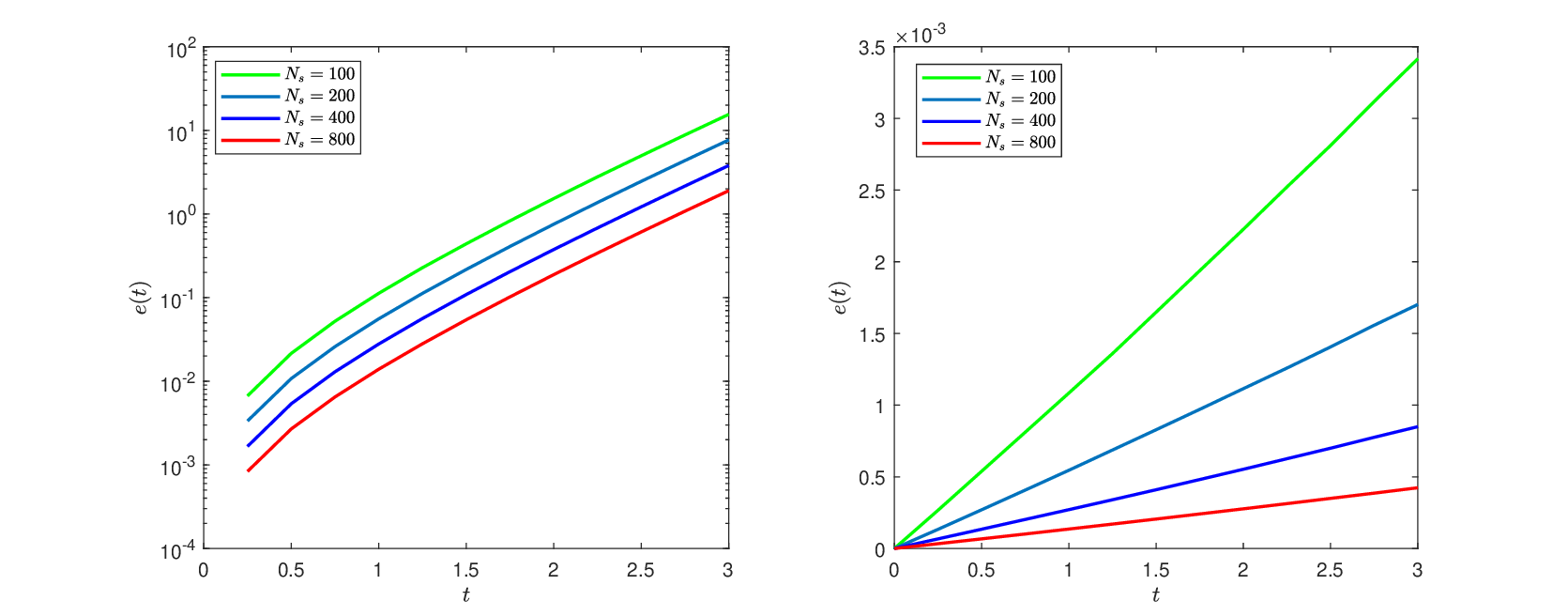}
    \caption{Evolution of numerical error $e(t)$ (left: $K=10$, right: $K=3$).}
    \label{fig:evolution diff}
\end{figure}

To verify the convergence rate with respect to $h$ and $N_s$ in the estimate \eqref{diff 2ed error upper bound}, we set $K=1$ and $T=1$ in \eqref{numexp:diff eq} and consider the numerical error at $t=0.5$ and $t=1$.
We first fix $N_s=100$ and reduce $h$ from $1/2$ to $1/64$, and then fix $h = 1/4$ and increase $N_s$ from $100$ to $3200$. The numerical errors are listed in Table \ref{tab:convergence rate}, from which we can easily observe the first-order convergence for both $h$ and $N_s$, as agrees with our estimate \eqref{diff 2ed error upper bound}.

\begin{table}[h]
\centering \small 
\begin{tabular}{| c |c c c c|| c | c c c c|}
\hline
$h, N_s$ & $e(0.5)$ & order & $e(1)$ & order & $h,N_s$  &  $e(0.5)$ & order & $e(1)$ & order \\
\hline
1/2, 100 &  1.0917e-04& -- & 2.1940e-04 &  -- & 1/4, 100 &  5.2721e-05 &  -- & 1.0593e-04 &   --  \\
1/4, 100 & 5.2721e-05 & 1.0502 & 1.0593e-04 &  1.0505 & 1/4, 200 &  2.6533e-05 & 0.9906 & 5.3332e-05 &  0.9901 \\
1/8, 100 & 2.6257e-05 &  1.0057 & 5.2776e-05 &  1.0052 & 1/4, 400 & 1.3210e-05 & 1.0062 &  2.6520e-05 & 1.0079  \\
1/16, 100 & 1.3027e-05 &  1.0111 &  2.6039e-05 &   1.0192 & 1/4, 800 & 6.6185e-06 & 0.9970 & 1.3254e-05 &     1.0007 \\
1/32, 100 & 6.5086e-06 & 1.0011 &1.3013e-05 &  1.0007 & 1/4, 1600 & 3.3043e-06 &  1.0022 & 6.5942e-06 &    1.0072  \\
1/64, 100 & 3.2579e-06 & 0.9984 & 6.5124e-06 &  0.9987 & 1/4, 3200 & 1.6528e-06 & 0.9995 & 3.3060e-06 & 0.9961  \\
\hline
\end{tabular}
       \caption{Numerical error $e(0.5),e(1)$ with different time step $h$ and number of samples $N_s$ and the order of accuracy.}
    \label{tab:convergence rate}
\end{table}

\subsection{Numerical experiments for the inchworm Monte Carlo method}\label{sec: numer exp integ diff}
To verify the error growth of the inchworm Monte Carlo method, we consider the spin-boson model where the system Hamiltonian has the energy difference $\epsilon =0.1$ and frequency of the spin flipping $\Delta =1$. For the bath part, we assume an Ohmic spectral density, which formulates the two-point correlation as 
\begin{equation} \label{def: B}
B(\tau_1, \tau_2) = \sum_{l=1}^L \frac{c_l^2}{2\omega_l} \left[
  \coth \left( \frac{\beta \omega_l}{2} \right) \cos \big( \omega_l (|\tau_1 -t| - |\tau_2 -t|) \big)
  - \ii \sin\big( \omega_l(|\tau_1 -t| - |\tau_2 -t|) )
\right] 
\end{equation}
where the coupling intensity $c_l$ and frequency of each harmonic oscillator $\omega_l$ above are respectively given by 
\begin{displaymath}
c_l = \omega_l \sqrt{\frac{\xi \omega_c}{L} [1 - \exp(-\omega_{\max}/\omega_c)]}, \quad \omega_l = -\omega_c
  \ln \left( 1 - \frac{l}{L} [1 - \exp(-\omega_{\max} / \omega_c)] \right),
  \quad l = 1,\cdots,L.
\end{displaymath}
As for the parameters above, we set $L = 200$, $\omega_{\max} = 4\omega_c$ with the primary frequency $\omega_c =3$, $\xi = 0.6$ and $\beta = 5$ throughout our experiments.

\subsubsection{Evolution of observable}
The observable of interest is set to be $O = \hat{\sigma}_z \otimes \mathrm{Id}_b$ which only acts on the system part, and the initial density matrix $\rho = \rho_s \otimes \rho_b$ is given by 
\begin{displaymath}
   \rho_s = \ket{1} \bra{1}  = \begin{pmatrix} 1 & 0 \\ 0 & 0 \end{pmatrix}\quad  \text{~and~} \quad  \rho_b = Z^{-1} \exp(-\beta H_b)\,,
\end{displaymath}
where $Z$ is a normalizing factor satisfying $\tr(\rho_b) = 1$. Therefore, evolution of the observable $\langle \hat{\sigma}_z(t)\rangle$ can be approximated via inchworm Monte Carlo method by 
\begin{displaymath}
\langle \hat{\sigma}_z(j h)\rangle \approx \bra{1} \tG_{N+j,N-j} \ket{1} =\tG^{(11)}_{N+j,N-j}, \text{~for~} j = 0,1,\cdots,N
\end{displaymath}
where $\tG_{n,m} \in \C^{2\times 2}$ is obtained by the scheme \eqref{def: scheme 2} and recall that $\tG^{(11)}_{n,m}$ is the $(1,1)$ entry. Such evolution is plotted in Figure \ref{fig:observable}. Note that due to the numerical error, the computed $\langle \hat{\sigma}_z(t)\rangle$ may contain a nonzero imaginary part, and here only the real part of the numerical result is plotted.

\begin{figure}[h]
    \centering
    \includegraphics[width=0.5\textwidth]{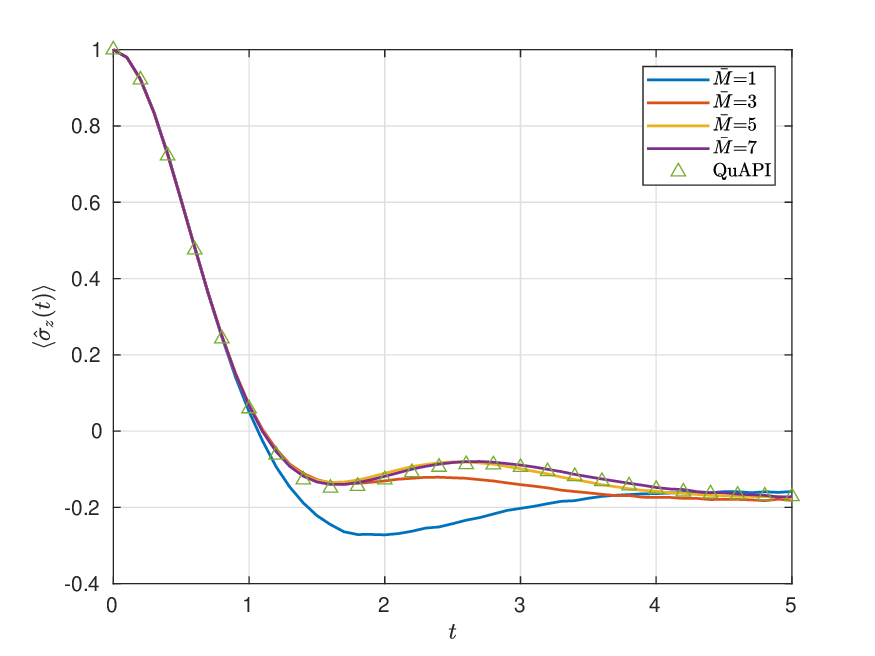}
    \caption{Evolution of $\text{Re}\langle \hat{\sigma}_z(t)\rangle$ by inchworm Monte Carlo method.}
    \label{fig:observable}
\end{figure}

The numerical results in Figure \ref{fig:observable} are obtained using the inchworm Monte Carlo method \eqref{def: scheme 2} with time step $h = 1/10$. We choose $N_s=10^4$ for $\bar{M}=1$, $N_s=10^5$ for $\bar{M}=3$, $N_s=10^6$ for $\bar{M}=5$ and $N_s=10^7$ for $\bar{M}=7$. For larger $\bar{M}$, the sign problem becomes more severe and thus we need more samples for each Monte Carlo integration in order for the curves to be sufficiently smooth. In addition, the results by the iterative QuAPI method \cite{Makri1995,Makri1998} are given as the reference solutions. 

One can observe from Figure \ref{fig:observable} that when $t < 1$, the four curves are hardly distinguishable, 
which agrees with the fact that smaller $\bar{M}$ is required for shorter-time simulations. For larger $t$, the truncation with $\bar{M}=1$ becomes inadequate, while $\bar{M}=3$ still provides reasonable approximation to the solution up to $t=5$. The convergence with respect to $\bar{M}$ can be observed by further increasing the terms in the truncated series.
In these simulations, due to the sufficient number of samples, the stochastic error is mostly suppressed, and one can see that the numerical result with $\bar{M}=5$ can already provide a satisfactory matching with the reference solution thanks to the rapid convergence of the series after partial resummation. 

As a comparison, the values of $\text{Re}\langle \hat{\sigma}_z(t)\rangle$ computed using the Dyson series truncated with $\bar{M}=2,4,6,8$ are plotted in Figure \ref{fig:observable_dyson}. Here $\bar{M}$ is the maximum value of $M$ when we truncate the Dyson series \eqref{eq:Ge_Dyson}, and it is chosen as an even number since the terms with odd $M$ are all zero. All these numerical results are obtained based on $N_s=10^8$. When $t<2$ (see the left panel), the curves appear to converge to the reference solution provided by the QuAPI method. However, the convergence is obviously much slower compared to the inchworm method. Even when $\bar{M}=8$, the numerical result of the Dyson series is still considerably far away from the reference solution around $t = 2$. Meanwhile, due to the quadratic exponential sign problem of Dyson series, the curves for $\bar{M}=6$ and $\bar{M}=8$ become oscillatory after $t=3$ despite of the large $N_s$ we use, and eventually turn out to be unreliable as the solution continues to evolve (see the right panel). This shows that Dyson series has encountered a faster error growth than inchworm method for $t<5$ and thus demonstrates that the inchworm method can effectively mitigate the sign problem for short-time simulations via partial resummation.

\begin{figure}[h]
    \centering
    \includegraphics[width=\textwidth]{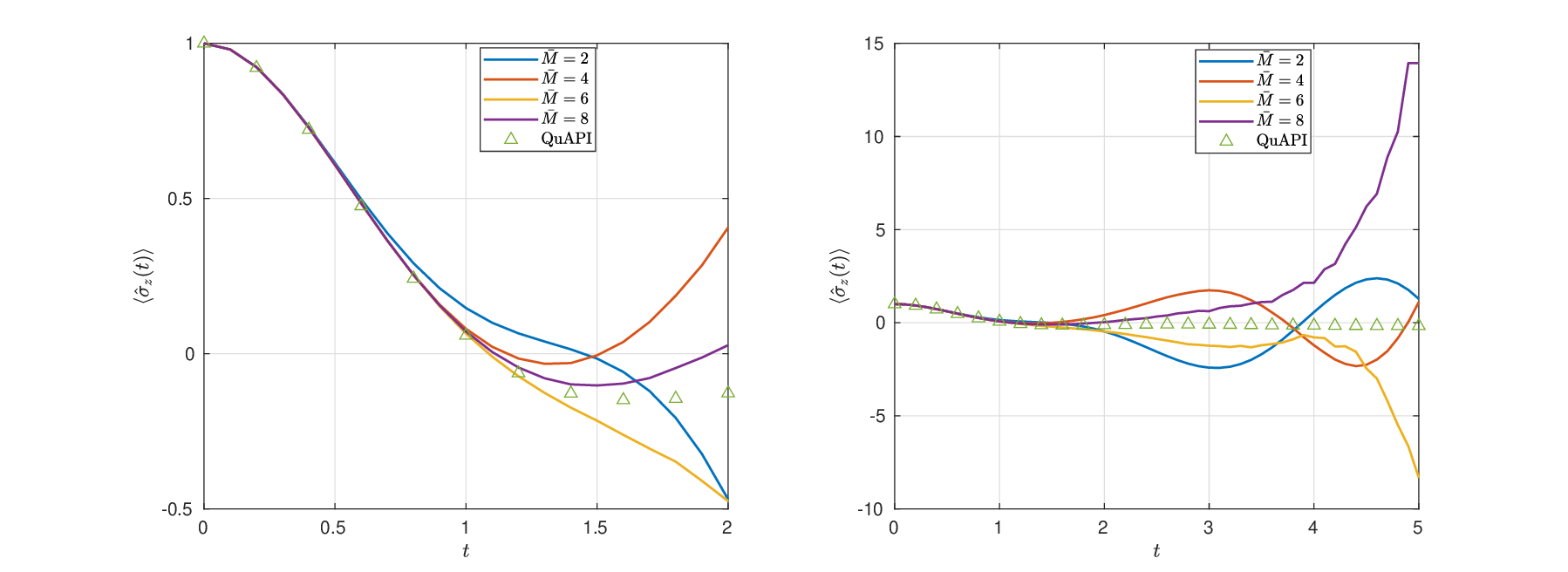}
    \caption{Evolution of $\text{Re}\langle \hat{\sigma}_z(t)\rangle$ by Dyson series (left: evolution up to $t=2$; right: evolution up to $t=5$).}
    \label{fig:observable_dyson}
\end{figure}

\subsubsection{Order of convergence}
Unlike the differential equation case, now it is much harder to find the solution of the deterministic scheme due to the high-dimensional integral on the right-hand side of \eqref{eq: inchworm equation}. Therefore, instead of verifying \eqref{2ed error upper bound} directly, we use
\begin{equation}
    \E(\|  \tG_{n,m} - G_{n,m}  \|^2) = \Var(\tG_{n,m}) + \|\E(\tG_{n,m}-G_{n,m})\|^2,
\end{equation}
and only take the first term on the right-hand side (the variance of $\tG_{n,m}$) to approximate our numerical error. Such an approximation is reasonable since the second term $\|\E(\tG_{n,m}-G_{n,m})\|^2$ has a higher order $O(h^2/N^2_s)$ by the bias estimation \eqref{1st error upper bound}. To compute the variance numerically, we run the same simulation $N_{\exp}$ times, and compute the unbiased estimation of the variance:
\begin{displaymath}
\Var(\tG_{n,m}) \approx \bar{\mu}_{n,m} := \frac{N_{\exp}}{N_{\exp} - 1} \Bigg(  \frac{1}{N_{\exp}}\sum_{k=1}^{N_{\exp}} \|\tG_{n,m}^{[k]} \|^2   -  \Bigg\|\frac{1}{N_{\exp}} \sum_{i=1}^{N_{\exp}} \tG_{n,m}^{[k]} \Bigg\|^2   \Bigg),
\end{displaymath}
where $\tG_{n,m}^{[k]}$ is the result of the $k$th simulation. For a given time step $h$, we let $e(jh) = \bar{\mu}_{N+j,N-j}$. Below we first check the numerical order for $\bar{M}=3$. By choosing $N_{\exp} = 1000N N_s$, we get results shown in Table \ref{tab:order}, from which one can clearly observe the order of convergence given in \eqref{2ed error upper bound}.
\begin{table}[h]
\centering \small
\begin{tabular}{| c |c c c c|| c | c c c c|}
\hline
$h,N_s$ & $e(0.5)$ & order & $e(1)$ & order & $h,N_s$  & $e(0.5)$ & order & $e(1)$ & order \\
\hline
$1/10, 2$ & 0.0417 & -- &  0.1488 & -- & $1/4,1$ & 0.1939 & -- & 0.8579 & -- \\
$1/12, 2$ & 0.0350 & 0.9658  & 0.1228 & 1.0505 & $1/4,2$ & 0.0972 & 0.9959 & 0.3908 & 1.1344 \\
$1/14, 2$ &  0.0303 & 0.9293 & 0.1051 & 1.0083 & $1/4,4$ & 0.0473 &  1.0386 & 0.1824 & 1.0990 \\
$1/16, 2$ & 0.0263 & 1.0574 & 0.0915 & 1.0409 & $1/4,8$ & 0.0237 & 0.9998 & 0.0886 & 1.0417 \\
$1/18, 2$ & 0.0237 & 0.8936 & 0.0811 & 1.0203 & $1/4,16$ & 0.0119 & 0.9962 & 0.0436 & 1.0235 \\
$1/20, 2$ & 0.0214 & 0.9614 & 0.0728 & 1.0341 & $1/4,32$ & 0.0059 & 1.0053 & 0.0217 & 1.0091 \\
\hline
\end{tabular}
\caption{Numerical error $e(0.5)$, $e(1)$ with different time step $h$ and number of samples $N_s$ and the order of accuracy} \label{tab:order}
\end{table}

\subsubsection{Error growth}

Despite the good performance the inchworm Monte Carlo method in this example, we will show in this section that the evolution of the numerical error indeed follows the asymptotic behavior in our error analysis.
To this end, we compare the growth of error using inchworm Monte Carlo method with $\bar{M}=1$ and $\bar{M}=3$ in Figure \ref{fig:error_growth}, where the time step is set to be $h = 1/8$, and we choose $N_{\exp} = 700N N_s$. As predicted, the two curves almost coincide for $t < 1$. For $\bar{M} = 1$, the numerical error starts to show the exponential growth from $t = 4.5$, and for $\bar{M} = 3$, the quadratic exponential growth becomes obvious from $t = 2.5$. Both results are in accordance with the theoretical results in Theorem \ref{thm: bounds}.

\begin{figure}[h]
    \centering
    \includegraphics[width=\textwidth]{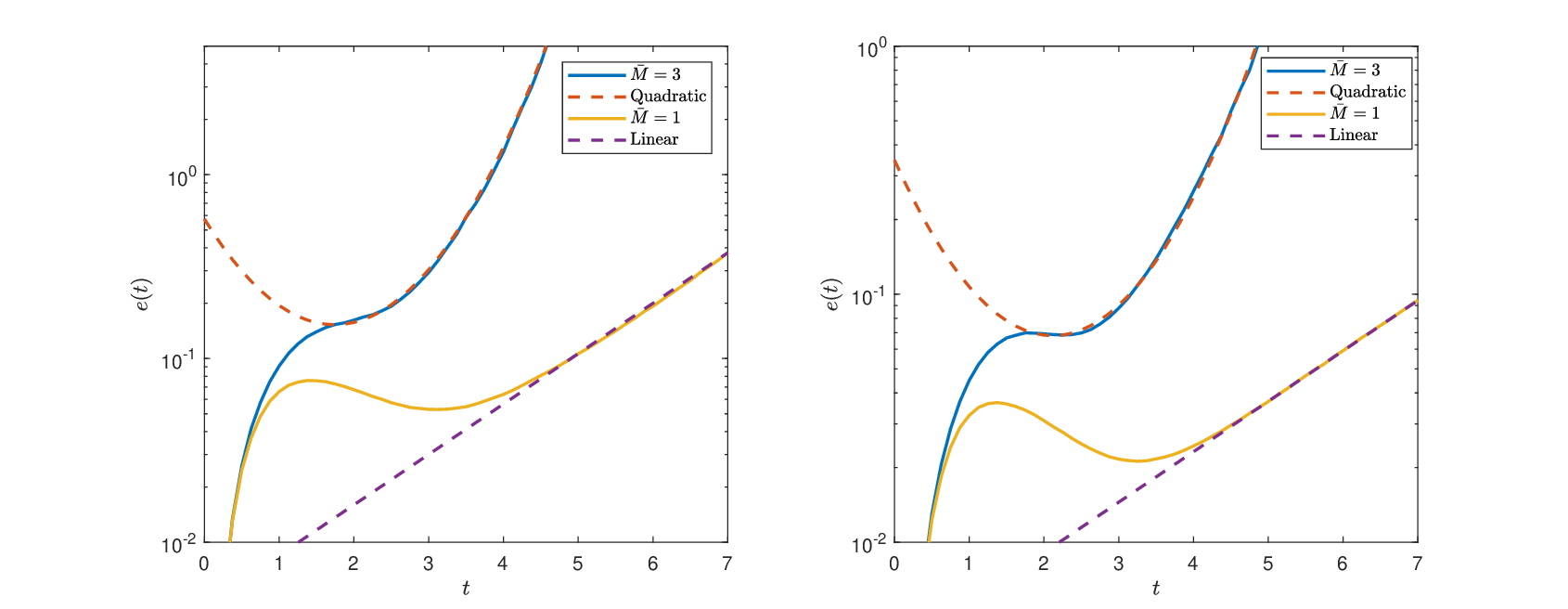}
    \caption{Evolution of numerical error $e(t)$ (left: $N_s=4$, right: $N_s = 8$).}
    \label{fig:error_growth}
\end{figure}

By now, we have stated all the results in this paper. From the next section, we will start to prove the theorems and propositions.

\section{Proofs for the case of differential equations}
\label{sec: proof diff eq}
In this section, we prove the results for differential equations as stated in Section \ref{sec: diff results}.

\subsection{Proof of Proposition \ref{thm: diff recurrence relations} --- Part I: Recurrence relation for the bias}
\label{sec: diff recurrence 1}
In this section we focus on the proof of \eqref{eq: diff recurrence 1}. By taking the difference of the schemes \eqref{def: scheme diff rk} and \eqref{def: scheme diff mc} and applying the triangle inequality and the bounds of the coefficients, we get
\begin{equation}\label{eq: evalue_1st_err}
\|\E(u_{n+1} -\tu_{n+1})\|_2 \le \|\E (u_{n} -\tu_{n})\|_2 + h R \sum^s_{i=1} \|\E(k_i - \tk_i)\|_2
\end{equation}
for all non-negative integer $n$. We then focus on the estimate for $\|\E(k_i - \tk_i)\|_2$. In fact, we have the following results:
\begin{lemma}\label{lemma: deltak and deltak^2}
Given a sufficiently small time step length $h$. If the boundedness assumption \eqref{assump: bd} are satisfied, we have 
\begin{equation}\label{eq: est_krk_1st_global}
 \begin{split}
\|\E(k_{i} - \tk_{i})\|_2 \le \alpha' \Big( \|\E(u_n -\tu_n)\|_2  + \E\big(\|u_n -\tu_n\|_2^2\big) + \frac{h^2}{N_s} \rkn^2 \Big)
  \end{split}
\end{equation}
and 
\begin{equation}\label{eq: est_krk_2ed_global}
\E(\|k_{i} - \tk_{i}\|_2^2) \le   \beta' \Big(  \E\big(\|u_n -\tu_n\|_2^2\big)  + \frac{1}{N_s} \rkn^2 \Big),
\end{equation}
for $\beta' = 2^{s+1} \max(2ds M'^2, 1)$ and $\alpha'= 2^s\sqrt{d} \max(M', 2s M'', s^2M'' R^2 \beta')$. Here we recall that $s$ is the number of Runge-Kutta stages, $d$ is the dimension of solution $u$ and $R,M',M''$ are some upper bounds defined in \eqref{assump: bd}--\eqref{assump: rk bd}.
\end{lemma}

With the above Lemma, one may insert the estimate \eqref{eq: est_krk_1st_global} into \eqref{eq: evalue_1st_err} to get the recurrence relation \eqref{eq: diff recurrence 1} stated in Proposition \ref{thm: diff recurrence relations} for the bias $\|\E(u_{n+1}-\tu_{n+1})\|_2$. The proof of Lemma \ref{lemma: deltak and deltak^2} is given below:

\begin{proof}[Proof of Lemma \ref{lemma: deltak and deltak^2}]
Apply the relation \eqref{def: g} and use Taylor expansion at the deterministic point $\big(t_n + c_{i} h, u_{n} + h \sum^{i-1}_{j = 1} a_{ij} k_j\big)$, we have for the $m$th component of $\E(k_{i} - \tk_{i})$, 
\begin{equation}\label{eq: k^1 expansion}
  \begin{split}
|\E (k^{(m)}_{i} - \tk^{(m)}_{i})| &=
\left|\E \Big( f^{(m)}\big(t_n + c_{i} h, u_{n} + h \sum^{i-1}_{j = 1} a_{ij} k_j\big) - f^{(m)}\big(t_n + c_{i} h, \tu_n + h \sum^{i-1}_{j = 1} a_{ij} \tk_j\big) \Big)\right| \\
&\le M'\|\E w_i\|_2 + \frac{M''}{2} \E \|w_i\|_2^2.
  \end{split}
\end{equation}
where
\begin{displaymath}
w_i = (u_n -\tu_n)+h\sum^{i-1}_{j = 1} a_{ij} (k_j - \tk_j),
\end{displaymath}
and we have applied the boundedness assumption \eqref{assump: bd}. The above inequality immediately yields
\begin{equation}
\begin{split}
  \|\E (k_{i} - \tk_{i})\|_2 
 & \le \sqrt{d}  M' \Big( \|\E(u_n -\tu_n)\|_2 + h R \sum^{i-1}_{j  = 1}\|\E(k_j - \tk_j)\|_2 \Big) \\
 & +  \sqrt{d} sM'' \Big( \E(\|u_n -\tu_n)\|_2^2) +  h^2 R^2 \sum^{i-1}_{j  = 1}\E(\|k_j - \tk_j\|_2^2)  \Big).
\end{split}
\end{equation}
By recursion, we obtain
\begin{equation}\label{eq: est_krk_1st}
\|\E (k_{i} - \tk_{i})\|_2 \le (1+h\sqrt{d}RM')^s \bigg[
 \sqrt{d} M' \|\E(u_n - \tu_n)\| + \sqrt{d}sM'' \Big( \E(\|u_n -\tu_n\|_2^2) +  h^2 R^2 \sum^{i-1}_{j  = 1}\E(\|k_j - \tk_j\|_2^2)  \Big) \bigg].
\end{equation}

We observe from the inequality above that the upper bound of $\|\E(k_{i} - \tk_{i})\|_2$ is partially determined by the second moment $\E(\|k_j - \tk_j\|_2^2)$ up to $(i-1)$-th Runge-Kutta stage. Therefore, we subsequently consider the estimate for $\E(\|k_j - \tk_j\|_2^2)$. By direct calculation, 
\begin{displaymath}
  \begin{split}
\E(\|k_{i} - \tk_{i}\|_2^2)
& =  \ \E\Big[ \ \Big\|f\big(t_n + c_{i} h, u_{n} + h \sum^{i-1}_{j = 1} a_{ij} k_j\big) - \frac{1}{N_s} \sum^{N_s}_{l = 1} g\big(t_n + c_{i} h, \tu_n + h \sum^{i-1}_{j = 1} a_{ij} \tk_j,X^{(i)}_l\big) \Big\|_2^2 \ \Big]\\
& \le \  2\E\Big[ \ \Big\|f\big(t_n + c_{i} h, u_{n} + h \sum^{i-1}_{j = 1} a_{ij} k_j\big) - \frac{1}{N_s} \sum_{l=1}^{N_s} g\big(t_n + c_{i} h,u_n + h \sum^{i-1}_{j = 1} a_{ij} k_j, X_l^{(i)} \big)\Big\|_2^2 \ \Big] \\
 &\quad +\frac{2}{N_s^2} \E\Big[ \ \Big\|\sum^{N_s}_{l = 1} \Big[ g\big(t_n + c_{i} h, u_n + h \sum^{i-1}_{j = 1} a_{ij} k_j,X^{(i)}_l\big) - g\big(t_n + c_{i} h, \tu_n + h \sum^{i-1}_{j = 1} a_{ij} \tk_j,X^{(i)}_l\big) \big] \Big\|_2^2 \ \Big]\\
& \le \ \frac{2}{N_s} \rkn^2 + 2M'^2 d \E \|w_i\|_2^2 \\
& \le \ 4ds{M'}^2 \Big(  \E\big(\|u_n -\tu_n\|_2^2\big) + R^2 h^2 \sum_{j=1}^{i-1} \E\big(\|k_j - \tk_j\|_2^2 \big)  \Big) + \frac{2}{N_s} \rkn^2.
   \end{split}
\end{displaymath}
Here we have used the mean value theorem and the standard error estimates for the Monte Carlo method to obtain the upper bound. 
By applying the above inequality recursively backwards to the first Runge-Kutta stage, we obtain a uniform bound for $\E(\|k_{i} - \tk_{i}\|_2^2)$:
\begin{equation} \label{eq: est_krk_2ed}
\E(\|k_{i} - \tk_{i}\|_2^2) \le (1+ 4 M'^2 R^2 d s h^2)^{s} \Big(   4 dsM'^2 \E\big(\|u_n -\tu_n\|_2^2\big)  + \frac{2}{N_s} \rkn^2   \Big).
\end{equation}
The estimation \eqref{eq: est_krk_2ed_global} can be obtained by setting $\beta' = 2^{s+1} \max(2 d s M'^2, 1)$ for $h \le 1/(2M'R\sqrt{ds})$. Substituting \eqref{eq: est_krk_2ed} into \eqref{eq: est_krk_1st}, we can obtain \eqref{eq: est_krk_1st_global} with $h \le 1/ \max(M'R\sqrt{d}, R\sqrt{s\beta'})$.
\end{proof}

\subsection{Proof of Proposition \ref{thm: diff recurrence relations} --- Part II: Recurrence relation for the numerical error}
\label{sec: diff recurrence 2}
We again insert the two schemes and expand the numerical error $\E\big(\|u_{n+1}-\tu_{n+1}\|_2^2\big)$ into
\begin{equation}\label{eq: est_2ed_order_twoparts}
   \begin{split}
\E\big(\|u_{n+1} -\tu_{n+1}\|_2^2\big) &= \E\Big[ \ \Big\|( u_n - \tu_n ) + h  \sum_{i=1}^s b_i (k_i - \tk_i )   \Big\|_2^2 \ \Big]  \\
&= \E\big(\|u_n - \tu_n\|_2^2\big) +h^2 \E\Big[ \ \Big\| \sum_{i=1}^s b_i (k_i - \tk_i )  \Big\|_2^2 \ \Big]  \\
&\hspace{20pt}+ h\E\Big[(u_n - \tu_n)^\dagger \sum_{i=1}^s b_i (k_i - \tk_i ) \Big]  + h\E\Big[ \Big(\sum_{i=1}^s b_i (k_i - \tk_i ) \Big)^\dagger (u_n - \tu_n)\Big].
    \end{split}
\end{equation} 
The second term on the right-hand side can be immediately estimated using the previous result \eqref{eq: est_krk_2ed_global} given in Lemma \ref{lemma: deltak and deltak^2}:
\begin{equation}\label{eq: est_2ed_order_term_1}
h^2 \E\Big[ \ \Big\| \sum_{i=1}^s b_i (k_i - \tk_i )  \Big\|_2^2 \ \Big] \le \ R^2 s h^2   \sum_{i=1}^s\E\big(\|  k_i - \tk_i  \|_2^2\big) \le \ R^2 s^2  \beta' \Big( h^2  \E\big(\|u_n -\tu_n\|_2^2\big)  + \frac{h^2}{N_s} \rkn^2  \Big).
\end{equation}
Using this estimate, naively we can bound the last two cross terms in \eqref{eq: est_2ed_order_twoparts} by Cauchy-Schwarz inequality. However, such a strategy will lead to an error estimate with the form
\begin{displaymath}
h\E\Big[(u_n - \tu_n)^\dagger \sum_{i=1}^s b_i (k_i - \tk_i ) \Big]  + h\E\Big[ \Big(\sum_{i=1}^s b_i (k_i - \tk_i ) \Big)^\dagger (u_n - \tu_n)\Big]
\le C \left(h \E(\|u_n - \tu_n\|_2^2) + \frac{h}{N_s} \rkn^2\right),
\end{displaymath}
where the last term is sub-optimal and will lead to a deterioration in the final error estimate. Therefore we need a more careful estimate as in the following lemma:
\begin{lemma}
 \label{lemma: 2ed error cross term}
Given a sufficiently small time step length $h$. If the boundedness assumptions \eqref{assump: bd} are satisfied, we have 
\begin{equation}\label{eq: est_2ed_order_term_2}
h\E\Big[(u_n - \tu_n)^\dagger \sum_{i=1}^s b_i (k_i - \tk_i ) \Big]  + h\E\Big[ \Big(\sum_{i=1}^s b_i (k_i - \tk_i ) \Big)^\dagger (u_n - \tu_n)\Big] \le  C_{\emph{cr}} \Bigl( \frac{\alpha'^2 h^5}{N_s^2} \rkn^4 + h   \E(\|u_n - \tu_n\|_2^2) \Bigr),
\end{equation}
where $C_{\text{cr}} = \max(R^2 s^2, 2+2^{s+2} {M'}^2 d R^2 s^3)$. Here we recall that $s$ is the number of Runge-Kutta stages, $d$ is the dimension of solution $u$, $R,M'$ are some upper bounds defined in the assumptions \eqref{assump: bd}--\eqref{assump: rk bd} and $\alpha'$ is given in Lemma \ref{lemma: deltak and deltak^2}.
\end{lemma}

With Lemma \ref{lemma: 2ed error cross term}, we now plug the estimates \eqref{eq: est_2ed_order_term_1} and \eqref{eq: est_2ed_order_term_2} into \eqref{eq: est_2ed_order_twoparts} to obtain the recurrence relation \eqref{eq: diff recurrence 1} for the numerical error by
\begin{equation}
\E(\|u_{n+1}-\tu_{n+1}\|_2^2) \le (1+C_{\text{cr}} h + R^2 s^2 \beta' h^2) \E(\|u_{n}-\tu_{n}\|_2^2) + \Big(C_{\text{cr}}\frac{\alpha'^2 h^5}{N_s^2} \rkn^4 + R^2 s^2 \beta'\frac{h^2}{N_s} \rkn^2 \Big),
\end{equation}
from which one can see that if $h$ and $N_s$ satisfy $h\le \frac{C_{\text{cr}}}{R^2 s^2 \beta'}$, then \eqref{eq: diff recurrence 2} holds for $\beta = \max(2C_{\text{cr}}, R^2 s^2 \beta')$ with $\beta'$ is given in Lemma \ref{lemma: deltak and deltak^2}.

The rest of this section devotes to the proof of Lemma \ref{lemma: 2ed error cross term}.
We introduce a ``semi-stochastic'' approximation $\bu_{n+1}$ defined by
\begin{equation} \label{eq: semi-stochastic}
\begin{aligned}
& \bk_i = f(t_n + c_i h, \tu_n + h \sum^{i-1}_{j = 1} a_{ij} \bk_j), \quad i = 1,\cdots,s; \\
& \bu_{n+1} = \tu_n + h \sum^s_{i=1} b_i \bk_i.
\end{aligned}
\end{equation}
This approximation applies the deterministic Runge-Kutta scheme to the stochastic solution $\tu_n$ for one time step. The following Lemma controls the difference between this local approximation and the stochastic scheme \eqref{def: scheme diff mc}.

\begin{lemma}
\label{thm: diff_tkvsbk}
Let $X_i :=\big(X^{(1)},X^{(2)},\cdots,X^{(i)}\big)$ be the collection of samples up to $i$th Runge-Kutta stage where each $X^{(j)} = (X^{(j)}_1,X^{(j)}_2,\cdots,X^{(j)}_{N_s})$, we have 
\begin{equation}
 \begin{split}
 &\|\E_{X_i}(\bk_{i} - \tk_{i})\|_2  \le  \ \alpha' \frac{h^2}{N_s} \rkn^2,
\end{split}
\end{equation}
where $\alpha'$ is given in \eqref{eq: est_krk_1st_global}.
\end{lemma}

The proof of this lemma is omitted since it is almost identical to the proof of Lemma \ref{lemma: deltak and deltak^2}. The first and second terms on the right-hand side of \eqref{eq: est_krk_1st_global} do not appear in the above result, since $\bk_{i}$ and $\tk_{i}$ are computed based on the same solution at the $n$th step. Below we provide the proof of Lemma \ref{lemma: 2ed error cross term}:

\begin{proof}[Proof of Lemma \ref{lemma: 2ed error cross term}]
It suffices to only focus on one factor $h\E\Big[(u_n - \tu_n)^\dagger \sum_{i=1}^s b_i (k_i - \tk_i ) \Big]$ since the other one is simply its conjugate transpose which can be controlled by exactly the same upper bound. We use $\bk_i$ as a bridge and split 
\begin{equation}\label{eq: est_2ed_order_crossterm}
    \begin{split}
\Big| h\E\Big[(u_n - \tu_n)^\dagger \sum_{i=1}^s b_i (k_i - \tk_i ) \Big] \Big| &
\le h\Big| \E\Big[(u_n - \tu_n)^\dagger \sum_{i=1}^s b_i (k_i - \bk_i ) \Big] \Big|  +  h \Big| \E\Big[(u_n - \tu_n)^\dagger \E_{X_s}\Big( \sum_{i=1}^s b_i (\bk_i - \tk_i )  \Big) \Big]  \Big|  \\
& \le h \E(\|u_n - \tu_n\|_2^2) + \frac{h}{2} \E\Big[ \ \Big\|\sum_{i=1}^s b_i (k_i - \bk_i )\Big\|_2^2 + \Big\|\E_{X_s}\Big( \sum_{i=1}^s b_i (\bk_i - \tk_i )  \Big)  \Big\|_2^2 \ \Big] \\
& \le h \E(\|u_n - \tu_n\|_2^2) + \frac{R^2 s h}{2} \sum_{i=1}^s \Big( \E\|k_i - \bk_i\|_2^2 + \E \|\E_{X_s} (\bk_i - \tk_i)\|_2^2 \Big).
\end{split}
\end{equation}
From the first line to the second line above, we have taken advantage of the fact that $\tu_n$ is independent from $X_s$ which is sampled at the $(n+1)^{\text{th}}$ time step when calculating $\tu_{n+1}$. The difference between $k_i$ and $\bk_i$ can be estimated in the same way as the derivation of \eqref{eq: est_krk_2ed}. The result is
\begin{equation}
 \E(\|k_i - \bk_i\|_2^2) \le  4 {M'}^2 s d ( 1 +4{M'}^2 d R^2 s h^2)^s \E(\|u_n - \tu_n\|_2^2) .
\end{equation}
Inserting the above inequality and the result of Lemma \ref{thm: diff_tkvsbk} into \eqref{eq: est_2ed_order_crossterm}, we get
\begin{displaymath}
\Big| h\E\Big[(u_n - \tu_n)^\dagger \sum_{i=1}^s b_i (k_i - \tk_i ) \Big] \Big|
\le h \left[1 + 2 R^2 s^3 M'^2 d ( 1 + 4{M'}^2 d R^2 s h^2)^s \right] \E(\|u_n - \tu_n\|_2^2)
  + \frac{R^2 s^2 \alpha'^2 h^5}{2 N_s^2} \rkn^4,
\end{displaymath}
from which one can easily observe that the lemma holds if $4M'^2 d R^2 s h^2 < 1$.
\end{proof}

\subsection{Proof of Theorem \ref{thm: diff bounds} --- error bounds}
In this section, we apply the two recurrence relations stated in Proposition \ref{thm: diff recurrence relations} to get the estimates for the bias $\|\E(u_{n+1} - \tu_{n+1})\|_2$ as well as the numerical error $\E(\|u_{n+1} - \tu_{n+1}\|_2^2)$.

By using \eqref{eq: diff recurrence 2} recursively backwards w.r.t $n$, we have 
\begin{equation}\label{eq: est_2ed_order_tk}
   \begin{split}
\E(\|u_{n+1}-\tu_{n+1}\|_2^2) & \le (1+\beta h )^{n+1}  \E(\|u_{0}-\tu_{0}\|_2^2) + \beta \left( \frac{h^2}{N_s} \rkn^2 + \frac{\alpha^2 h^5}{s^2 R^2 N_s^2} \rkn^4 \right) \sum_{i=0}^{n} (1+\beta h)^i \\
& = \Big( \frac{h}{N_s} \rkn^2 + \frac{\alpha^2 h^4}{s^2 R^2 N_s^2} \rkn^4 \Big) \big( e^{\beta t_{n+1}}-1\big)
   \end{split}
\end{equation}
which leads to the global estimate \eqref{diff 2ed error upper bound} for the bias. Inserting \eqref{eq: est_2ed_order_tk} into the recurrence relation \eqref{eq: diff recurrence 1} and expanding the recursion in a similar way, we get
\begin{equation}\label{upper bound diff}
  \begin{split}
&\|\E(u_{n+1}-\tu_{n+1})\|_2 \\
\le & \ \alpha \frac{h^3}{N_s} \rkn^2 \sum^n_{i=0}(1+\alpha h)^i + \alpha h \Big( \frac{h}{N_s} \rkn^2 + \frac{\alpha^2 h^4}{s^2 R^2 N_s^2} \rkn^4 \Big) \sum^{n-1}_{i=0} (1+\alpha h)^i \big(e^{\beta t_{n-i}} -1\big) \\
\le & \  \frac{h^2}{N_s}\big( e^{\alpha t_{n+1}}-1\big) \rkn^2 + \alpha h \Big( \frac{h}{N_s} \rkn^2 + \frac{\alpha^2 h^4}{s^2 R^2 N_s^2} \rkn^4 \Big)\sum^{n-1}_{i=0} e^{\alpha t_i}\big(e^{\beta t_{n-i}} -1\big) \\
\le & \  \frac{h^2}{N_s}\big( e^{\alpha t_{n+1}}-1\big) \rkn^2 + \alpha h \Big( \frac{h}{N_s} \rkn^2 + \frac{\alpha^2 h^4}{s^2 R^2 N_s^2} \rkn^4 \Big) \sum^{n-1}_{i=0} \big(e^{\max(\alpha,\beta) t_{n}} -1\big) \\
= & \  \frac{h^2}{N_s}\big( e^{\alpha t_{n+1}}-1\big) \rkn^2 + \alpha t_n \Big( \frac{h}{N_s} \rkn^2 + \frac{\alpha^2 h^4}{s^2 R^2 N_s^2} \rkn^4 \Big) \big(e^{\max(\alpha,\beta) t_{n}} -1\big),
   \end{split}
\end{equation}
which completes the proof of \eqref{diff 1st error upper bound}.



\section{Proofs of estimates for inchworm Monte Carlo method} \label{sec: proof}
In this section, the proofs of theorems in Section \ref{sec: inchworm results} are detailed. We will again first focus on the difference between the deterministic method and the stochastic method, and the error of the deterministic method will be discussed at the end of this section. Thanks to the previous discussion on the differential equation case, we may follow this framework which guides the general flow of our derivation. Below we point out the major differences as well as difficulties for the case of this integro-differential equation before the detailed proof:
\begin{itemize}
\item Since $K_2$ depends on more previously-computed time steps than $K_1$ due to the nonlocal integral term in \eqref{eq: inchworm equation} (this can be easily observed by comparing $\gb_{n,m}$ with $\gb^*_{n,m}$), a uniform expression for $K_i$ like \eqref{def: scheme diff rk} is no longer available for the integro-differential equation. Therefore, we need individual analysis for each $K_i$. 
\item Recall that the Taylor expansion is applied in the proof of Lemma \ref{lemma: deltak and deltak^2} (e.g. in \eqref{eq: k^1 expansion}), which requires to estimate the first- and second-order derivatives of the source term $f(t,u)$. This can no longer be simply assumed as in \eqref{assump: bd} and has to be carefully studied. They play crucial roles in understanding the behavior of the inchworm Monte Carlo method.
\item The derivation of the error amplification can no longer be handled by the simple discrete Gr\"onwall inequality, due to the involvement of a large number of previous steps on the right-hand side of the numerical scheme. The error estimation must be handled with more care, e.g. the estimation we used to handle the cross term in \eqref{eq: est_2ed_order_term_2} (Lemma \ref{lemma: 2ed error cross term}) will lead to a pessimistic (sub-optimal) fast growth rate in the error estimate of integro-differential equations. 
\item Most importantly, the magnitude of the derivatives depends on $\bar{M}$, as it is determined by the dimensionality of the integral in the equation. This will result in different error amplification with different choices of $\bar{M}$. This is the key point which explains whether/how the inchworm Monte Carlo method mitigates the numerical sign problem.
\end{itemize}

\subsection{Proof of Proposition \ref{thm: recurrence relations} --- Recurrence relation for the numerical error}
\label{sec: recurrence 2}
By the definitions of the deterministic method \eqref{def: scheme 1} and the inchworm Monte Carlo method \eqref{def: scheme 2}, it is straightforward to check that 
\begin{equation}\label{eq: formula first order error}
\dG_{n+1,m} =  A_{n,m}(h)\dG_{n,m} + \frac{1}{2}h\left(  B_{n,m}(h)\Delta K_1 + \Delta K_2\right),
\end{equation}
where for simplicity we have used the short-hands
\begin{equation*}
  \begin{split}
   & \Delta K_i = \tK_i - K_i, \\
   & A_{n,m}(h) = I + \frac{1}{2}\big(\sgn(t_n - t) + \sgn(t_{n+1} - t)\big)\ii H_s h - \frac{1}{2}\sgn(t_n - t)\sgn(t_{n+1} - t) H^2_s h^2,\\
   & B_{n,m}(h)  =  I + \sgn(t_{n+1} - t)\ii H_s h.
      \end{split}
\end{equation*}
By triangle inequality, the error can be bounded by 
\begin{equation}\label{eq: est second order error 1/2}
\left[\E(\|\dG_{n+1,m}\|^2)\right]^{1/2} \le \left[\E(\|A_{n,m}(h)\dG_{n,m})\|^2\right]^{1/2} + \frac{1}{2}h \left[\E\left(  \big\| B_{n,m}(h)\Delta K_1 + \Delta K_2 \big\|^2  \right) \right]^{1/2}.
\end{equation}
For the first term on the right-hand side, we have 
\begin{equation*}
   \begin{split}
&\E(\|A_{n,m}(h)\dG_{n,m})\|^2) \\
\le& \  \left[\rho\left(I+\ \frac{1}{2}\big(\sgn(t_n - t) + \sgn(t_{n+1} - t)\big)\ii H_s h - \frac{1}{2}\sgn(t_n - t)\sgn(t_{n+1} - t) H^2_s h^2\right)\right]^2 \cdot\E(\|\dG_{n,m})\|^2),
  \end{split}
\end{equation*}
where $\rho(\cdot)$ denotes the spectral radius of a matrix.
Let $\lambda_1$ and $\lambda_2$ be the two eigenvalues of $H_s$. Then
\begin{equation}\label{eq: est spectrum}
    \begin{split}
     &\ \left[\rho\left(I+\ \frac{1}{2}\big(\sgn(t_n - t) + \sgn(t_{n+1} - t)\big)\ii H_s h - \frac{1}{2}\sgn(t_n - t)\sgn(t_{n+1} - t) H^2_s h^2\right) \right]^2\\
    = & \ \max_{i=1,2} \left| 1+\frac{1}{2}\left(\sgn(t_n - t) + \sgn(t_{n+1} - t)\right)\ii \lambda_i h - \frac{1}{2}\sgn(t_n - t)\sgn(t_{n+1} - t) \lambda_i^2 h^2 \right|^2\\
    = & \ \max_{i=1,2} \left( 1+ \frac{1}{4}\left(\sgn(t_{n+1} - t) - \sgn(t_{n} - t)\right)^2 \lambda^2_i h^2 + \frac{1}{4}\left(\sgn(t_{n+1} - t)\sgn(t_{n} - t)\right)^2 \lambda^4_i h^4 \right)\\
    = & \ 1 + \frac{1}{4}\left(\rho(H_s)\right)^4 h^4 \le 1 + \frac{1}{4}\bdH^4 h^4.
    \end{split}
\end{equation}
Note that in the third line of the above equation, the second term vanishes due to the fact that the scheme evolves according to (R1)(R3). Consequently, the first term on the right-hand side of \eqref{eq: est second order error 1/2} can be estimated by
\begin{equation} \label{eq: AG}
    \left[\E(\|A_{n,m}(h)\dG_{n,m})\|^2)\right]^{1/2} \le \sqrt{ 1 + \frac{1}{4}\bdH^4 h^4 }\cdot\left[\E(\|\dG_{n,m})\|^2)\right]^{1/2} \le (1 + \frac{1}{8}\bdH^4 h^4)\cdot\left[\E(\|\dG_{n,m})\|^2)\right]^{1/2}.
\end{equation}

To estimate the second term on the right-hand side of \eqref{eq: est second order error 1/2}, we again need to bound $\left[\E(\|\tK_i - K_i\|^2)\right]^{1/2}$ to obtain a recurrence relation for the numerical error. Such results are given in the following lemma: 
\begin{lemma}\label{lemma: deltak and deltak^2 integ diff} 
Assume that the hypotheses (H1) and (H3) hold. For a sufficiently small time step length $h$, we have 
\begin{align}
\label{eq: bound k_1}
\left\|\E(\tK_1 - K_1)\right\| &\le 8P_1(t_{n-m}) h \sum_{i=1}^{n-m} \left( 2 + (n-m-i)h  \right)\left\| \E \left(\Delta \gb_{n,m}  \right) \right\|_{\Gamma_{n,m}(i)} + \bar{\alpha}(t_{n-m}) \left[ \Ns^{(\emph{std})}_{\Omega_{n,m}}(\Delta \gb_{n,m})\right]^2, \\
\label{eq: bound k_2}
    \begin{split}
\left\|\E(\tK_2 - K_2)\right\| &\le 28P_1(t_{n-m+1}) h  \sum_{i = 1}^{n-m} \left( 2 + (n-m+1-i)h  \right)\left\| \E \left( \Delta \gb^*_{n,m}  \right) \right\|_{\Gamma^*_{n,m}(i)} \\
&\hspace{100pt}+ 5\bar{\alpha}(t_{n-m+1}) \left[\Ns^{(\emph{std})}_{\bar{\Omega}_{n,m}}(\Delta \gb^*_{n,m})\right]^2 + 16 \bar{\alpha}(t_{n-m+1})\bar{\gamma}(t_{n-m+1})\cdot \frac{h^2}{N_s},
      \end{split}
\end{align}
and
\begin{align}
\label{eq: bound k_1^2}     
  \left[\E( \|\tK_1 - K_1\|^2 )\right]^{1/2} &\le  8P_1(t_{n-m}) h \sum_{i = 1}^{n-m} \left( 2 + (n-m-i)h  \right) \Ns^{(\emph{std})}_{\Gamma_{n,m}(i)}(\Delta \gb_{n,m})+ 2\sqrt{\bar{\gamma}(t_{n-m})}\cdot \frac{1}{\sqrt{N_s}}, \\
\label{eq: bound k_2^2}
  \begin{split}
        \left[\E(\|\tK_2 - K_2\|^2)\right]^{1/2} & \le 28P_1(t_{n-m+1}) h  \sum_{i = 1}^{n-m} \left( 2 + (n-m+1-i)h  \right) \Ns^{(\emph{std})}_{\Gamma^*_{n,m}(i)}(\Delta \gb^*_{n,m}) \\
        &\hspace{220pt}+ 3\sqrt{\bar{\gamma}(t_{n-m+1})}\cdot \frac{1}{\sqrt{N_s}},
        \end{split}
\end{align}
where $\bar{\alpha}$ and $\bar{\gamma}$ are defined in \eqref{eq: alpha gamma}.
Here we recall that $P_1(t),P_2(t)$ are given in Propositions \ref{thm: first order derivative} and \ref{thm: second order derivative}, and $\bdW,\bdG,\bdL$ are some upper bounds given in the assumptions (H1) and (H3).
\end{lemma}

We only write down the proof for \eqref{eq: bound k_1^2} and \eqref{eq: bound k_2^2} which are related to the numerical error in this section. The other two will only be used when estimating the bias so we put the corresponding proof in the Appendix.

\begin{proof}[Proof of \eqref{eq: bound k_1^2} and \eqref{eq: bound k_2^2}]

\textbf{(i)} Estimate of $\E(\|\tK_1 - K_1\|^2)$:

The definition of $\tK_i$ indicates that 
\begin{equation}\label{eq: relation F and G}
\begin{split}
&\E_{\vec{\sb}}[ \tilde{F}_1( \tg_{n,m};\vec{\sb}) ] = F_1( \tg_{n,m}),  \\
&\E_{\vec{\sb}'}[ \tilde{F}_2( \tg^*_{n,m};\vec{\sb}') ] = F_2( \tg^*_{n,m}) 
\end{split}
\end{equation}
by the fact that $\vec{\sb}$ and $\vec{\sb}'$ are sampled independently from $\tg_{n,m}$ and $\tg^*_{n,m}$. Therefore, for each $rs-$entry we have
\begin{equation}\label{eq: est k_1^2}
\begin{split}
\E_{\vec{\sb}}(|\tK^{(rs)}_1 - K^{(rs)}_1|^2) =& \ \E_{\vec{\sb}}\left[    \left| F^{(rs)}_1(\gb_{n,m})- \frac{1}{N_s}\sum_{i=1}^{N_s} \tilde{F}^{(rs)}_1( \tg_{n,m};\vec{\sb}^i) \right|^2  \right]     \\
=& \ \frac{1}{N_s} \E_{\vec{\sb}}\left[  \left|\tilde{F}^{(rs)}_1( \tg_{n,m};\vec{\sb})\right|^2 - \left|F^{(rs)}_1(\tg_{n,m})\right|^2  \right] + \left|F^{(rs)}_1(\gb_{n,m})-F^{(rs)}_1(\tg_{n,m})\right|^2 
\end{split}
\end{equation}
which gives 
\begin{equation} \label{eq: est k_1^2 sqrt}
    \begin{split}
   & \left[\E(|\tK^{(rs)}_1 - K^{(rs)}_1|^2)\right]^{1/2}  \le  \frac{1}{\sqrt{N_s}}\cdot \left[\E\left(  \left|\tilde{F}^{(rs)}_1( \tg_{n,m};\vec{\sb})\right|^2 - \left|F^{(rs)}_1(\tg_{n,m})\right|^2  \right) \right]^{1/2} \\
   & \hspace{160pt} \qquad \quad+ \left[\E \left(  \left|F^{(rs)}_1(\gb_{n,m})-F^{(rs)}_1(\tg_{n,m})\right|^2  \right) \right]^{1/2}.
     \end{split}
\end{equation}

According to the boundedness assumption (H1)(H3), the first term on the right-hand side of the inequality above is immediately bounded by
\begin{equation}\label{eq: est variance}
  \frac{1}{\sqrt{N_s}}\cdot \left[\E\left(  \left|\tilde{F}^{(rs)}_1( \tg_{n,m};\vec{\sb})\right|^2 - \left|F^{(rs)}_1(\tg_{n,m})\right|^2  \right) \right]^{1/2} \le \sqrt{\bar{\gamma}(t_{n-m})}\cdot\frac{1}{\sqrt{N_s}}.
\end{equation}
with $\bar{\gamma}(t)$ defined in \eqref{eq: alpha gamma}. For the second term, we use mean value theorem to get 
\begin{equation}\label{eq: est mean value}
   \begin{split}
&\left[\E \left(  \left|F^{(rs)}_1(\gb_{n,m})-F^{(rs)}_1(\tg_{n,m})\right|^2  \right) \right]^{1/2}=  \left[ \E \left(  \left|\left(  \nabla F_1^{(rs)}(\etab_{n,m})\right)^{\TT} \cdot(\htg_{n,m} - \hg_{n,m})\right|^2  \right) \right]^{1/2} \\
={} & \left[ \E \left(  \left|  \sum_{i=1}^{n-m}\quad \sum_{(k,\ell) \in \Gamma_{n,m}(i)}\quad \sum_{p,q=1,2} \frac{\partial F^{(rs)}_1(\etab_{n,m}) }{\partial G_{k,\ell}^{(pq)}} \cdot \Delta G_{k,\ell}^{(pq)} \right|^2  \right) \right]^{1/2}  \\
\le{} & \sum_{i=1}^{n-m}  \sum_{(k,\ell) \in \Gamma_{n,m}(i)}\quad \sum_{p,q=1,2}\left\{ \left[ \E\left( \left|\frac{\partial F^{(rs)}_1(\etab_{n,m}) }{\partial G_{k,\ell}^{(pq)}}\right|^2 \right) \right]^{1/2} \cdot \left[ \E\left( \left| \Delta G_{k,\ell}^{(pq)}\right|^2 \right) \right]^{1/2} \right\} \\
\le{} &    \sum_{i=1}^{n-m} \left\{ \left(  \sum_{(k,\ell) \in \Gamma_{n,m}(i)}\quad \sum_{p,q=1,2} \left[\E\left(\left|\frac{\partial F^{(rs)}_1(\etab_{n,m}) }{\partial G_{k,\ell}^{(pq)}}\right|^2 \right)\right]^{1/2} \right) \cdot \max_{(k,\ell)\in \Gamma_{n,m}(i); \atop p,q=1,2} \left[\E\left(\left| \Delta G_{k,\ell}^{(pq)}  \right|^2 \right)\right]^{1/2}  \right\} \\
\le{} & 4  P_1(t_{n-m}) \left[ h\Ns^{(\std)}_{\Gamma_{n,m}(n-m)}(\Delta \gb_{n,m}) + \sum_{i=1}^{n-m-1}(2h + (n-m-1-i)h^2)\Ns^{(\std)}_{\Gamma_{n,m}(i)}(\Delta \gb_{n,m}) \right]    \\
\le{} &  4P_1(t_{n-m}) h \sum_{i = 1}^{n-m} \left( 2 + (n-m-i)h  \right) \Ns^{(\std)}_{\Gamma_{n,m}(i)}(\Delta \gb_{n,m})  .
   \end{split}
\end{equation}
Here we have considered the derivatives of $F_1(\cdot)$ for different locations in $\Omega_{n,m}$. Also, we have applied Minkowski inequality in the first $``\le"$ and H\"older's inequality in the second $``\le"$. The estimate \eqref{eq: bound k_1^2} can then be obtained by inserting \eqref{eq: est variance} and \eqref{eq: est mean value} into \eqref{eq: est k_1^2 sqrt}.

\textbf{(ii)} Estimate of $\E(\|\tK_2 - K_2\|^2)$:

Similar to \eqref{eq: est k_1^2 sqrt}, we use the triangle inequality to bound $ \left[\E(|\tK^{(rs)}_2 - K^{(rs)}_2|^2)\right]^{1/2} $ by
\begin{equation} \label{eq: K2 diff}
   \begin{split}
   &  \left[\E(|\tK^{(rs)}_2 - K^{(rs)}_2|^2)\right]^{1/2} \le   \frac{1}{\sqrt{N_s}} \cdot \left[ \E\left(  \left|\tilde{F}^{(rs)}_2( \tg^*_{n,m};\vec{\sb}')\right|^2 - \left|F^{(rs)}_2(\tg^*_{n,m})\right|^2  \right) \right]^{1/2}\\
   & \hspace{200pt}+\left[ \E\left(  \left|F^{(rs)}_2(\gb^*_{n,m})-F^{(rs)}_2(\tg^*_{n,m})\right|^2   \right) \right]^{1/2}
       \end{split}
\end{equation}
where the first term on the right-hand side can be estimated similarly to \eqref{eq: est variance}, and the result is
\begin{equation} \label{eq: est variance 2}
  \frac{1}{\sqrt{N_s}}\cdot \left[\E\left(  \left|\tilde{F}^{(rs)}_2( \tg_{n,m}^*;\vec{\sb})\right|^2 - \left|F^{(rs)}_2(\tg_{n,m}^*)\right|^2  \right) \right]^{1/2} \le \sqrt{\bar{\gamma}(t_{n-m+1})}\cdot\frac{1}{\sqrt{N_s}}.
\end{equation}
For the second term on the right-had side of \eqref{eq: K2 diff}, we mimic the analysis in \eqref{eq: est mean value} to get 
\begin{equation}
   \begin{split}
&\left[\E\left(  \left|F^{(rs)}_2(\tg^*_{n,m})-F^{(rs)}_2(\gb^*_{n,m})\right|^2  \right)  \right]^{1/2}\\
\le & \  4P_1(t_{n-m+1}) h \left\{  \left[ \E \left(\|\tG^*_{n+1,m} - G^*_{n+1,m}\|^2\right) \right]^{1/2} + \sum_{i = 1}^{n-m} \left( 2 + (n-m+1-i)h  \right) \Ns^{(\std)}_{\Gamma^*_{n,m}(i)}(\Delta \gb^*_{n,m}) \right\}.
   \end{split}
\end{equation}
Here the difference between $\tG^*_{n+1,m}$ and $G^*_{n+1,m}$ can be estimated by
\begin{equation}
\begin{split}
    & \left[ \E \left(\|\tG^*_{n+1,m} - G^*_{n+1,m}\|^2\right) \right]^{1/2}  \\
     \le & \ \left[\E\left(  \left\|  \left(  I + \sgn(t_n - t)\ii H_s h \right) \Delta G_{n,m} \right\|^2\right)\right]^{1/2} + h\left[\E(\|\tK_1 - K_1\|^2)\right]^{1/2}\\
     \le &  \  (1+\frac{1}{2}\bdH^2 h^2 )\left[\E(\| \Delta G_{n,m}\|^2)\right]^{1/2} \\
     & \hspace{70pt}+  8P_1(t_{n-m}) h^2 \sum_{i = 1}^{n-m} \left( 2 + (n-m-i)h  \right) \Ns^{(\std)}_{\Gamma_{n,m}(i)}(\Delta \gb_{n,m})+ 2\sqrt{\bar{\gamma}(t_{n-m})}\cdot \frac{h}{\sqrt{N_s}},
   \end{split}
\end{equation}
where we have applied our previous estimate \eqref{eq: est k_1^2} to bound $\E(\|\tK_1 - K_1\|^2)$, and we have omitted the details of the estimation of $\E\left(  \left\|  \left(  I + \sgn(t_n - t)\ii H_s h \right) \Delta G_{n,m} \right\|^2\right)$, which is similar to \eqref{eq: est spectrum}.
\begin{equation} \label{eq: est mean value 2}
   \begin{split}
&\left[ \E\left(  \left|F^{(rs)}_2(\gb^*_{n,m})-F^{(rs)}_2(\tg^*_{n,m})\right|^2   \right) \right]^{1/2} \le    4P_1(t_{n-m+1}) h \Big\{   (1+\frac{1}{2}\bdH^2 h^2 )\Ns^{(\std)}_{\Gamma^*_{n,m}(n-m)}(\Delta \gb^*_{n,m}) \\
&\hspace{20pt}+  \left(1+8P_1(t_{n-m}) h^2\right)\sum_{i = 1}^{n-m} \left( 2 + (n-m+1-i)h  \right) \Ns^{(\std)}_{\Gamma^*_{n,m}(i)}(\Delta \gb^*_{n,m}) +  2\sqrt{\bar{\gamma}(t_{n-m})}\cdot \frac{h}{\sqrt{N_s}}  \Big\}.
   \end{split}
\end{equation}

Again, we insert the estimates \eqref{eq: est variance 2} and \eqref{eq: est mean value 2} into \eqref{eq: K2 diff} to get  
 \begin{equation}
    \begin{split}
&\left[\E( \|\tK_2 - K_2\|^2 )\right]^{1/2} \le  8\left(3+h + (\frac{1}{2}\bdH^2 + 16P_1(t_{n-m}))h^2 + 8P_1(t_{n-m})h^3\right) P_1(t_{n-m+1}) h \times    \\
&\hspace{10pt} \sum_{i = 1}^{n-m} \left( 2 + (n-m+1-i)h  \right) \Ns^{(\std)}_{\Gamma^*_{n,m}(i)}(\Delta \gb^*_{n,m})   +  \left(2+16P_1(t_{n-m+1}) h^2\right)\sqrt{\bar{\gamma}(t_{n-m+1})}\cdot \frac{1}{\sqrt{N_s}}.
   \end{split}
\end{equation}
By choosing a sufficiently small time step such that $h + (\frac{1}{2}\bdH^2 + 16P_1(t_{n-m}))h^2 + 8P_1(t_{n-m})h^3 \le \frac{1}{2}$ and $h \le \sqrt{\frac{1}{16P_1(t_{n-m+1})}}$, the estimate \eqref{eq: bound k_2^2} can be obtained. 
\end{proof}

With the results above, we now return to the formula \eqref{eq: est second order error 1/2} and give the recurrence relation for the numerical error $\left[\E(\|\dG_{n+1,m}\|^2)\right]^{1/2}$ as  
\begin{equation}\label{eq: recurrence 2 1/2 derivation}
   \begin{split}
&\left[\E(\|\dG_{n+1,m}\|^2)\right]^{1/2} \\
\le & \ (1+\frac{1}{8}\bdH^4 h^4)\left[\E(\|\dG_{n,m}\|^2)\right]^{1/2} + \frac{1}{2}(1+\frac{1}{2}\bdH^2 h^2)h\left[\E(\|\Delta K_1\|^2)\right]^{1/2} +  \frac{1}{2}h\left[\E(  \| \Delta K_2 \|^2  ) \right]^{1/2} \\
\le & \ (1+\frac{1}{8}\bdH^4 h^4)\left[\E(\|\dG_{n,m}\|^2)\right]^{1/2} \\
&+ 22 P_1(t_{n-m+1}) h^2  \sum_{i = 1}^{n-m} \left( 2 + (n-m+1-i)h  \right) \Ns^{(\std)}_{\Gamma^*_{n,m}(i)}(\Delta \gb^*_{n,m}) + \frac{7}{2}  \sqrt{\bar{\gamma}(t_{n-m+1})}\cdot \frac{h}{\sqrt{N_s}}
 \end{split}
\end{equation}
upon assuming $h \le \frac{\sqrt{2}}{\bdH}$.

Next, we consider the recurrence relation of $\E(\|\dG_{n+1,m}\|^2)$. By straightforward expansion,
\begin{equation}\label{eq: est second order error}
 \begin{split}
\E(\|\dG_{n+1,m}\|^2) =& \ \E\left[\left\|  A_{n,m}(h)\dG_{n,m} + \frac{1}{2}h\left(  B_{n,m}(h)\Delta K_1 + \Delta K_2 \right)\right\|^2\right]\\
=& \   \E(\|A_{n,m}(h)\dG_{n,m})\|^2)  + \underbrace{ \frac{1}{4}h^2\E\left[  \left\| B_{n,m}(h)\Delta K_1 + \Delta K_2\right\|^2  \right] }_{\text{quadratic term}} + \\
& + \underbrace{ \mathrm{Re} \, h\E\left[ \text{tr} \left( \left(B_{n,m}(h)\Delta K_1 + \Delta K_2\right)^\dagger  \left( A_{n,m}(h)\dG_{n,m} \right)  \right) \right]}_{\text{cross term}}.
 \end{split}
\end{equation}
To bound the quadratic term, we first derive the following results from the estimates \eqref{eq: bound k_1^2} and \eqref{eq: bound k_2^2}:
\begin{equation}
    \begin{split}
    &\E( \|\tK_1 - K_1\|^2 ) \le  128P^2_1(t_{n-m}) h^2 \left[ \sum_{i = 1}^{n-m} \left( 2 + (n-m-i)h  \right)\Ns^{(\std)}_{\Gamma_{n,m}(i)}(\Delta \gb_{n,m})  \right]^2+ 8\bar{\gamma}(t_{n-m})\cdot \frac{1}{N_s},\\
   &\E( \|\tK_2 - K_2\|^2 ) \le 1568P^2_1(t_{n-m+1}) h^2 \left[  \sum_{i = 1}^{n-m} \left( 2 + (n-m+1-i)h  \right) \Ns^{(\std)}_{\Gamma^*_{n,m}(i)}(\Delta \gb^*_{n,m}) \right]^2\\
   & \hspace{340pt}+ 18\bar{\gamma}(t_{n-m+1})\cdot \frac{1}{N_s}.
    \end{split}
\end{equation}
Then the quadratic term is bounded by 
\begin{equation}\label{eq: est quadratic term}
    \begin{split}
&\underbrace{\frac{1}{4}h^2\E\left[  \left\|B_{n,m}(h)\Delta K_1 + \Delta K_2\right\|^2  \right] }_{\text{quadratic term}} \le  \frac{1}{2}(1+ \bdH^2 h^2)h^2\E(\|\Delta K_1\|^2) +\frac{1}{2}h^2 \E(\|\Delta K_2\|^2) \\
\le & \ h^2\E(\|\Delta K_1\|^2) +\frac{1}{2}h^2 \E(\|\Delta K_2\|^2) \\
\le& \ 912 P^2_1(t_{n-m+1})^2 h^4 \left[  \sum_{i = 1}^{n-m} \big( 2 + (n-m+1-i)h  \big) \Ns^{(\std)}_{\Gamma^*_{n,m}(i)}(\Delta \gb^*_{n,m})  \right]^2+ 17\bar{\gamma}(t_{n-m+1})\cdot \frac{h^2}{N_s}.
  \end{split}
\end{equation}
thanks to the previous requirement on $h$ in \eqref{eq: recurrence 2 1/2 derivation} and the results \eqref{eq: bound k_1^2} and \eqref{eq: bound k_2^2} in Lemma \ref{lemma: deltak and deltak^2 integ diff} obtained in the previous section.

Similar to the proof of Lemma \ref{lemma: 2ed error cross term}, the estimation of the cross term in \eqref{eq: est second order error} is more subtle. We will again need some key estimates from the following local scheme for the inchworm equation:
\begin{equation}\label{def: scheme 3}
 \begin{split}
&\bG^*_{n+1,m} = (I+\sgn(t_n - t) \ii H_s h)\tG_{n,m} +  \bK_1 h,  \\
&\bG_{n+1,m} = (I + \frac{1}{2}\sgn(t_n - t)\ii H_s h  )\tG_{n,m} +\frac{1}{2}\sgn(t_{n+1} - t)\ii  H_s h  \bG^*_{n+1,m} +  \frac{1}{2}(\bK_1+\bK_2) h, \quad 0 \le m\le n\le 2N,
   \end{split}
\end{equation}
where 
\begin{equation} \label{def: scheme 3 notation}
 \begin{gathered}
\bK_1 = F_1(\bg_{n,m}),\qquad \bg_{n,m} = \tg_{n,m};\\
\bK_2 = F_2(\bg^*_{n,m}), \qquad \bg^*_{n,m} = (\tg_{n,m};\tG_{n+1,n},\cdots,\tG_{n+1,m+1},\bG^*_{n+1,m}).
  \end{gathered}
\end{equation}
These quantities are introduced as the counterpart of \eqref{eq: semi-stochastic}, which is a deterministic time step applied to the stochastic solutions. The following results are similar to Lemma \ref{thm: diff_tkvsbk} for the case of differential equations: 

\begin{lemma}\label{thm: ktilde vs kbar}
Given the time step length $h$ and the number of Monte Carlo samples at each step $N_s$, we have
\begin{align}
&\| \E_{\vec{\sb}}(\bK_1 - \tK_1) \|=  0, \label{thm: 1st moment runge kutta 1} \\
&\| \E_{\vec{\sb},\vec{\sb}'}(\bK_2 - \tK_2) \|\le 4 \bar{\alpha}(t_{n-m+1})\bar{\gamma}(t_{n-m+1})\cdot \frac{h^2}{N_s},\label{thm: 1st moment runge kutta 2} \\
&\left[\E( \|\bK_1 - K_1\|^2 )\right]^{1/2} \le 8P_1(t_{n-m})  h   \sum_{i = 1}^{n-m} \left( 2 + (n-m-i)h  \right) \Ns^{(\emph{std})}_{\Gamma_{n,m}(i)}(\Delta \gb_{n,m}) , \label{thm: 2ed moment runge kutta 1} \\
&\left[ \E( \|\bK_2 - K_2\|^2 )\right]^{1/2} \le 28P_1(t_{n-m+1}) h  \sum_{i = 1}^{n-m} \left( 2 + (n-m+1-i)h  \right) \Ns^{(\emph{std})}_{\Gamma^*_{n,m}(i)}(\Delta \gb^*_{n,m}) \label{thm: 2ed moment runge kutta 2}
\end{align}
where the formula of $\bar{\alpha}(t)$ is given in \eqref{eq: bound k_1}.
\end{lemma}
The rigorous proof of this lemma is omitted since it is almost identical to that of Lemma \ref{lemma: deltak and deltak^2 integ diff}.

Now we are ready to bound the cross-term in \eqref{eq: est second order error}. By the same treatment as the case of differential equations, we have
\begin{equation}\label{eq: 2ed error cross split}
  \begin{split}
 & \left|\mathrm{Re} \, h\E\left[ \text{tr} \left( \left(B_{n,m}(h)\Delta K_1 + \Delta K_2\right)^\dagger  \left( A_{n,m}(h)\dG_{n,m} \right)  \right) \right] \right| \\
   \le {} & \left| h\E\left[ \text{tr} \left( \left(  B_{n,m}(h) ( K_1 - \bK_1 )  + ( K_2 - \bK_2 ) \right)^\dagger  \left(A_{n,m}(h)\dG_{n,m}\right) \right)  \right] \right| \\
   &\hspace{80pt}+   \left| h\E\left[ \text{tr}\left( \left(  B_{n,m}(h) ( \bK_1 - \tK_1 )  + ( \bK_2 - \tK_2 ) \right)^\dagger  \left(A_{n,m}(h)\dG_{n,m}\right) \right)  \right] \right|\\
={}& \left| h\E\left[ \text{tr} \left( \left(  B_{n,m}(h) ( K_1 - \bK_1 )  + ( K_2 - \bK_2 ) \right)^\dagger  \left(A_{n,m}(h)\dG_{n,m}\right) \right) \right] \right| \\
&\hspace{50pt} +   \left| h\E\left\{ \text{tr} \left( \left[ \E_{\vec{\sb},\vec{\sb}'}   \left(  B_{n,m}(h) ( \bK_1 - \tK_1 )  + ( \bK_2 - \tK_2 )   \right) \right]^\dagger \left(A_{n,m}(h)\dG_{n,m}\right)  \right) \right\}\right| \\
 \le {}& h \left[\E\left(\|A_{n,m}(h) \dG_{n,m} \|^2\right)\right]^{1/2}\left\{   \left[ \E\left(   \left\|  B_{n,m}(h) ( K_1 - \bK_1 )  + ( K_2 - \bK_2 ) \right\|^2  \right) \right]^{1/2} \right.\\
 &\left. \hspace{100pt}  +\left[\E\left(  \left\|\E_{\vec{\sb},\vec{\sb}'}  \left( B_{n,m}(h) ( \bK_1 - \tK_1 )  + ( \bK_2 - \tK_2 ) \right)  \right\|^2  \right) \right]^{1/2}  \right\}
\end{split}
\end{equation}
where we have applied Cauchy-Schwarz inequality in the last step.

On the right-hand side of \eqref{eq: 2ed error cross split}, the term $\left[ \E\left(\|A_{n,m}(h) \dG_{n,m} \|^2\right)\right]^{1/2}$ has already been bounded in \eqref{eq: AG}. For the other term, we can find the bounds by Lemma \ref{thm: ktilde vs kbar} immediately: 
\begin{equation} 
    \begin{split}
& \left[ \E\left(   \big\|  B_{n,m}(h) ( K_1 - \bK_1 )  + ( K_2 - \bK_2 ) \big\|^2  \right) \right]^{1/2} \le   2  \left[\E\left(\|K_1 - \bK_1\|^2\right) \right]^{1/2} + \left[ \E\left(\|K_2 - \bK_2\|^2\right) \right]^{1/2}\\
& \hspace{140pt}\le  44 P_1(t_{n-m+1}) h  \sum_{i = 1}^{n-m} \left( 2 + (n-m+1-i)h  \right) \Ns^{(\std)}_{\Gamma^*_{n,m}(i)}(\Delta \gb^*_{n,m})  ,\\
& \left[\E\left(  \Big\|\E_{\vec{\sb},\vec{\sb}'}  \left( B_{n,m}(h) ( \bK_1 - \tK_1 )  + ( \bK_2 - \tK_2 ) \right)  \Big\|^2  \right) \right]^{1/2} \le  \Big\|\E_{\vec{\sb},\vec{\sb}'}  \left( B_{n,m}(h) ( \bK_1 - \tK_1 )  + ( \bK_2 - \tK_2 ) \right)  \Big\| \\
& \hspace{100pt}\le  2\big \|\E_{\vec{\sb},\vec{\sb}'}  ( \bK_1 - \tK_1) \big\|^2 + \big \|\E_{\vec{\sb},\vec{\sb}'}  ( \bK_2 - \tK_2) \big\|^2   \le     4 \bar{\alpha}(t_{n-m+1})\bar{\gamma}(t_{n-m+1})\cdot \frac{h^2}{N_s}.
   \end{split}
\end{equation}
Thus the final estimation of the cross term is
\begin{equation}\label{eq: est cross term}
   \begin{split}
     & \left|\mathrm{Re} \, h\E\left[ \text{tr} \left( \left(B_{n,m}(h)\Delta K_1 + \Delta K_2\right)^\dagger  \left( A_{n,m}(h)\dG_{n,m} \right)  \right) \right] \right| \le \left(1+ \frac{1}{8}\bdH^4 h^4\right)h \left[\E(\|\dG_{n,m}\|^2)\right]^{1/2} \times\\ 
     & \left\{  44 P_1(t_{n-m+1}) h  \sum_{i = 1}^{n-m} \big( 2 + (n-m+1-i)h  \big) \Ns^{(\std)}_{\Gamma^*_{n,m}(i)}(\Delta \gb^*_{n,m})  +  4 \bar{\alpha}(t_{n-m+1})\bar{\gamma}(t_{n-m+1})\cdot \frac{h^2}{N_s} \right\}.
          \end{split}
\end{equation}
Finally, we combine the estimate \eqref{eq: est quadratic term} for the quadratic term with \eqref{eq: est cross term} for the cross term to obtain the recurrence relation \eqref{eq: recurrence 2} for the numerical error.

\subsection{Proof of  \eqref{2ed error upper bound} in Theorem \ref{thm: bounds}---Estimation of the numerical error} 
\label{sec: estimates}
In this section we discuss how to apply the recurrence relations in Proposition \ref{thm: recurrence relations} to obtain the estimates in Theorem \ref{thm: bounds}. Here we only focus on the estimate for the numerical error which we are more interested in. For the bias, we refer the readers to \ref{sec: recurrence 1} for the detailed proof.

In Proposition \ref{thm: recurrence relations}, two recurrence relations are given, in which the first relation \eqref{eq: recurrence 2 1/2} is easier to analyze due to its linearity. For simplicity, we rewrite this estimate as
\begin{equation}\label{eq: recurrence 2 1/2 simple}
    \begin{split}
     &\Ns^{(\std)}_{\Gamma^*_{n,m}(j-m+1)}(\Delta \gb^*_{n,m})   \le   (1+c_1 h^4)\Ns^{(\std)}_{\Gamma^*_{n,m}(j-m)}(\Delta \gb^*_{n,m})\\
&\hspace{30pt}+ c_2 h^2  \sum_{i = 1}^{j-m} \big( 2 + (j-m+1-i)h  \big) \Ns^{(\std)}_{\Gamma^*_{n,m}(i)}(\Delta \gb^*_{n,m}) +c_3 \frac{h}{\sqrt{N_s}}, \qquad j = m, \cdots,n-1,
    \end{split}
\end{equation}
where we have introduced the notations $c_1 =\frac{1}{8}\bdH^4$, $c_2 =22 P_1(t_{n-m})$ and $c_3 = \frac{7}{2}  \sqrt{\bar{\gamma}(t_{n-m})}$ for simplicity. The inequality \eqref{eq: recurrence 2 1/2 simple} is obtained by taking the maximum on the diagonals, and using the fact that $\bar{\gamma} (t_{j+1-m}) \le \bar{\gamma}(t_{n-m})$.

This inequality shows that the recurrence relation of the error with two indices $n$ and $m$ can be simplified as a recurrence relation with only one index $j$. For each $j$, the quantity $\Ns^{(\std)}_{\Gamma^*_{n,m}(j-m)}(\Delta \gb^*_{n,m})$ denotes the maximum numerical error on the $(j-m)$th diagonal (see Figure \ref{fig:sets}). This can also be understood by an alternative order of computation: once all the propagators on the diagonals $\Gamma^*_{n,m}(i)$ for $i \le j-m$ are computed, the propagators on $\Gamma^*_{n,m}(i)$ can actually be computed in an arbitrary order, i.e., the computation of all the propagators on $\Gamma^*_{n,m}(i)$ are independent of each other. The derivation of \eqref{eq: recurrence 2 1/2 simple} is inspired by this observation, and this idea will also be used in the proof of Theorem \ref{thm: bounds} to be presented later in this section.

To study the growth of the error from \eqref{eq: recurrence 2 1/2 simple}, one can define a sequence $\{A_j\}$ with the following recurrence relation:
\begin{equation}\label{eq: rec num err 1}
    A_{j+1} = (1+3c_2 h^2)A_j + c_2h^2  \sum_{i = m+1}^{j-1} \big( 2 + (j+1-i)h  \big) A_i + c_3 \frac{h}{\sqrt{N_s}} \text{~for~} j = m,\cdots, n-1,
\end{equation}
and initial condition $A_m=0$. Then we have $ [\E(\|\dG_{j+1,m}\|^2)]^{1/2} \le  A_{j+1}$ if we require $\frac{c_1}{c_2} h^2 + h \le 1$. Increasing the index $j$ in \eqref{eq: rec num err 1} by one, we get
\begin{equation}\label{eq: rec num err 2}
    A_{j+2} = (1+3c_2 h^2)A_{j+1} + c_2h^2  \sum_{i = m+1}^{j} \big( 2 + (j+2-i)h  \big) A_i + c_3 \frac{h}{\sqrt{N_s}}.
\end{equation}
Subtracting \eqref{eq: rec num err 1} from \eqref{eq: rec num err 2} yields
\begin{equation} \label{eq: rec num err 3}
      A_{j+2} =  (2+ 3c_2h^2 ) A_{j+1} - (1+c_2 h^2 -2c_2 h^3)A_j + c_2 h^3 \sum^{j-1}_{i=m+1}A_i. 
\end{equation}
Similarly, we can reduce the index $j$ in \eqref{eq: rec num err 3} by $1$ and again subtract the two equations, so that a recurrence relation without summation can be derived: 
\begin{equation}
A_{j+2} - (3+ 3c_2 h^2) A_{j+1} + (3+4c_2 h^2  -2 c_2 h^3)A_j - (1+c_2 h^2 -c_2 h^3)A_{j-1} = 0.
\end{equation}
The general formula of $A_j$ can then be found by solving the corresponding characteristic equation. We denote $A_j$ as 
\begin{equation} \label{eq:A_j}
    A_{j} = \sigma_1 r^j_1 +  \sigma_2 r^j_2 + \sigma_3 r^j_3.
\end{equation}
The formula of each $r_i$ is give in \ref{sec: char poly}, based on which we can estimate $A_{n}$ by
\begin{equation}\label{eq: induction n step}
     A_n \le C(h,N_s) \cdot (1+ \theta_1 \sqrt{P_1(t_{n-m})} h)^{n-m},
     \end{equation}
where $\theta_1$ is a constant and $C(h,N_s)$ is a function to be determined.

The recurrence relation \eqref{eq: recurrence 2 1/2} helps determine the growth rate of the numerical error. However, if we use \eqref{eq: recurrence 2 1/2} to determine the function $C(h,N_s)$, we can only find $C(h,N_s) \propto \sqrt{1/N_s}$, whereas the desired result is $C(h,N_s) \propto \sqrt{h/N_s}$. To this end, the other recurrence relation \eqref{eq: recurrence 2} has to be utilized, as in the proof given below:

\begin{proof}[Proof of Theorem \ref{thm: bounds} (Numerical error)]
As mentioned previously, we only present the proof of \eqref{2ed error upper bound} in this section. We claim that the error satisfies
\begin{equation} \label{eq: error bound}
  \begin{split}
&\Ns^{(\std)}_{\Gamma^*_{n,m}(l+1)}(\Delta \gb^*_{n,m}) \le \theta_2 \sqrt{\bar{\gamma}(t_{n-m+1})}\cdot \sqrt{\frac{h}{N_s}} (1+ \theta_1 \sqrt{P_1(t_{n-m+1})} h)^l \\
& \text{for~}l = 0,1,\cdots, n-m.
  \end{split}
\end{equation}
Here $\theta_2$ is some $O(1)$ constant. We will prove this claim by mathematical induction using \eqref{eq: recurrence 2}. 

We first check the initial case $l = 0$. By the recurrence relation \eqref{eq: recurrence 2}, the left-hand side of \eqref{eq: error bound} is bounded by 
\begin{equation}\label{induction:initial step}
  \begin{split}
  \Ns^{(\std)}_{\Gamma^*_{n,m}(1)}(\Delta \gb^*_{n,m}) = & \ \max\left(  [ \E\big(\|\Delta G_{m+1,m}\|^2\big)]^{1/2}, [ \E\big(\|\Delta G_{m+2,m+1}\|^2\big)]^{1/2},\cdots, [ \E\big(\|\Delta G_{n+1,n}\|^2\big)]^{1/2} \right) \\
  \le & \  \sqrt{17} \sqrt{\bar{\gamma}(h)} \frac{h}{\sqrt{N_s}} \le \theta_2 \sqrt{\bar{\gamma}(t_{n-m+1})}\cdot \sqrt{\frac{h}{N_s}}, 
     \end{split}
\end{equation}
as holds for any constant $\theta_2 \geq \sqrt{17}$. 

Assume that \eqref{eq: error bound} holds for all $l = 0, 1, \cdots, k-1$, when $l=k$,
\begin{equation}\label{induction:parallel}
   \Ns^{(\std)}_{\Gamma^*_{n,m}(k+1)}(\Delta \gb^*_{n,m})= \max\left(  \left[ \E\big(\|\Delta G_{m+k+1,m}\|^2\big)\right]^{1/2}, \cdots, \left[ \E\big(\|\Delta G_{n+1,n-k}\|^2\big)\right]^{1/2}  \right).
\end{equation}
Therefore we just need to find the bounds for each $\left[ \E\big(\|\Delta G_{m+k+j,m+j-1}\|^2\big)\right]^{1/2}$, $j=1,\cdots,n+1-m-k$. This will be done by the recurrence relation \eqref{eq: recurrence 2}. For clearer presentation, we rewrite the equation \eqref{eq: recurrence 2} below only with subscripts replaced:
\begin{equation}\label{eq: recurrence induction}
  \begin{split}
  &  \E\big(\|\Delta G_{m+k+j,m+j-1}\|^2\big) \le \left(1+\frac{1}{4} \bdH^4 h^4\right) \cdot \E\big(\|\Delta G_{m+k+j-1,m+j-1}\|^2\big)   \\
 & +  \left(1+\frac{1}{8} \bdH^4 h^4\right)44 P_1(t_{n-m+1})h^2 \cdot \left[\E\big(\|\Delta G_{m+k+j-1,m+j-1}\|^2\big)\right]^{1/2}\\ 
 & \hspace{130pt}\times \sum_{i=1}^{k}\left(2+(k+1-i)h\right) \Ns^{(\std)}_{\Gamma^*_{m+k+j-1,m+j-1}(i)}(\Delta \gb^*_{n,m})\\
 & + \left(1+\frac{1}{8} \bdH^4 h^4\right)4 \bar{\alpha}(t_{n-m+1}) \bar{\gamma}(t_{n-m+1})\frac{h^3}{N_s} \cdot \left[\E\big(\|\Delta G_{m+k+j-1,m+j-1}\|^2\big)\right]^{1/2} \\
 & +\left\{912 P_1^2(t_{n-m+1}) h^4 \left[ \sum_{i=1}^{k}\left(2+(k+1-i)h\right) \Ns^{(\std)}_{\Gamma^*_{m+k+j-1,m+j-1}(i)}(\Delta \gb^*_{n,m}) \right]^2 + 17  \bar{\gamma}(t_{n-m+1}) \frac{h^2}{N_s}\right\}.
  \end{split}
\end{equation}
For the first and third terms on the right-hand side, we use the inductive hypothesis \eqref{eq: error bound} with $l = k-1$, which indicates that
\begin{equation}
    \E\big(\|\Delta G_{m+k+j-1,m+j-1}\|^2\big) \le \left[ \Ns^{(\std)}_{\Gamma^*_{n,m}(k)}(\Delta \gb^*_{n,m})\right]^2 \le \theta_2^2 \bar{\gamma}(t_{n-m+1})\cdot \frac{h}{N_s} (1+ \theta_1 \sqrt{P_1(t_{n-m+1})} h)^{k-1},
\end{equation}
to get
\begin{align}\label{eq: recurrence induction 1} 
  \begin{split}
& \left(1+\frac{1}{4} \bdH^4 h^4\right) \cdot \E\big( \|\Delta G_{m+k+j-1,m+j-1}\|^2\big)\\
& \hspace{30pt}\le \left(1+\frac{1}{4} \bdH^4 h^4\right) \theta_2^2 \bar{\gamma}(t_{n-m+1})\cdot \frac{h}{N_s} (1+ \theta_1 \sqrt{P_1(t_{n-m+1})} h)^{2k-2},
\end{split}\\\label{eq: recurrence induction 3}
\begin{split}
&\left(1+\frac{1}{8} \bdH^4 h^4\right)4 \bar{\alpha}(t_{n-m+1}) \bar{\gamma}(t_{n-m+1})\frac{h^3}{N_s} \cdot \left[\E\big( \|\Delta G_{m+k+j-1,m+j-1}\|^2\big)\right]^{1/2} \\
& \hspace{30pt} \le  \left(1+\frac{1}{8} \bdH^4 h^4\right)4 \bar{\alpha}(t_{n-m+1}) [\bar{\gamma}(t_{n-m+1})]^{3/2} \theta_2 \cdot  \sqrt{\frac{h^7}{N_s^3}} (1+ \theta_1 \sqrt{P_1(t_{n-m+1})} h)^{k-1}  \\
& \hspace{30pt} \le h\theta_2^2 \bar{\gamma}(t_{n-m+1})\frac{h}{N_s} (1+ \theta_1 \sqrt{P_1(t_{n-m+1})} h)^{2k-2},
    \end{split}
\end{align}
where we have assumed $\left(1+\frac{1}{8} \bdH^4 h^4\right) \sqrt{\frac{h^3}{N_s}} \le \frac{\theta_2 }{4\bar{\alpha}(t_{n-m+1})\sqrt{\bar{\gamma}(t_{n-m+1})}}$ to get the last ``$\le$'' in \eqref{eq: recurrence induction 3}. For the second and the fourth terms on the right-hand side of \eqref{eq: recurrence induction}, we first use the inductive hypothesis to get
\begin{equation}
    \begin{split}
  \Ns^{(\std)}_{\Gamma^*_{m+k+j-1,m+j-1}(i)}(\Delta \gb^*_{n,m}) \le & \ \Ns^{(\std)}_{\Gamma^*_{n,m}(i)}(\Delta \gb^*_{n,m})\\
  \le & \ \theta_2 \sqrt{\bar{\gamma}(t_{n-m+1})}\cdot \sqrt{\frac{h}{N_s}} (1+ \theta_1 \sqrt{P_1(t_{n-m+1})} h)^{i-1},\text{~for~}i=1,\cdots,k.
        \end{split}
\end{equation}
The next step is to insert the above estimation into \eqref{eq: recurrence induction}. To assist our further calculation, we let $c_4 =\theta_1 \sqrt{P_1(t_{n-m+1})}$ and perform the following calculation: 
\begin{equation}
  \begin{split}
  &\sum_{i = 1}^{k} \big( 2 + (k+1-i)h \big)\cdot (1+ c_4 h)^{i}\\
  = & \ 2\sum_{i = 1}^{k} (1+ c_4 h)^{i} + h  \sum_{i = 1}^{k}  (k+1-i)\cdot (1+ c_4 h)^{i} \\
  = & \ 2 \frac{(1+c_4 h)^{k+1} - (1+c_4 h)}{c_4 h} + h \frac{(1+c_4 h)^{k+2} - (1+c_4 h)^{2}  - (1+c_4 h)c_4 k h }{c_4^2 h^2} \\
  \le & \ 2 \frac{(1+c_4 h)^{k+1} }{c_4 h} +  2\frac{(1+c_4 h)^{k+1} }{c_4^2 h} = \frac{(1+c_4 h)^{k+1}}{h}\cdot 2\left(  \frac{1}{c_4}+\frac{1}{c_4^2}  \right),
  \end{split}
\end{equation}
where we have assumed $h \le \frac{1}{c_4}$. By this inequality, we see that when $h$ is sufficiently small, for any $\theta_1 > 1$, the second term on the right-hand side of \eqref{eq: recurrence induction} satisfies
\begin{equation}\label{eq: recurrence induction 2}
    \begin{split}
& \left(1+\frac{1}{8} \bdH^4 h^4\right)44 P_1(t_{n-m+1})h^2 \cdot \left[\E\big(\|\Delta G_{m+k+j-1,m+j-1}\|^2\big)\right]^{1/2}\\ 
 & \hspace{130pt}\times \sum_{i=1}^{k}\left(2+(k+1-i)h\right) \Ns^{(\std)}_{\Gamma^*_{m+k+j-1,m+j-1}(i)}(\Delta \gb^*_{n,m})\\ \le & \ \left(1+\frac{1}{8} \bdH^4 h^4\right)44 P_1(t_{n-m+1})\theta_2^2 \bar{\gamma}(t_{n-m+1}) \frac{h^3}{N_s} \sum_{i = 1}^{k} \big( 2 + (k+1-i)h \big)\cdot (1+ \theta_1 \sqrt{P_1(t_{n-m+1})} h)^{k+i-2} \\
 \le & \ 88\underbrace{\left(1+\frac{1}{8} \bdH^4 h^4\right)}_{\le 2} \underbrace{(1+\theta_1 \sqrt{P_1(t_{n-m+1})} h)}_{\le 2} \underbrace{\left(  \frac{\sqrt{P_1(t_{n-m+1})}}{\theta_1 } + \frac{1}{\theta^2_1}\right)}_{\le 2\sqrt{P_1(t_{n-m+1})}} \\
& \hspace{200pt}\times h\theta_2^2 \bar{\gamma}(t_{n-m+1})\frac{h}{N_s} (1+ \theta_1 \sqrt{P_1(t_{n-m+1})} h)^{2k-2} \\
\le& \  704 \sqrt{P_1(t_{n-m+1})} h\theta_2^2 \bar{\gamma}(t_{n-m+1})\frac{h}{N_s} (1+ \theta_1 \sqrt{P_1(t_{n-m+1})} h)^{2k-2},
    \end{split}
\end{equation}
and when $\theta_2 = \sqrt{34}$, the fourth term on the right-hand side of \eqref{eq: recurrence induction} satisfies
\begin{equation}\label{eq: recurrence induction 4}
    \begin{split}
&912 P_1^2(t_{n-m+1}) h^4 \left[ \sum_{i=1}^{k}\left(2+(k+1-i)h\right) \Ns^{(\std)}_{\Gamma^*_{m+k+j-1,m+j-1}(i)}(\Delta \gb^*_{n,m}) \right]^2 + 17  \bar{\gamma}(t_{n-m+1}) \frac{h^2}{N_s}\\
\le & \ 912 P_1^2(t_{n-m+1}) \theta_2^2 \bar{\gamma}(t_{n-m+1}) \frac{h^5}{N_s}  \left[ \sum_{i = 1}^{k} \left( 2 + (k+1-i)h  \right)  (1+ \theta_1 \sqrt{P_1(t_{n-m+1})} h)^{i-1} \right]^2 + 17  \bar{\gamma}(t_{n-m+1}) \frac{h^2}{N_s} \\
= & \ 3648\underbrace{(1+\theta_1 \sqrt{P_1(t_{n-m+1})} h)^2}_{\le 2}\left(  \frac{\sqrt{P_1(t_{n-m+1})}}{\theta_1 } + \frac{1}{\theta^2_1 }\right)^2 h \times  h\theta_2^2 \bar{\gamma}(t_{n-m+1})\frac{h}{N_s} (1+ \theta_1 \sqrt{P_1(t_{n-m+1})} h)^{2k-2}\\
&\hspace{130pt}+\underbrace{\frac{17}{\theta_2^2}}_{\le \frac{1}{2}}\times h\theta_2^2 \bar{\gamma}(t_{n-m+1})\frac{h}{N_s} (1+ \theta_1 \sqrt{P_1(t_{n-m+1})} h)^{2k-2} \\
\le & \ h\theta_2^2 \bar{\gamma}(t_{n-m+1})\frac{h}{N_s} (1+ \theta_1 \sqrt{P_1(t_{n-m+1})} h)^{2k-2}
    \end{split}
\end{equation}
by assuming $7296 h \left(  \frac{\sqrt{P_1(t_{n-m+1})}}{\theta_1 } + \frac{1}{\theta^2_1 }\right)^2 \le 1/2$.

Inserting \eqref{eq: recurrence induction 1}\eqref{eq: recurrence induction 2}\eqref{eq: recurrence induction 3}\eqref{eq: recurrence induction 4} into \eqref{eq: recurrence induction}, we get
\begin{equation}
   \begin{split}
 & \E\big(\|\Delta G_{m+k+j,m+j-1}\|^2\big) \\
 \le{} & \frac{1 + (704\sqrt{P_1(t_{n-m+1})}+2)h + \frac{1}{4}\bdH^4 h^4}{(1+ \theta_1 \sqrt{P_1(t_{n-m+1})} h)^2} \times  \theta_2^2 \bar{\gamma}(t_{n-m+1})\frac{h}{N_s} (1+ \theta_1 \sqrt{P_1(t_{n-m+1})} h)^{2k}\\
 \le{} & \theta_2^2 \bar{\gamma}(t_{n-m+1})\frac{h}{N_s} (1+ \theta_1 \sqrt{P_1(t_{n-m+1})} h)^{2k}
     \end{split}
\end{equation}
by setting $\theta_1 = 353$. Since this inequality holds for all $j = 1,\cdots,n+1-m-k$, by \eqref{induction:parallel}, we know that \eqref{eq: error bound} holds for $l = k$. By the principle of mathematical induction, we have completed the proof of \eqref{eq: error bound}.

Finally, we set $l=n-m$ in \eqref{eq: error bound} to get 
\begin{equation}
    [ \E\big(\|\Delta G_{n+1,m}\|^2\big)]^{1/2} = \Ns^{(\std)}_{\Gamma^*_{n,m}(n-m+1)}(\Delta \gb^*_{n,m}) \le  \theta_2 \sqrt{\bar{\gamma}(t_{n-m+1})}\cdot \sqrt{\frac{h}{N_s}} (1+ \theta_1 \sqrt{P_1(t_{n-m+1})} h)^{n-m},
\end{equation}
resulting in the final estimate \eqref{2ed error upper bound} for the numerical error.
\end{proof}

\begin{remark}
     Due to the jump conditions \eqref{eq: jump condition}, we actually need to multiply by the norm of the observable $\|O_s\|$ on the right-hand side of \eqref{eq: recurrence 2 1/2 simple} when crossing the discontinuities. However, we can always first consider the observable $O_s/\|O_s\|$, and then multiply the result by $\|O_s\|$. Therefore we can always assume $\|O_s\| = 1$ in this paper and thus our analysis is not affected.
\end{remark}

\subsection{Proof of Proposition \ref{thm: Runge Kutta error}---Estimation of the error for the deterministic method}
\label{sec: Runge Kutta error}
In this section, we consider the error $E_{n+1,m} = \Ge(t_{n+1},t_m) - G_{n+1,m}$ for the deterministic scheme \eqref{def: scheme 1}. By triangle inequality,
\begin{equation} \label{eq: part 1&2}
   \begin{split}
&\|E_{n+1,m}\| \le \underbrace{\left\|\Ge(t_{n+1},t_m) - \left[A_{n,m}(h)\Ge(t_{n},t_m) + \frac{1}{2}h  \left( B_{n,m}(h) F_1(\ge_{n,m}) + F_2(\ges_{n,m})   \right)   \right]\right\|}_{\text{Part 1}}\\
& \hspace{100pt}+ \underbrace{\left\|\left[A_{n,m}(h)\Ge(t_{n},t_m) + \frac{1}{2}h  \left( B_{n,m}(h) F_1(\ge_{n,m}) + F_2(\ges_{n,m})   \right)   \right] - G_{n+1,m} \right\|}_{\text{Part 2}}
   \end{split}
\end{equation}
where $\ge_{n,m}$ and $\ges_{n,m}$ are defined as
\begin{equation}\label{gbe}
    \begin{split}
    \ge_{n,m} &=  \left(\Ge(t_{m+1},t_m);\Ge(t_{m+2},t_{m+1}),\Ge(t_{m+2},t_m);\cdots;\Ge(t_n,t_{n-1}),\cdots,\Ge(t_n,t_m)\right), \\
\ges_{n,m} &= \left(\ge_{n,m};\Ge(t_{n+1},t_n),\cdots,\Ge(t_{n+1},t_{m+1}),\Ge(t_{n+1},t_m)\right),
 \end{split}
\end{equation}
which are similar to the definitions of $\gb_{n,m}$ and $\gb^*_{n,m}$.
We note that $\Ge$ on the discontinuities are again defined to be multiple-valued as 
 \begin{equation*}
\Ge(t_N,t_k) = \left(\Ge(t^+,t_k),\Ge(t^-,t_k) \right)\text{~and~}\Ge(t_j,t_N ) = \left(\Ge(t_j,t^+),\Ge(t_j,t^-)  \right) \\
  \end{equation*}
for $0\le k \le N-1$ and $N+1\le j \le 2N$. We further define $\eb_{n,m} = \ge_{n,m} - \gb_{n,m}$ and $\eb^*_{n,m} = \ges_{n,m} - \gb^*_{n,m}$, which will be used later. The estimation of the two parts in \eqref{eq: part 1&2} will be discussed in the following two subsections.

\subsubsection{Estimation of Part 1 in \eqref{eq: part 1&2}} 
We further split this part of the error by 
\begin{equation}\label{eq: conv_heuns_1}
  \begin{split}
&\Ge(t_{n+1},t_m) - \left[A_{n,m}(h)\Ge(t_{n},t_m) + \frac{1}{2}h  \left( B_{n,m}(h) F_1(\ge_{n,m}) + F_2(\ges_{n,m})   \right)   \right] \\
={} & \left(\Ge(t_{n+1},t_m) -A_{n,m}(h)\Ge(t_{n},t_m) \right) - \frac{1}{2}h \left( B_{n,m}(h)\Hs(t_n,\Ge,t_m) + \Hs(t_{n+1},\Ge,t_m)  \right) \\
&\hspace{50pt} +\frac{1}{2}h \left( B_{n,m}(h)\Hs(t_n,\Ge,t_m) + \Hs(t_{n+1},\Ge,t_m)  \right) - \frac{1}{2}h  \left( B_{n,m}(h) F_1(\ge_{n,m}) + F_2(\ges_{n,m})   \right)
   \end{split}
\end{equation}
where the definition of $\Hs$ is given in \eqref{eq: inchworm equation}.

Using Taylor expansion, we may easily obtain the bound for the first term of \eqref{eq: conv_heuns_1}:   
\begin{equation}\label{eq: trunc_err_1}
\left\|\Ge(t_{n+1},t_m) -A_{n,m}(h)\Ge(t_{n},t_m) - \frac{h}{2} \left[ B_{n,m}(h)\Hs(t_n,\Ge,t_m) + \Hs(t_{n+1},\Ge,t_m)  \right] \right\| \le (\frac{1}{4} \bdH \bdG''+ \frac{5}{12} \bdG''' )h^3.
\end{equation}
Meanwhile, since 
\begin{equation}\label{eq: interp_error}
  \begin{split}
&\Hs(t_n,\Ge,t_m) - F_1(\ge_{n,m}) \\
= & \ \sgn(t_n-t)\sum^{\bar{M}}_{\substack{M=1 \\ M \text{~is odd~}}} \ii^{M+1} \int_{t_n > \vec{\sb} > t_m } (-1)^{\#\{\vec{\sb} < t\}}\\
&\hspace{20pt}W_s\left[\Ge(t_n,s_M)W_s\cdots W_s\Ge(s_1,t_m)-I_h\Ge(t_n,s_M)W_s\cdots W_s I_h \Ge(s_1,t_m) \right]\Ls(t_n,\vec{\sb})\dd \vec{\sb},
   \end{split}
\end{equation}
we need the following estimation to bound the above term:
\begin{equation} \label{eq: integrand estimation}
 \begin{split}
& \left\| \Ge(t_n,s_M)W_s\cdots W_s\Ge(s_1,t_m)-I_h\Ge(t_n,s_M)W_s\cdots W_s I_h \Ge(s_1,t_m) \right\|  \\
\le & \ \sum_{j = 0}^M \left\| \Ge(t_n,s_M)W_s \cdots W_s \Ge(s_{j+2},s_{j+1})\right\| \cdot \bdW \cdot  \left\|\Ge(s_{j+1},s_{j}) - I_h \Ge(s_{j+1},s_{j})\right\| \cdot \bdW \\
& \hspace{160pt}\cdot \left\| I_h\Ge(s_{j},s_{j-1})W_s \cdots W_s I_h\Ge(s_{1},t_m)\right\|.
  \end{split}
\end{equation}
The term $\left\|\Ge(s_{j+1},s_{j}) - I_h \Ge(s_{j+1},s_{j})\right\|$ in the above equation is the linear interpolation error. By the standard error estimation of interpolation (see e.g. \cite{Lin_Interp_Err_Simpl}), if the point $(s_{j+1},s_j)$ locates inside the triangle $T'$, the interpolation error is
\begin{equation}
\left|\Ge^{(rs)}(s_{j+1},s_j) - I_h \Ge^{(rs)}(s_{j+1},s_j)\right| \le  \frac{1}{2}R_{T'}^2  \left\| \rho\left(D^2 \Ge^{(rs)}\right)  \right\|_{L^\infty(T')},
\end{equation}
where $R_{T'}$ is the radius of the circumscribed circle of $T'$ and $\rho(D^2 \Ge^{(rs)}\big)$ denotes the spectral radius of the Hessian of $\Ge^{(rs)}$, which can be bounded by
\begin{equation*}   
\rho \left(D^2 \Ge^{(rs)}\right) \le  \|\nabla^2 \Ge^{(rs)}\| \le  2 \bdG''.
\end{equation*}
Plugging the above result into \eqref{eq: integrand estimation} and using the bounds of the exact and numerical solutions, we get
\begin{equation}
\left\| \Ge(t_n,s_M)W_s\cdots W_s\Ge(s_1,t_m)-I_h\Ge(t_n,s_M)W_s\cdots W_sI_h \Ge(s_1,t_m) \right\| \le (M+1)\bdG^M \bdW^M \bdG'' h^2.
\end{equation}
Finally we apply this bound to \eqref{eq: interp_error} to obtain
\begin{equation}\label{eq: trunc_err_2_1}
\left\|\Hs(t_n,\Ge,t_m) - F_1(\ge_{n,m})\right\| \le  \bar{\beta}(t_{n-m}) h^2,
\end{equation}
where 
\begin{equation*}
  \bar{\beta}(t) =   \bdW \bdG'' \bdL^{1/2} \left( \sum^{\bar{M}}_{\substack{M=1 \\ M \text{~is odd~}}} \frac{M+1}{(M-1)!!}(\bdW \bdG \bdL^{1/2} t)^M \right).
\end{equation*}

Note that $\Hs(t_{n+1},\Ge,t_m) - F_2(\ges_{n,m})$ can be obtained by changing $t_n$ to $t_{n+1}$ in the expression of $\Hs(t_n,\Ge,t_m) - F_1(\ge_{n,m})$. Therefore its bound can be given by $\bar{\beta}(t_{n-m+1}) h^2$. Thus,
\begin{equation}\label{eq: trunc_err_2}
\left\|\frac{1}{2}h \left( \Hs(t_n,\Ge,t_m) + \Hs(t_{n+1},\Ge,t_m)  \right) - \frac{1}{2}h  \left(  F_1(\ge_{n,m}) + F_2(\ges_{n,m})   \right)\right\| 
  \le\bar{\beta}(t_{n-m+1}) h^3.
 \end{equation}
Inserting the estimates \eqref{eq: trunc_err_1} \eqref{eq: trunc_err_2} into \eqref{eq: conv_heuns_1}, we have the following estimation for Part 1 of \eqref{eq: part 1&2}:
\begin{equation} \label{eq: part 1 final}
  \begin{split}
&\underbrace{\left\| \Ge(t_{n+1},t_m) - \left[A_{n,m}(h)\Ge(t_{n},t_m) + \frac{1}{2}h  \left( B_{n,m}(h) F_1(\ge_{n,m}) + F_2(\ges_{n,m})   \right)   \right] \right\|}_{\text{Part 1}} \\
 & \hspace{200pt} \le \left(\frac{1}{4} \bdH \bdG''+ \frac{5}{12} \bdG'''  + \bar{\beta}(t_{n-m+1}) \right) h^3.
  \end{split}
\end{equation}

\subsubsection{Estimation of Part 2 in \eqref{eq: part 1&2}} 
The estimation for this part of the error is similar to that of the stochastic error. By the numerical scheme \eqref{def: scheme 1}, we have
\begin{equation}\label{_ccumulate_err_split}
  \begin{split}
&\left[A_{n,m}(h)\Ge(t_{n},t_m) + \frac{1}{2}h  \left( B_{n,m}(h) F_1(\ge_{n,m}) + F_2(\ges_{n,m})   \right)   \right] - G_{n+1,m}  \\
= & \ A_{n,m}(h) E_{n,m} + \frac{1}{2}h B_{n,m}(h)\left( F_1(\ge_{n,m}) - F_1(\gb_{n,m})\right) + \frac{1}{2}h \left( F_2(\ges_{n,m}) - F_2(\gb^*_{n,m})\right),
   \end{split}
\end{equation}
For the second term on the right-hand side, we can mimic the analysis of \eqref{eq: est mean value} to get 
\begin{equation}\label{accumulate_err_1}
\left\| F_1(\ge_{n,m}) - F_1(\gb_{n,m})\right\|  \le  8P_1(t_{n-m}) h \sum_{i = 1}^{n-m} \left( 2 + (n-m-i)h  \right) \|\eb_{n,m}\|_{\Gamma_{n,m}(i)}.
\end{equation}
For the third term, due to the same analysis, we have
\begin{equation} \label{eq: third term}
    \begin{split}
&\left\| F_2(\ges_{n,m}) - F_2(\gb^*_{n,m})\right\|\\
\le{} & 8P_1(t_{n-m+1}) h \left\{ \left\|\Ge(t_{n+1},t_m) - G^*_{n+1,m} \right\| + \sum_{i = 1}^{n-m} \left( 2 + (n-m+1-i)h  \right) \|\eb^*_{n,m}\|_{\Gamma^*_{n,m}(i)} \right\} 
 \end{split}
\end{equation}
where 
\begin{equation}
  \begin{split}
&\left\| \Ge(t_{n+1},t_m) - G^*_{n+1,m}  \right\|\\
={} & \left\|\left[\Ge(t_n,t_m) + h \frac{\partial}{\partial \sa}\Ge(t_n,t_m) + \frac{1}{2}h^2  \frac{\partial^2}{\partial \sa^2}\Ge(\nu_n,t_m) \right]- \left[  \left( I + \sgn(t_n-t)\ii H_s h\right)G_{n,m} + h F_1(\gb_{n,m}) \right] \right\|\\
={} & \left\|\left( I + \sgn(t_n-t)\ii H_s h\right) E_{n,m}  + h \left[  \Hs(t_n,\Ge,t_m) - F_1(\gb_{n,m}) \right] + \frac{1}{2}h^2  \frac{\partial^2}{\partial \sa^2}\Ge(\nu_n,t_m) \right\| \\
\le{} & \left\| \left( I + \sgn(t_n-t)\ii H_s h\right)  E_{n,m}\right\|  + h \left\|  \Hs(t_n,\Ge,t_m) - F_1(\ge_{n,m}) \right\| \\
& \hspace{150pt}+  h \left\|  F_1(\ge_{n,m}) - F_1(\gb_{n,m}) \right\| + \left\|\frac{1}{2}h^2  \frac{\partial^2}{\partial \sa^2}\Ge(\nu_n,t_m)\right\| .
    \end{split}
\end{equation}
Note that the second-order derivative above is the remainder of the Taylor expansion and should be interpreted as 
\begin{equation*}
    \frac{\partial^2}{\partial \sa^2}\Ge(\nu_{n},t_m) = \begin{pmatrix}
     \frac{\partial^2}{\partial \sa^2}\Ge^{(11)}(\nu^{(11)}_{n},t_m) &  \frac{\partial^2}{\partial \sa^2}\Ge^{(12)}(\nu^{(12)}_{n},t_m) \\
      \frac{\partial^2}{\partial \sa^2}\Ge^{(21)}(\nu^{(21)}_{n},t_m) &  \frac{\partial^2}{\partial \sa^2}\Ge^{(22)}(\nu^{(22)}_{n},t_m)
  \end{pmatrix}.
\end{equation*}

Using the previous results \eqref{eq: trunc_err_2_1} and \eqref{accumulate_err_1}, we have the bound 
\begin{equation}
  \begin{split}
 & \left\|\Ge(t_{n+1},t_m) - G^*_{n+1,m}\right\|\\
 \le{} & (1+\frac{1}{2}\bdH^2 h^2)\|E_{n,m}\| +  8P_1(t_{n-m}) h^2 \sum_{i = 1}^{n-m} \left( 2 + (n-m-i)h  \right) \|\eb_{n,m}\|_{\Gamma_{n,m}(i)}+\left(\bdG''h^2+\bar{\beta}(t_{n-m}) h^3\right)\\
 \le{} & (1+\frac{1}{2}\bdH^2 h^2)\|E_{n,m}\| +  8P_1(t_{n-m}) h^2 \sum_{i = 1}^{n-m} \left( 2 + (n-m-i)h  \right) \|\eb_{n,m}\|_{\Gamma_{n,m}(i)}+2\bdG''h^2
    \end{split}
\end{equation}
upon assuming $h \le \bdG''/\bar{\beta}(t_{n-m})$. Plugging the above estimate into \eqref{eq: third term}, we obtain
\begin{equation}\label{accumulate_err_2}
\left\| F_2(\ges_{n,m}) - F_2(\gb^*_{n,m})\right\| \le  28P_1(t_{n-m+1}) h \sum_{i = 1}^{n-m} \left( 2 + (n-m+1-i)h  \right) \|\eb^*_{n,m}\|_{\Gamma^*_{n,m}(i)}  + 16P_1(t_{n-m+1})\bdG'' h^3 
\end{equation}
which is a result similar to \eqref{eq: bound k_2^2}.

Now we plug the estimates \eqref{accumulate_err_1} and \eqref{accumulate_err_2} into \eqref{accumulate_err_split}, we obtain the following estimation for Part 2 of $E_{n+1,m}$:
\begin{equation} \label{eq: part 2 final}
  \begin{split}
&\underbrace{\left\|\left[A_{n,m}(h)\Ge(t_{n},t_m) + \frac{1}{2}h  \left( B_{n,m}(h) F_1(\ge_{n,m}) + F_2(\ges_{n,m})   \right)   \right] - G_{n+1,m} \right\|}_{\text{Part 2}} \\
\le {}& \left(1+\frac{1}{8}\bdH^4 h^4\right) \|E_{n,m}\| + 22P_1(t_{n-m+1}) h^2 \sum_{i = 1}^{n-m} \left( 2 + (n-m+1-i)h  \right) \|\eb^*_{n,m}\|_{\Gamma^*_{n,m}(i)} + 8P_1(t_{n-m+1})\bdG'' h^3.
 \end{split}
\end{equation}
By now, we can combine the estimates for Part 1 and \eqref{eq: part 1 final} Part 2 \eqref{eq: part 2 final}, so that the estimation \eqref{eq: part 1&2} yields the recurrence relation:
\begin{equation}\label{eq: Runge Kutta error}
  \begin{split}
&\|E_{n+1,m}\| \\
\le{} & \left(1+\frac{1}{8}\bdH^4 h^4 \right)\|E_{n,m}\| + 22P_1(t_{n-m+1}) h^2 \sum_{i = 1}^{n-m} \left( 2 + (n-m+1-i)h  \right) \|\eb^*_{n,m}\|_{\Gamma^*_{n,m}(i)} + P^{\text{e}}(t_{n-m+1})\cdot h^3,
  \end{split}
\end{equation}
where 
\begin{equation*}
     P^{\text{e}}(t) = \left( \frac{1}{4} \bdH + 8P_1(t) \right)\bdG''+ \frac{5}{12} \bdG'''  + \bar{\beta}(t) .
\end{equation*}

One may compare the recurrence relation above with \eqref{eq: recurrence 2 1/2} to find that the only difference between these two inequalities is the truncation error (last term). Therefore, we may simply replicate the procedures \eqref{eq: recurrence 2 1/2 simple}--\eqref{eq: induction n step} and conclude that the deterministic error has the exact same growth rate in the exponential part as the numerical error:
\begin{equation}\label{recurrence 3} 
    \|E_{n+1,m}\| \le  P^{\text{e}}(t_{n-m+1})  (1+\theta_1 \sqrt{P_1(t_{n-m+1})}h)^{n-m+1}\cdot h^2
\end{equation}
which can again be verified by mathematical induction. Therefore, we arrive at the estimate for the deterministic error at stated in Proposition \ref{thm: Runge Kutta error}.

\section{Conclusion}
\label{sec: conclusion}
We have presented a detailed analysis to study the error growth in the inchworm Monte Carlo method, emphasizing the trade-off between the numerical sign problem and error growth due to accumulation and amplification due to time marching. The result explains why the inchworm Monte Carlo method has a slower error growth than the classical quantum Monte Carlo method, and our analysis reveals how partial resummation trades-off the numerical sign problem and the error amplification. Our work points to the research direction of improving the time integrator to further suppress the error growth, which will be considered in future works.

\section*{Acknowledgement}
Zhenning Cai was supported by the Academic Research Fund of the Ministry of Education of Singapore under grant No. R-146-000-291-114. The work of Jianfeng Lu was supported in part by the National Science Foundation via grants DMS-1454939 and DMS-2012286.

\appendix 

\section{Formulas of the roots of characteristic polynomial} \label{sec: char poly}
Here we provide the formulas for $r_i$ appearing in \eqref{eq:A_j}. Let $\epsilon = c_2 h^2$. Then we have
\begin{equation} \label{eq: r_i}
   \begin{split}
   &r_1 = 1+\epsilon+\frac{\left(\frac{2}{3}\right)^{1 / 3} \epsilon(2+2 h+3 \epsilon)}{R}+\frac{1}{2^{1 / 3} \times 3^{2 / 3}}R,\\
&r_2 = 1+\epsilon-\frac{(1+ \sqrt{3}\ii) \epsilon (2+2 h+3 \epsilon)} {2^{2 / 3} \times 3^{1 / 3}R}+\frac{\mathrm{i}(\mathrm{i}+\sqrt{3})}{2 \times 2^{1 / 3} \times 3^{2 / 3}}R, \\
&r_3 = 1+\epsilon+\frac{\ii(\ii+ \sqrt{3}) \epsilon (2+2 h+3 \epsilon)} {2^{2 / 3} \times 3^{1 / 3}R}+\frac{\mathrm{i}(\mathrm{i}+\sqrt{3})}{2 \times 2^{1 / 3} \times 3^{2 / 3}} R
   \end{split}
\end{equation}
where 
\begin{equation}
    R = \left(\epsilon\left(9 h+18 \epsilon+18 h \epsilon+18 \epsilon^{2}+\sqrt{3} \sqrt{-4 \epsilon(2+2 h+3 \epsilon)^{3}+27(h+2 h \epsilon+2 \epsilon(1+\epsilon))^{2}}\right)\right)^{1 / 3}.
\end{equation}
When $h$ is small, it can be verified that $R \sim O(h)$. One can then see from \eqref{eq: r_i} that $r_i \sim O(h)$ for $i = 1,2,3$.

\section{Proofs related to the bias estimation of the inchworm Monte Carlo method}
\label{sec: recurrence 1}
In this appendix, we would like to complete the proof of the bias estimation for the inchworm Monte Carlo method. Specifically, the proofs of \eqref{eq: bound k_1}\eqref{eq: bound k_2} in Lemma \ref{lemma: deltak and deltak^2 integ diff} and the proof of \eqref{1st error upper bound} in Theorem \ref{thm: bounds} will be given below. The final result \eqref{final 1st error upper bound} can be obtained directly by the triangle inequality.

\begin{proof}[Proof of \eqref{eq: bound k_1} and \eqref{eq: bound k_2}]

\textbf{(i)} Estimate of $\left\|\E(\tK_1 - K_1)\right\|$:

We again use the relation \eqref{eq: relation F and G} to get
\begin{equation}
\E_{\vec{\sb}}(\tK_1 - K_1) = F_1 (\tg_{n,m}) - F_1(\gb_{n,m}).
\end{equation}
Then by Taylor expansion, we get 
\begin{equation}\label{eq: est k1 expansion}
\begin{split}
&\E\left(\tK_1^{(rs)} - K_1^{(rs)}\right)  
= \left(  \nabla F_1^{(rs)}(\gb_{n,m})\right)^{\TT}\cdot \E(\htg_{n,m} - \hg_{n,m}) + \\
&\hspace{140pt}\frac{1}{2}\E\left[(\htg_{n,m} - \hg_{n,m})^{\TT}\left(\nabla^2 F_1^{(rs)}(\xib_{n,m})\right) (\htg_{n,m} - \hg_{n,m}) \right]
\end{split}
\end{equation}
where $\xib_{n,m}$ is a convex combination of $\tg_{n,m}$ and $\gb_{n,m}$. 

The estimate for the first term on the right-hand side of the equation above is similar to \eqref{eq: est mean value} which is  
\begin{equation}\label{eq: est k1 Taylor 1}
    \left| \left(  \nabla F_1^{(rs)}(\gb_{n,m})\right)^{\TT}\cdot \E(\htg_{n,m} -  \hg_{n,m}) \right|  \le 4P_1(t_{n-m}) h \sum_{i = 1}^{n-m} \left( 2 + (n-m-i)h  \right)\left\| \E \left( \Delta \gb_{n,m}  \right) \right\|_{\Gamma_{n,m}(i)}. 
\end{equation}
For the second term on the right-hand side of \eqref{eq: est k1 expansion}, we use Proposition \ref{thm: second order derivative} to bound it by 
\begin{equation}\label{eq: est k1 Taylor 2}
  \begin{split}
&\left|\E\left[\left(\htg_{n,m} - \hg_{n,m}\right)^{\TT}\left(\nabla^2 F_1^{(rs)}(\xib_{n,m})\right) \left(\htg_{n,m} - \hg_{n,m}\right)\right] \right|  \\
={} & \left|\E \left(\sum_{(k_1,\ell_1) \in \Omega_{n,m}; \atop (k_2,\ell_2) \in \Omega_{n,m}} \sum_{p_1,q_1 =1,2; \atop p_2,q_2 =1 ,2}  \frac{\partial^2 F^{(rs)}_1(\xib_{n,m})}{\partial G_{k_1,\ell_1}^{(p_1 q_1)}\partial G_{k_2,\ell_2}^{(p_2 q_2)}}   \Delta G^{(p_1 q_1)}_{k_1,\ell_1} \Delta G^{(p_2 q_2)}_{k_2,\ell_2}  \right)\right|\\
\le{} & \left[  \left(\sum_{(k_1,\ell_1) \in \partial \Omega_{n,m}; \atop (k_2,\ell_2) \in \partial \Omega_{n,m}} + \sum_{(k_1,\ell_1) \in \partial \Omega_{n,m}; \atop (k_2,\ell_2) \in \mathring{\Omega}_{n,m}} +  \sum_{(k_1,\ell_1) \in  \mathring{\Omega}_{n,m}; \atop (k_2,\ell_2) \in \partial \Omega_{n,m}} + \sum_{(k_1,\ell_1) \in  \mathring{\Omega}_{n,m} ;\atop (k_2,\ell_2) \in \mathring{\Omega}_{n,m}} \right)  \right.  \\
&\hspace{150pt}  \left. \sum_{p_1,q_1 =1,2; \atop p_2,q_2 =1 ,2}  \E \left( \left|\frac{\partial^2 F^{(rs)}_1(\xib_{n,m})}{\partial G_{k_1,\ell_1}^{(p_1 q_1)}\partial G_{k_2,\ell_2}^{(p_2 q_2)}} \right| \right)  \right]\cdot \max_{(k,\ell)\in \Omega_{n,m}; \atop p,q=1,2}\E\left( \left| \Delta G_{k,\ell}^{(pq)}  \right|^2 \right)  \\
\le{} & \bar{\alpha}(t_{n-m})\left[ \Ns^{(\std)}_{\Omega_{n,m}}(\Delta \gb_{n,m}) \right]^2,
    \end{split}
\end{equation}
where 
\begin{equation*}
    \bar{\alpha}(t) = 16 P_2(t)(10t+16t^2+5t^3+\frac{1}{4}t^4).
\end{equation*}
Here we first need to count the number of the nodes in the set $\left|\partial \Omega_{n,m}\right| = 2(n-m)-1$ and $\left|\mathring{\Omega}_{n,m}\right| = \frac{1}{2}(n-m-1)(n-m-2)$. Then the last ``$\le$" above is done by combining the second-order derivatives of different magnitudes given in Proposition \ref{thm: second order derivative} with the corresponding number of such derivatives. For example, the first summation in the third line above together with the second-order derivatives such that conditions \textbf{(a)}-\textbf{(d)} are satisfied will contribute $P_2(t_{n-m})h \cdot (4+2\times 3(n-m)) \le  P_2(t_{n-m}) \cdot 10(t_{n-m})$ to the final estimate in the last line. Similar analysis apply to other three summations.




Inserting the two estimates above into \eqref{eq: est k1 expansion} gives us the evaluation for $\left\|\E(\tK_1 - K_1)\right\|$ as \eqref{eq: bound k_1}.

\textbf{(ii)} Estimate of $\left\|\E(\tK_2 - K_2)\right\|$:

Using the same method as the estimation of $\left\|\E(\tK_1 - K_1)\right\|$, we have
\begin{equation} \label{eq: K2}
\begin{split}
\left\|\E(\tK_2 - K_2)\right\| & \le 4P_1(t_{n-m+1}) h \left\{ \left\|\E(\tG^*_{n+1,m} - G^*_{n+1,m}) \right\| + \sum_{i = 1}^{n-m} \left( 2 + (n-m+1-i)h  \right) \left\| \E \left( \Delta \gb^*_{n,m}  \right) \right\|_{\Gamma^*_{n,m}(i)} \right\} \\
& \quad + \bar{\alpha}(t_{n-m+1})\left[ \Ns^{(\std)}_{\Omega^*_{n,m}}(\Delta \gb^*_{n,m}) \right]^2,
\end{split}
\end{equation}
which is similar to \eqref{eq: bound k_1}.
By the definition of $G^*_{n+1,m}$ and $\tG^*_{n+1,m}$ given in \eqref{def: scheme 1} and \eqref{def: scheme 2} respectively, we can estimate $\left\|\E(\tG^*_{n+1,m} - G^*_{n+1,m}) \right\|$ as
\begin{equation} \label{eq: EG^*}
\begin{split}
& \|\E(\tG^*_{n+1,m} - G^*_{n+1,m})\| \\
 \le{} & \left\|\E\left((  I + \sgn(t_n - t)\ii H_s h ) \Delta G_{n,m}\right) \right\| + h\|\E(\tK_1 - K_1)\| \\
\le{} & \left(1+\frac{1}{2}\bdH^2 h^2\right)\left\|\E\left( \Delta G_{n,m}\right) \right\| + 8P_1(t_{n-m}) h^2 \sum_{i = 1}^{n-m} \left( 2 + (n-m-i)h  \right)\left\| \E \left( \Delta \gb_{n,m}  \right) \right\|_{\Gamma_{n,m}(i)}\\
&\hspace{200pt}+ \bar{\alpha}(t_{n-m})h \left[ \Ns^{(\std)}_{\Omega_{n,m}}(\Delta \gb_{n,m}) \right]^2,
\end{split}
\end{equation}
where we have applied \eqref{eq: bound k_1} in the last inequality. Thus it remains only to bound
\begin{equation} \label{eq:Nstd}
\left[\Ns^{(\std)}_{\Omega^*_{n,m}}(\Delta \gb^*_{n,m}) \right]^2 = \max \left(\left[\Ns^{(\std)}_{\bar{\Omega}_{n,m}}(\Delta \gb^*_{n,m}) \right]^2,\E(\|\tG^*_{n+1,m} - G^*_{n+1,m}\|^2)\right),
\end{equation}
for which we just need to focus on the estimation of $\E(\|\tG^*_{n+1,m} - G^*_{n+1,m}\|^2)$. Again by the definitions \eqref{def: scheme 1} and \eqref{def: scheme 2}, we have
\begin{equation}
  \begin{split}
\E(\|\tG^*_{n+1,m} - G^*_{n+1,m}\|^2) \le 2  (1 + \bdH^2 h^2 ) \E(\|\Delta G_{n,m}\|^2) + 2h^2\E(\|\tK_1 - K_1\|^2).
      \end{split}
\end{equation}
By \eqref{eq: bound k_1^2}, we can estimate $\E(\|\tK_1 - K_1\|^2)$ as
\begin{equation}
  \begin{split}
   & \E(\|\tK_1 - K_1\|^2) \\
   \le & \  128P^2_1(t_{n-m}) h^2 \left\{  \sum_{i = 1}^{n-m} \left( 2 + (n-m-i)h  \right) \Ns^{(\std)}_{\Gamma_{n,m}(i)}(\Delta \gb_{n,m})  \right\}^2+ 8\bar{\gamma}(t_{n-m})\cdot \frac{1}{N_s} \\
   \le & \ 128 P^2_1(t_{n-m})(2+t_{n-m})^2 t^2_{n-m} \left[ \Ns^{(\std)}_{\Omega_{n,m}}(\Delta \gb_{n,m}) \right]^2 + 8\bar{\gamma}(t_{n-m})\cdot \frac{1}{N_s}
      \end{split}
\end{equation}
Therefore
\begin{equation*}
    \begin{split}
&\E(\|\tG^*_{n+1,m} - G^*_{n+1,m}\|^2)\\ 
 \le{} & 2  (1 + \bdH^2 h^2 ) \E(\|\Delta G_{n,m}\|^2) + 256 P^2_1(t_{n-m})(2+t_{n-m})^2 t^2_{n-m}h^2 \left[ \Ns^{(\std)}_{\Omega_{n,m}}(\Delta \gb_{n,m}) \right]^2+ 16\bar{\gamma}(t_{n-m})\cdot \frac{h^2}{N_s} \\
 \le{} & 2\left[1 + \left( \bdH^2  +  128 P^2_1(t_{n-m})(2+t_{n-m})^2 t^2_{n-m}\right)h^2     \right]\left[ \Ns^{(\std)}_{\Omega_{n,m}}(\Delta \gb_{n,m}) \right]^2 + 16\bar{\gamma}(t_{n-m})\cdot\frac{h^2}{N_s} \\
 \le{} & 4\left[ \Ns^{(\std)}_{\Omega_{n,m}}(\Delta \gb_{n,m}) \right]^2 + 16\bar{\gamma}(t_{n-m})\cdot\frac{h^2}{N_s}
    \end{split}
\end{equation*}
if $h\le 1/\sqrt{\bdH^2  +  128 P^2_1(t_{n-m})(2+t_{n-m})^2 t^2_{n-m}}$. Inserting this inequality into \eqref{eq:Nstd}, one obtains
\begin{equation}\label{eq: est delta G^2}
\left[ \Ns^{(\std)}_{\Omega^*_{n,m}}(\Delta \gb^*_{n,m}) \right]^2   \le 4\left[ \Ns^{(\std)}_{\bar{\Omega}_{n,m}}(\Delta \gb^*_{n,m}) \right]^2 + 16\bar{\gamma}(t_{n-m})\cdot\frac{h^2}{N_s}.
\end{equation}


Finally, the estimate \eqref{eq: bound k_2} can be obtained by inserting the estimates \eqref{eq: EG^*} and \eqref{eq: est delta G^2} into \eqref{eq: K2} and require $h \le \frac{1}{\sqrt{2 P_1(t_{n-m+1})}}$.
\end{proof}

\begin{proof}[Proof of \eqref{1st error upper bound}]

By \eqref{eq: formula first order error}, we estimate the bias by
\begin{equation}\label{eq: est first order error}
   \begin{split}
\|\E(\dG_{n+1,m})\| \le & \ \|A_{n,m}(h)\E(\dG_{n,m})\| + \frac{1}{2} h \|B_{n,m}(h)\E(\tK_1 - K_1)\| + \frac{1}{2} h \|\E(\tK_2 - K_2)\| \\
\le & \ \left(1+\frac{1}{8}\bdH^4 h^4\right)\|\E(\dG_{n,m})\| +  h \|\E(\tK_1 - K_1)\| + \frac{1}{2} h \|\E(\tK_2 - K_2)\|.
  \end{split}
\end{equation}
Now we can insert \eqref{eq: bound k_1} and \eqref{eq: bound k_2} into the above equation to get the recurrence relation stated in \eqref{eq: recurrence 1}, where the error $\Ns^{(\mathrm{std})}_{\bar{\Omega}_{n,m}}(\Delta \gb^*_{n,m})$ can be bounded by \eqref{2ed error upper bound}, resulting in
\begin{equation}
  \begin{split}
&\|\E(\dG_{n+1,m})\|\\
\le&  \   \left(1+\frac{1}{8}\bdH^4 h^4\right)\|\E(\dG_{n,m})\| 
+ 22P_1(t_{n-m+1}) h^2  \sum_{i = 1}^{n-m} \left( 2 + (n-m+1-i)h  \right)     \left\| \E ( \Delta \gb^*_{n,m}) \right\|_{\Gamma^*_{n,m}(i)}\\
& + \left(\frac{7}{2}\bar{\alpha}(t_{n-m+1})\theta^2_2 \bar{\gamma}(t_{n-m+1})\left(\ee^{2 \theta_1 \sqrt{P_1(t_{n-m+1})} t_{n-m+1}}\right) \cdot \frac{h^2}{N_s}  + 8 \bar{\alpha}(t_{n-m+1})\bar{\gamma}(t_{n-m+1})\cdot \frac{h^3}{N_s} \right) \\
\le&  \   \left(1+\frac{1}{8}\bdH^4 h^4\right)\|\E(\dG_{n,m})\| 
+ 22P_1(t_{n-m+1}) h^2  \sum_{i = 1}^{n-m} \left( 2 + (n-m+1-i)h  \right)  \left\| \E ( \Delta \gb^*_{n,m}) \right\|_{\Gamma^*_{n,m}(i)}\\
&+ 4\bar{\alpha}(t_{n-m+1})\theta^2_2 \bar{\gamma}(t_{n-m+1})\left(\ee^{2 \theta_1 \sqrt{P_1(t_{n-m+1})} t_{n-m+1}}\right) \cdot \frac{h^2}{N_s} 
  \end{split}
\end{equation}
upon assuming $h \le \frac{1}{16}$.

We notice that the above inequality is simply the recurrence relation \eqref{eq: Runge Kutta error} with the last term changed. Therefore, we can repeat the application of \eqref{eq: Runge Kutta error} and find the following estimate: 
\begin{equation}
    \|\E(\dG_{n+1,m})\| \le 4\theta^2_2 \bar{\alpha}(t_{n-m+1})\bar{\gamma}(t_{n-m+1})\left(\ee^{2 \theta_1 \sqrt{P_1(t_{n-m+1})} t_{n-m+1}}\right)(1+ \theta_1\sqrt{P_1(t_{n-m+1})}h)^{n-m+1}\cdot \frac{h}{N_s}
\end{equation}
which leads to the final estimate for the bias stated in \eqref{1st error upper bound}. 
\end{proof}

\section{Estimation of the derivatives of $F_1(\cdot)$}
\label{sec: derivatives}
In this appendix, we provide a detailed proof for the bounds of first-order derivatives of $F_1(\cdot)$. The proof for the second-order derivatives will only be sketched.

\subsection{Proof of Proposition \ref{thm: first order derivative}---Estimate the first-order derivatives}

\begin{proof}
We are looking for an upper bound for the derivative $\frac{ \partial F_1(\xib_{n,m})}{\partial G_{k,\ell}^{(pq)}}$ with $\xib_{n,m}$ being a convex combination of $\ge_{n,m}$, $\gb_{n,m}$ and $\tg_{n,m}$ defined by 
\begin{equation}
    \xib_{n,m}: = (\Xi_{m+1,m};\Xi_{m+2,m+1},\Xi_{m+2,m};\cdots;\Xi_{n,n-1},\cdots,\Xi_{n,m})
\end{equation}
satisfying 
\begin{equation}
    \Xi_{j,k} =  c_1 \Ge(t_j,t_k) + c_2  G_{j,k} + (1-c_1 -  c_2) \tG_{j,k}
\end{equation}
for some constants $0 \le c_1,c_2 \le 1$ and $c_1+c_2 \le 1$ given $m\le k<j \le n$. According to the assumption (H1) on the boundedness of $\Ge$, $G$ and $\tG$, we immediately have
\begin{equation}
    \|\Xi_{j,k}\| \le  \bdG \text{~for any~}  j,k = 0,1,\cdots,N-1,N^-,N^+,N+1,\cdots 2N-1,2N.
\end{equation}

In \eqref{def: scheme 1}, we write $F_1(\xib_{n,m})$ as 
\begin{equation*}
F_1(\xib_{n,m}) = \sgn(t_n - t) \sum^{\bar{M}}_{\substack{M=1 \\ M \text{~is odd~}}} \ii^{M+1}  \Fs_M(\xib_{n,m})   
\end{equation*}
where 
\begin{equation}\label{eq:deriv_FM}
\Fs_M(\xib_{n,m}):=  \int_{t_n > \vec{\sb} > t_m } (-1)^{\#\{\vec{\sb} < t\}}W_s I_h \Xi(t_n,s_M) W_s  \cdots W_s I_h \Xi(s_1,t_m) \Ls(t_n,\vec{\sb})  \ \dd \vec{\sb}.
\end{equation}
Therefore, we focus on the estimation of $\frac{ \partial \Fs_M(\xib_{n,m})}{\partial G_{k,\ell}^{(pq)}}$ for each odd integer $M$. The two cases given in equation \eqref{eq: 1st order bound} will be discussed separately below.

\paragraph{(I) $(k,\ell) \in \partial \Omega_{n,m}$}
This case includes $\frac{ \partial \Fs_M(\xib_{n,m})}{\partial G_{n,\ell}^{(pq)}}$ and $\frac{ \partial \Fs_M(\xib_{n,m})}{\partial G_{k,m}^{(pq)}}$. Here we only consider the derivative $\frac{ \partial \Fs_M(\xib_{n,m})}{\partial G_{n,\ell}^{(pq)}}$ since the analysis for the other is similar. 

For each $\Fs_M(\xib_{n,m})$, we split the derivative by  
\begin{equation} \label{eq: FM splitting}
\frac{\partial \Fs_M (\xib_{n,m})}{\partial G_{n,\ell}^{(pq)}} = \sum_{j = 0}^M \frac{\partial \Is_j (\xib_{n,m})}{\partial G_{n,\ell}^{(pq)}}.
\end{equation}
where
\begin{equation}\label{eq:deriv_Ij}
  \begin{split}
\Is_j(\xib_{n,m}) = &\int_{t_n > s_M > \cdots > s_{j+1} > t_{n-1}  > s_j  > \cdots > s_1 > t_m}  \\
&\hspace{30pt} (-1)^{\#\{\vec{\sb} < t\}}W_s  I_h \Xi(t_n,s_M) W_s I_h \Xi(s_M,s_{M-1}) \cdots W_s I_h \Xi(s_1,t_m) \Ls(t_n,\vec{\sb}) \ \dd \vec{\sb},
  \end{split}
\end{equation}
in which we require that $t_{n-1}$ is between $s_j$ and $s_{j+1}$.
We write the integrand of each $\Is_j(\xib_{n,m})$ as $\Gs^j_1 \times I_h \Xi(s_{j+1},s_j) \times \Gs^j_2$  where
\begin{equation}\label{def:G1 and G2}
 \begin{split}
\Gs^j_1 =\ &  (-1)^{\#\{\vec{\sb} < t\}} W_s  I_h \Xi(t_n,s_M) W_s \cdots W_s I_h \Xi(s_{j+2},s_{j+1})W_s, \\
\Gs^j_2 =\ & W_s  I_h \Xi(s_j,s_{j-1}) W_s  \cdots W_s I_h \Xi(s_1,t_m)\Ls(t_n,\vec{\sb})
  \end{split}
\end{equation}
for $1 \le j \le M-1$ (the formula for $j=0$ and $j=M$ is slightly different but easy to get). One can easily verify that $\Gs^j_2$ is completely independent from the factor $\Xi_{n,\ell}$ since the time sequence $\{s_1, \cdots, s_j\}$ is at least one time step away from $t_n$ and thus the interpolation of any $I_h \Xi(s_{i+1},s_i)$ in $\Gs^j_2$ never uses the value of $\Xi_{n,\ell}$. On the other hand, the value of $\Gs^j_1$ as well as the ``interface'' $I_h \Xi(s_{j+1},s_j)$ may or may not rely on $\Xi_{n,\ell}$, depending on how $\ell$ is given, as leads to the two cases we are going to discuss below.

\input{images/fig_deriv_case1_1.tex} 

\subparagraph{Case 1: $\ell=n-1$.}
This is the most complicated case in this proof. Note that $\Gs^j_1$ depends on $\Xi_{n,n-1}$ due to the fact that we get each $I_h \Xi(s_{i+1},s_i)$ in $\Gs^j_1$ by the interpolation $I_h \Xi(s_{i+1},s_i) = c_{i,1} \Xi_{n,n} + c_{i,2}\Xi_{n-1,n-1} + c_{i,3} \Xi_{n,n-1}$ with some coefficients $|c_i| < 1$. The ``interface" $I_h \Xi(s_{j+1},s_j)$ depends on $\Xi_{n,n-1}$ only when $s_j$ is restricted between $(t_{n-2}, t_{n-1})$ where $I_h \Xi(s_{j+1},s_j) = c_{j,1} \Xi_{n-1,n-1} + c_{j,2}\Xi_{n-1,n-2} + c_{j,3} \Xi_{n,n-1}+ c_{j,4} \Xi_{n,n-2}$. One may refer to Figure \ref{fig:deriv_case1_1} for a better understanding. 

For these reasons, we further divide the derivative $\frac{\partial\Is_j (\xib_{n,m})}{\partial G_{n,n-1}^{(pq)}}$ into two parts:
\begin{equation}\label{eq: deriv_Ij}
   \begin{split}
&\frac{\partial\Is_j (\xib_{n,m})}{\partial G_{n,n-1}^{(pq)}} =\\
 ={} & \int_{t_n > s_{M} > \cdots > s_{j+1} > t_{n-1}} \int_{t_{n-1} > s_j > t_{n-2}} \int_{s_j  > s_{j-1} >\cdots  >s_1  > t_m} \left[   \frac{\partial}{\partial \Xi_{n,n-1}^{(pq)}}\big(  \Gs^j_1  I_h \Xi(s_{j+1},s_j) \big)  \right] \Gs^j_2 \ \dd\vec{\sb} \\
& \hspace{45pt}+  \int_{t_n  > s_M > \cdots > s_{j+1} > t_{n-1}} \int_{t_{n-2} > s_j > t_{m}} \int_{s_j  > \cdots  > t_m} \left(   \frac{\partial}{\partial \Xi_{n,n-1}^{(pq)}}  \Gs^j_1   \right) I_h \Xi(s_{j+1},s_j) \Gs^j_2  \ \dd\vec{\sb}.
   \end{split}
\end{equation}
For the first integral above, we compute the derivative in the square bracket by
\begin{align*}
 &\frac{\partial}{\partial \Xi_{n,n-1}^{(pq)}}  \left( \Gs^j_1  I_h \Xi(s_{j+1},s_j) \right) =  \notag \\
 =& \   
 \sum_{i=j}^{M} (-1)^{\#\{\vec{\sb} < t\}} W_s  I_h \Xi(t_n,s_M) W_s \cdots W_s \left[  \frac{\partial}{\partial \Xi_{n,n-1}^{(pq)}} I_h \Xi(s_{i+1},s_i)  \right] W_s \cdots W_s I_h \Xi(s_{j+1},s_j)  \notag \\
 =&  \ \sum_{i=j}^{M} (-1)^{\#\{\vec{\sb} < t\}} W_s  I_h \Xi(t_n,s_M) W_s \cdots W_s c_{i,3}E_{pq} W_s \cdots W_s I_h \Xi(s_{j+1},s_j) .
 \end{align*} 
Here $E_{pq}$ is defined as a 2-by-2 matrix with its $pq-$entry being the only non-zero entry equal to 1. By the hypothesis (H1)(H3), this integral is therefore bounded by
\begin{equation}\label{eq: est_case_1_1_part_1}
 \begin{split}
&\left\| \int_{t_n > s_M > \cdots > s_{j+1} > t_{n-1}} \int_{t_{n-1} > s_j > t_{n-2}} \int_{s_j > s_{j-1} > \cdots > s_1 > t_m} \left[  \frac{\partial}{\partial \Xi_{n,n-1}^{(pq)}}  \Gs^j_1  I_h \Xi(s_{j+1},s_j)  \right] \Gs^j_2 \ \dd\vec{\sb} \right\|  \\
\le & \ (M-j+1) \bdW^{M+1}\bdG^M  (M!! \bdL^{\frac{M+1}{2}}) \times \int_{t_n > s_M > \cdots > s_{j+1} > t_{n-1}} \int_{t_{n-1} > s_j > t_{n-2}} \int_{s_j > s_{j-1} > \cdots > s_1 > t_m} 1 \ \dd\vec{\sb}  \\
\le & \  \bdW^{M+1}\bdG^M (M!! \bdL^{\frac{M+1}{2}}) \times (M-j+1) \times \frac{1}{(M-j)!(j-1)!}(t_{n-1} - t_m)^{j-1} h^{M-j+1}. 
   \end{split}
\end{equation}
We notice that the upper bound above consists of three components: (1) $ \bdW^{M+1}\bdG^M  (M!! \bdL^{\frac{M+1}{2}})$: the upper bound of the integrand; (2) $M-j+1$: the number of terms with the form $I_h \Xi$ whose values depend on $\Xi_{n,n-1}$; (3) $\frac{1}{(M-j)!(j-1)!}(t_{n-1} - t_m)^{j-1} h^{M-j+1}$: the area of the domain of integration. Similarly, we may directly write down the upper bound for the second integral in \eqref{eq: deriv_Ij}:
\begin{multline}\label{eq: est_case_1_1_part_2}
 \left\| \int_{t_n > s_M > \cdots > s_{j+1} > t_{n-1}} \int_{t_{n-2} > s_j > t_{m}} \int_{s_j > s_{j-1} > \cdots > s_1 > t_m} \left[  \frac{\partial}{\partial \Xi_{n,n-1}^{(pq)}}  \Gs^j_1   \right] I_h \Xi(s_{j+1},s_j) \Gs^j_2 \ \dd\vec{\sb} \right\|  \\
\le  \ \bdW^{M+1}\bdG^M  (M!! \bdL^{\frac{M+1}{2}}) \times (M-j) \times \frac{1}{(M-j)!j!} (t_{n-2} -t_m)^j h^{M-j}. 
 \end{multline}
Combining the estimation \eqref{eq: est_case_1_1_part_1} and \eqref{eq: est_case_1_1_part_2} yields
\begin{equation}\label{eq: deriv Ij bd}
\left\| \frac{\partial\Is_j (\xib_{n,m})}{\partial G_{n,n-1}^{(pq)}} \right\|\le \ 2\bdW^{M+1} \bdG^M  (M!! \bdL^{\frac{M+1}{2}}) \frac{M-j+1}{(M-j)!(j-1)!} (t_{n-1}-t_m)^j h^{M-j}.
\end{equation}
As we have mentioned previously, the upper bound we obtained above is for $1 \le j \le M-1$. For $j=0$ and $j=M$, we may return to \eqref{eq:deriv_Ij} and consider these two individual cases and we easily reach to the following results with similar argument:  
 \begin{align*}
&\left\| \frac{\partial\Is_0 (\xib_{n,m})}{\partial G_{n,n-1}^{(pq)}} \right\| \le \bdW^{M+1} \bdG^M \bdL^{\frac{M+1}{2}} \frac{M+1}{(M-1)!!}  h^{M},\\
&\left\| \frac{\partial\Is_M (\xib_{n,m})}{\partial G_{n,n-1}^{(pq)}} \right\| \le \bdW^{M+1} \bdG^M \bdL^{\frac{M+1}{2}} \frac{M}{(M-1)!!} (t_{n-1}-t_m)^{M-1} h.
\end{align*}
Now we sum up all the upper bounds for $\frac{\partial \Is_j (\xib_{n,m})}{\partial G_{n,n-1}^{(pq)}}$ and use the combination relation $(1+X)^n = \sum_{k \geq 0} \left({n \atop k}\right) X^k$ to get: 
\begin{equation}\label{eq: 1st-order_case_1_1}
    \begin{split}
\left\| \frac{\partial \Fs_M (\xib_{n,m})}{\partial G_{n,n-1}^{(pq)}} \right\| \le  \sum_{j = 0}^{M} \left\|\frac{\partial \Is_j (\xib_{n,m})}{\partial G_{n,n-1}^{(pq)}} \right\| \le 5\bdW^{M+1} \bdG^M \bdL^{\frac{M+1}{2}} \frac{M}{(M-3)!!} (t_n - t_m)^{M-1} h.
   \end{split}
\end{equation}



Here we remark that the argument above is only valid when the odd number $M$ is chosen to be greater than $1$ due to the number $(M-3)!!$ in the estimate \eqref{eq: 1st-order_case_1_1}, which is not defined when $M=1$. For $M=1$, we simply return to the definition \eqref{eq:deriv_FM} and follow similar procedures to reach to the result  
\begin{equation} \label{eq: 1st-order_case_1_2}
\left\| \frac{\partial \Fs_1 (\xib_{n,m})}{\partial G_{n,n-1}^{(pq)}} \right\|\le \ 2\bdW^2 \bdG \bdL h.
\end{equation}

\subparagraph{Case 2: $\ell < n-1$.}
For this case, one may check that $\Gs^j_1$ and $\Gs^j_2$ defined in \eqref{def:G1 and G2} are both independent of $\Xi_{n,\ell}$, while the ``interface" $I_h \Xi(s_{j+1},s_j)$ depends on $\Xi_{n,\ell}$ only when $t_{\ell-1}<s_j<t_{\ell+1}$. Therefore, we simply calculate the derivative for each $\Is_j$ by 
\begin{equation*}
\frac{\partial \Is_j (\xib_{n,m})}{\partial G_{n,\ell}^{(pq)}} =  \int_{t_n > s_M > \cdots > s_{j+1} > t_{n-1}} \int_{t_{\ell+1} > s_j > t_{\ell-1}} \int_{s_j > s_{j-1} > \cdots > s_1 > t_m} \Gs^j_1 \left(\frac{\partial }{\partial \Xi_{n,\ell}^{(pq)}} I_h \Xi(s_{j+1},s_j) \right)\Gs^j_2 \ \dd\vec{\sb}.
\end{equation*}  
Note that we need to choose $1\le j \le M-1$ and when $j=0$ or $j=M$, the corresponding derivative above vanishes. We follow the analysis for \emph{Case 1} and can also obtain upper bounds for all $\left\|\frac{\partial \Is_j (\xib_{n,m})}{\partial G_{n,\ell}^{(pq)}}\right\|$, summing up which leads to the estimate of $\left\| \frac{\partial \Fs_M (\xib_{n,m})}{\partial G_{n,\ell}^{(pq)}} \right\|$. By calculation similar to Case 1, the derivative $\left\| \frac{\partial \Fs_M (\xib_{n,m})}{\partial G_{n,\ell}^{(pq)}} \right\|$ can also be bounded by the same upper bounds given in \eqref{eq: 1st-order_case_1_1} for $M \geq 3$ and \eqref{eq: 1st-order_case_1_2} for $M=1$. 

Overall, we arrive at the conclusion that for any $(k,\ell) \in \partial \Omega_{n,m}$,
\begin{equation}
 \begin{split}
\left\| \frac{\partial \Fs_M (\xib_{n,m})}{\partial G_{k,\ell}^{(pq)}} \right\| \le 
\begin{cases}
2\bdW^2 \bdG \bdL h, &\text{if~} M=1 , \\
 5\bdW^{M+1} \bdG^M \bdL^{\frac{M+1}{2}} \frac{M}{(M-3)!!} (t_n - t_m)^{M-1} h, &\text{if~} M \geq 3. 
\end{cases}
   \end{split}
\end{equation}
To complete our estimation for $\frac{ \partial F_1(\xib_{n,m})}{\partial G_{k,\ell}^{(pq)}}$, we now sum up the bounds for $\left\| \frac{\partial \Fs_M (\xib_{n,m})}{\partial G_{k,\ell}^{(pq)}} \right\|$ up to $\bar{M}$ and get a uniform bound
\begin{equation} \label{eq: est_case_1}
\left\| \frac{\partial F_1 (\xib_{n,m})}{\partial G_{k,\ell}^{(pq)}} \right\| \le 
\left(  2\bdW^2 \bdG \bdL  +  5 \bdW^2 \bdG \bdL \sum^{\bar{M}}_{\substack{M=3 \\ M \text{~is odd}}}
 \frac{M}{(M-3)!!} \left(\bdW \bdG \bdL^{1/2} t_{n-m}\right)^{M-1}     \right)   \cdot h,
\end{equation}
which proves the first case in the Theorem \ref{thm: first order derivative}.

\paragraph{(II) $(k,\ell) \in  \mathring{\Omega}_{n,m}$}
To compute $\frac{\partial \Fs_M (\xib_{n,m})}{\partial G_{k,\ell}^{(pq)}}$, we need to first find all $I_h \Xi(s_{j+1},s_j)$ in $\Fs_M (\xib_{n,m})$ that depends on $\Xi_{k,\ell}^{(pq)}$. Since $(k,\ell) \in  \mathring{\Omega}_{n,m}$, only those $I_h \Xi(s_{j+1},s_j)$ such that $t_{\ell-1}<s_{j}<t_{\ell+1}$ and $t_{k-1} < s_{j+1} < t_{k+1}$ may depend on $\Xi_{k,\ell}$. 

We first consider a special case  $M=1$, where the derivative is simply given by
\begin{equation}
\frac{\partial \Fs_1 (\xib_{n,m})}{\partial G_{k,\ell}^{(pq)}} = \frac{\partial}{\partial \Xi_{k,\ell}^{(pq)}} \int^{t_n}_{t_m}(-1)^{\#\{s_1 < t\}} W_s I_h \Xi(t_n,s_1)W_s I_h \Xi(s_1,t_m)\Ls(t_n,s_1) \ \dd s_1.
\end{equation}  
It is easy to see that neither $I_h \Xi(t_n,s_1)$ nor $I_h \Xi(s_1,t_m)$ depends on $\Xi_{k,\ell}$, since both are interpolated by $\Xi$ values only on $\partial \Omega_{n,m}$. As the result, $\frac{\partial \Fs_1 (\xib_{n,m})}{\partial G_{k,\ell}^{(pq)}} = 0$.

When $M\geq 3$, we need to consider the following two possibilities:
\subparagraph{Case 1: $k - \ell \geq 2$.}
Similar to \eqref{eq: FM splitting}, we apply the following splitting of the integral in the definition of $\Fs_M$ \eqref{eq:deriv_FM}:
\begin{equation*}
\frac{\partial \Fs_M (\xib_{n,m})}{\partial G_{k,\ell}^{(pq)}} = \sum_{j = 1}^{M-1} \frac{\partial \Js_j (\xib_{n,m})}{\partial G_{k,\ell}^{(pq)}},
\end{equation*}
where 
\begin{align*}
\Js_j(\xib_{n,m}) = & \int_{t_n>\cdots>s_{j+2}>s_{j+1}} \int_{t_{k+1}>s_{j+1}>t_{k-1}}\int_{t_{\ell+1}>s_j>t_{\ell-1}}\int_{s_j>s_{j-1}>\cdots>t_m}  \\
& \hspace{20pt} (-1)^{\#\{\vec{\sb} < t\}}W_s  I_h \Xi(t_n,s_M) W_s I_h \Xi(s_M,s_{M-1}) \cdots W_s I_h \Xi(s_1,t_m) \Ls(t_n,\vec{\sb}) \ \dd \vec{\sb}.
\end{align*}
Here a critical observation is that once we assume $t_{\ell-1}<s_{j}<t_{\ell+1},\ t_{k-1} < s_{j+1} < t_{k+1}$ for any fixed $j$, $I_h \Xi(s_{j+1},s_j)$ is then the unique term in the integrand of $\Fs_M (\xib_{n,m})$ that depends on $\Xi_{k,\ell}$ since $(t_{\ell-1},t_{\ell+1}) \cap (t_{k-1},t_{k+1}) = \emptyset$ when $k - \ell \geq 2$. This observation is illustrated in Figure \ref{fig:deriv_case2_1}.

\input{images/fig_deriv_case2_1} 

We again write the integrand above as $\Gs^j_1 \times I_h \Xi(s_{j+1},s_j) \times \Gs^j_2$ defined exactly the same as in \eqref{def:G1 and G2}, then $\Gs^j_1$ and $\Gs^j_2$ are both independent from $\Xi_{k,\ell}$. Therefore, 
\begin{equation*}
    \begin{split}
   \left\| \frac{\partial\Js_j (\xib_{n,m})}{\partial G_{k,\ell}^{(pq)}} \right\|  \le & \ \int_{t_n>\cdots>s_{j+1}} \int_{t_{k+1}>s_{j+1}>t_{k-1}}\int_{t_{\ell+1}>s_j>t_{\ell-1}}\int_{s_j>\cdots>t_m} \left\| \Gs^j_1 \left( \frac{ \partial}{\partial \Xi_{k,\ell}^{(pq)}} I_h \Xi(s_{j+1},s_j) \right)  \Gs^j_2  \right\| \ \dd \vec{\sb} \\
  \le & \ 4\bdW^{M+1} \bdG^M \bdL^{\frac{M+1}{2}} \frac{M!!}{(M-j-1)!(j-1)!}  (t_{n}-t_m)^{M-2}h^2,
   \end{split}
\end{equation*}
which leads to
\begin{equation}\label{eq: est_case_2_1}
\left\|\frac{\partial \Fs_M (\xib_{n,m})}{\partial G_{k,\ell}^{(pq)}}\right\| =  \sum_{j = 1}^{M-1} \left\| \frac{\partial \Js_j (\xib_{n,m})}{\partial G_{k,\ell}^{(pq)}} \right\| \le 4\bdW^{M+1} \bdG^M \bdL^{\frac{M+1}{2}} \frac{M}{(M-3)!!}(2t_{n-m})^{M-2} h^2.
\end{equation}

\subparagraph{Case 2: $k - \ell = 1$.}
There is a overlapping region $(t_{\ell-1},t_{\ell+1}) \cap (t_{k-1},t_{k+1}) = (t_{\ell},t_{\ell+1})$ in this case. Consequently, there can be multiple terms in the integrand depending on $\Xi_{k,\ell}$. To estimate the derivative $\frac{\partial \Fs_M (\xib_{n,m})}{\partial G_{k,\ell}^{(pq)}}$, we further divide it into three parts based on the distribution of the time sequence $\vec{\sb}$ in the integrand:
\begin{equation*}
 \frac{\partial \Fs_M (\xib_{n,m})}{\partial G_{k,\ell}^{(pq)}}  =   \sum_{v=2}^M \sum_{u=0}^{M-v} \frac{\partial \Ks^{u,v}_1}{\partial \Xi_{k,\ell}^{(pq)}} +  \left(\sum_{v=1}^{M-1} \sum_{u=1}^{M-v}  \frac{\partial \Ks^{u,v}_{2,L}}{\partial \Xi_{k,\ell}^{(pq)}} +  \sum_{v=1}^{M-1} \sum_{u=0}^{M-v-1} \frac{\partial \Ks^{u,v}_{2,R}}{\partial \Xi_{k,\ell}^{(pq)}} \right) + \sum_{v=0}^{M-2} \sum_{u=1}^{  M-v-1 }  \frac{\partial \Ks^{u,v}_3}{\partial \Xi_{k,\ell}^{(pq)}}
\end{equation*}
where 
\begin{equation*} 
 \begin{split}
 \Ks^{u,v}_1 = & \ \int_{t_n > s_M > \cdots >s_{u+v+1}>t_{k+1}}\int_{t_{\ell+1}>s_{u+v}>\cdots>s_{u+1}>t_\ell} \int_{t_{\ell-1}>s_u>\cdots>s_1>t_m} 
\\
& \hspace{80pt} \Gs_1^{u,v} \times I_h \Xi(s_{u+v},s_{u+v-1})W_s\cdots W_s I_h \Xi(s_{u+2},s_{u+1}) \times \Gs_2^{u,v} \ \dd \vec{\sb}; \\ 
 \Ks^{u,v}_{2,L} = & \ \int_{t_n > s_M > \cdots >s_{u+v+1}>t_{k+1}}\int_{t_{\ell+1}>s_{u+v}>\cdots>s_{u+1}>t_\ell}\int_{t_\ell > s_u > t_{\ell-1}} \int_{s_u>s_{u-1}>\cdots>s_1>t_m} 
\\
& \hspace{80pt} \Gs_1^{u,v} \times I_h \Xi(s_{u+v},s_{u+v-1})W_s\cdots W_s I_h \Xi(s_{u+1},s_{u}) \times \Gs_2^{u,v}  \ \dd \vec{\sb}, \\ 
 \Ks^{u,v}_{2,R} = & \ \int_{t_n > s_M > \cdots >s_{u+v+2}>s_{u+v+1}}\int_{t_{k+1>}s_{u+v+1}>t_k}\int_{t_{\ell+1}>s_{u+v}>\cdots>s_{u+1}>t_\ell}\int_{t_{\ell-1}>s_u>\cdots>s_1>t_m} 
\\
& \hspace{80pt} \Gs_1^{u,v} \times I_h \Xi(s_{u+v+1},s_{u+v})W_s\cdots W_s I_h \Xi(s_{u+2},s_{u+1}) \times \Gs_2^{u,v}  \ \dd \vec{\sb}; \\ 
 \Ks^{u,v}_{3} = & \ \int_{t_n > s_M > \cdots >s_{u+v+1}}\int_{t_{k+1>}s_{u+v+1}>t_k}\int_{t_{\ell+1}>s_{u+v}>\cdots>s_{u+1}>t_\ell}\int_{t_\ell > s_u > t_{\ell-1}}\int_{s_u>\cdots>s_1>t_m} 
\\
& \hspace{80pt} \Gs_1^{u,v} \times I_h \Xi(s_{u+v+1},s_{u+v})W_s\cdots W_s I_h \Xi(s_{u+1},s_{u}) \times \Gs_2^{u,v}  \ \dd \vec{\sb}. 
 \end{split}
\end{equation*}
With a slight abuse of notation, here $\Gs_1^{u,v}$ and $ \Gs_2^{u,v}$ denote the products with the form ``$W_s I_h \Xi \cdots W_s I_h \Xi$"  that complete the integrand.  Each $\Ks^{u,v}$ here represents a part of the integral $\Fs_M(\xib_{n,m})$ where there are $v$ time points in $\vec{\sb}$ locating in $(t_\ell,t_{\ell+1})$. These cases are illustrated in Figure \ref{fig:deriv_case2_2}: 
\begin{itemize}
\item In $\Ks^{u,v}_1$, no time point other than these $v$ points $s_{u+1}, \cdots, s_{u+v}$ is in the interval $(t_{\ell-1},t_{k+1})$.
\item In $\Ks^{u,v}_{2,L}$ (or $\Ks^{u,v}_{2,R}$), there exists at least one point other than $s_{u+1}, \cdots, s_{u+v}$ locating in $(t_{\ell-1},t_{\ell})$ (or $(t_{k},t_{k+1})$ ) while no time point appears in $(t_{k},t_{k+1})$ (or $(t_{\ell-1},t_{\ell})$).
\item In $\Ks^{u,v}_3$, there exists at least one point other than $s_{u+1}, \cdots, s_{u+v}$ in both $(t_{k},t_{k+1})$ and $(t_{\ell-1},t_{\ell})$.
\end{itemize}

By splitting $\Fs_M(\xib_{n,m})$ in this way, one may easily check that $\Gs_1^{u,v}$ and $\Gs_2^{u,v}$ in each $\Ks^{u,v}$ are all independent of $\Xi_{k,\ell}$, while all $I_h \Xi$ in between depend on $\Xi_{k,\ell}$. Therefore, we can now compute $\frac{\partial \Ks^{u,v}}{\partial \Xi_{k,\ell}^{(pq)}}$ as the product of the derivative of these $I_h \Xi$. Mimicking the previous analysis in \eqref{eq: est_case_1_1_part_1} and \eqref{eq: est_case_1_1_part_2}, we may bound the first summation 
\begin{equation}
    \begin{split}
 &\sum_{v=2}^M \sum_{u=0}^{M-v} \left\| \frac{\partial \Ks^{u,v}_1}{\partial \Xi_{k,\ell}^{(pq)}} \right\| \\
 \le & \   \sum_{v=2}^M \sum_{u=0}^{M-v} \bdW^{M+1} \bdG (M!! \bdL^{\frac{M+1}{2}}) \times (v-1) \times \frac{(t_n - t_{k+1})^{M-u-v}}{(M-u-v)!}\cdot \frac{h^v}{v!} \cdot \frac{(t_{\ell -1} - t_m)^u}{u!} \\
 \le & \ \sum_{v=2}^M  \bdW^{M+1} \bdG (M!! \bdL^{\frac{M+1}{2}}) \frac{v-1}{v!} (t_{n-m-1})^{M-v} h^v \sum_{u=0}^{M-v} \frac{1}{(M-v-u)!u!} \\
 = & \ \sum_{v=2}^M  \bdW^{M+1} \bdG (M!! \bdL^{\frac{M+1}{2}}) \frac{v-1}{(M-v)!v!} (2t_{n-m-1})^{M-v} h^v  \\
 = & \  \bdW^{M+1} \bdG (M!! \bdL^{\frac{M+1}{2}}) (2t_{n-m-1})^{M-2} h^2 \sum_{v=0}^{M-2} \frac{(M-2)!}{(M-2-v)!v!}\left( \frac{h}{2t_{n-m-1}}\right)^v \cdot \frac{1}{(M-2)!(v+2)} \\
 \le & \ \frac{1}{2} \bdW^{M+1} \bdG (M!! \bdL^{\frac{M+1}{2}})\frac{1}{(M-2)!}(2t_{n-m-1}+h)^{M-2}h^2 \\
 \le & \ \frac{1}{2}\bdW^{M+1} \bdG \bdL^{\frac{M+1}{2}} \frac{M}{(M-3)!!}(2t_{n-m})^{M-2}h^2.
 \end{split}
\end{equation}
Similarly, we may estimate the other summations as 
\begin{equation*}
    \left(\sum_{v=1}^{M-1} \sum_{u=1}^{M-v} \left\| \frac{\partial \Ks^{u,v}_{2,L}}{\partial \Xi_{k,\ell}^{(pq)}} \right\|+  \sum_{v=1}^{M-1} \sum_{u=0}^{M-v-1} \left\| \frac{\partial \Ks^{u,v}_{2,R}}{\partial \Xi_{k,\ell}^{(pq)}} \right\| \right) \le 2 \bdW^{M+1} \bdG \bdL^{\frac{M+1}{2}} \frac{M}{(M-3)!!}(2t_{n-m})^{M-2}h^2
\end{equation*}
and 
\begin{equation*}
    \sum_{v=0}^{M-2} \sum_{u=1}^{  M-v-1 }  \left\| \frac{\partial \Ks^{u,v}_3}{\partial \Xi_{k,\ell}^{(pq)}} \right\| \le   \bdW^{M+1} \bdG \bdL^{\frac{M+1}{2}} \frac{(M-1)M}{(M-3)!!}(2t_{n-m})^{M-2}h^2.
\end{equation*}
Therefore, we get the estimate
\begin{equation}\label{eq: est_case_2_2}
\left\| \frac{\partial \Fs_M (\xib_{n,m})}{\partial G_{k,\ell}^{(pq)}} \right\| \le  3\bdW^{M+1} \bdG^M \bdL^{\frac{M+1}{2}}\frac{(M-1)M}{(M-3)!!}(2t_{n-m})^{M-2}h^2.
\end{equation} 
Note that the upper bound above is strictly greater than the \emph{Case 1} bound given in \eqref{eq: est_case_2_1}. Therefore, we may use the upper bound in \eqref{eq: est_case_2_2} as the uniform bound for both cases.

\input{images/fig_deriv_case2_2} 

As a summary, we have obtained the following result: for $(k,\ell) \in \mathring{\Omega}_{n,m}$,
\begin{align*}
\left\| \frac{\partial \Fs_M (\xib_{m,n})}{\partial G_{k,\ell}^{(pq)}} \right\|\le
\begin{cases}
0, & \text{if}~ M=1, \\
3\bdW^{M+1} \bdG^M \bdL^{\frac{M+1}{2}}\frac{(M-1)M}{(M-3)!!}(2t_{n-m})^{M-2}h^2,& \text{if}~ M \geq 3.
\end{cases}
\end{align*}
Summing up these estimates, we see that for odd $\bar{M}>1$,
\begin{equation}\label{eq: est_case_2}
\left\| \frac{\partial F_1 (\xib_{n,m})}{\partial G_{k,\ell}^{(pq)}} \right\| \le   \sum^{\bar{M}}_{\substack{M=1 \\ M ~\text{is odd}}}  \left\| \frac{\partial \Fs_M (\xib_{n,m})}{\partial G_{k,\ell}^{(pq)}} \right\| \le  \left( 3\bdW^3 \bdG^2 \bdL^{\frac{3}{2}}  \sum^{\bar{M}}_{\substack{M=3 \\ M ~\text{is odd}}} \frac{(M-1)M}{(M-3)!!}(2\bdW \bdG \bdL^{1/2}t_{n-m})^{M-2} \right) \cdot h^2.
\end{equation}

By now, all the cases have been discussed, and the final conclusion \eqref{eq: 1st order bound} is a simple combination of \eqref{eq: est_case_1} and \eqref{eq: est_case_2}, as completes the proof.
\end{proof}

\subsection{Proof of Proposition \ref{thm: second order derivative}---Estimate the second-order derivatives}

The proof of Proposition \ref{thm: second order derivative} is quite tedious and does not shed much light. Moreover the error contributed by the second-order derivatives play a less important role in our final result. Thus we will only provide the outline of the proof stating the idea without technical details.

By the definition of $F_1(\cdot)$, we decompose the second-order derivative as 
\begin{equation} \label{eq: 2nd order derivative}
\frac{\partial^2 F_1 (\xib_{n,m})}{\partial G_{k_1,\ell_1}^{(p_1 q_1)}\partial G_{k_2,\ell_2}^{(p_2 q_2)}} =\sgn(t_n - t) \sum^{\bar{M}}_{\substack{M=1 \\ M \text{~is odd~}}} \ii^{M+1} \frac{\partial^2 \Fs_M(\xib_{n,m})}{\partial G_{k_1,\ell_1}^{(p_1 q_1)}\partial G_{k_2,\ell_2}^{(p_2 q_2)}} ,
\end{equation}
where the definition of the function $\Fs_M(\xib_{n,m})$ is given in \eqref{eq:deriv_FM}. Since each $I_h \Xi$ is obtained via linear interpolation, one can easily check the following result:
\begin{lemma}
If $\frac{\partial^2 \Fs_M(\xib_{n,m})}{\partial G_{k_1,\ell_1}^{(p_1 q_1)}\partial G_{k_2,\ell_2}^{(p_2 q_2)}} $ is non-zero, there exist at least two factors $I_h \Xi(s_{j_1},s_{j_1 -1})$ and $I_h \Xi(s_{j_2},s_{j_2 -1})$ in the integrand of \eqref{eq:deriv_FM} with $1\le j_1 \neq j_2 \le M+1$ such that
\begin{equation}\label{claim:non zero derivative}
    t_{k_i-1} < s_{j_i} < t_{k_i+1} \text{~and~} t_{\ell_i-1} < s_{j_i-1} < t_{\ell_i+1}, \quad \text{for } i=1,2,
\end{equation}
where we define $s_{M+1} = t_n$ and $s_{0} = t_m$.
\label{lemma: 2nd order}
\end{lemma}
The entire argument in the rest of this section will be based on this result. Similar to the proof of the bounds for the first-order derivatives, we need to take into account four possibilities for the locations of $(k_1,\ell_1)$ and $(k_2, \ell_2)$.

\paragraph{(I) $(k_1,\ell_1) \times (k_2,\ell_2) \in \partial \Omega_{n,m} \times \partial \Omega_{n,m}$}
Since $\partial \Omega_{n,m}$ includes two sides (see Figure \ref{fig:sets}), we are going to study the two cases where the two nodes $(k_1, \ell_1)$ and $(k_2, \ell_2)$ are on the same/different sides.
 
\subparagraph{Case 1: $k_1 = k_2 =n$ or $\ell_1 = \ell_2 =m$.}
This is the case where $(k_1,\ell_1)$ and $(k_2,\ell_2)$ are on the same side of $\partial \Omega_{n,m}$. Here we only focus on the case $k_1 = k_2 =n$ and the analysis for the other case $\ell_1 = \ell_2 =m$ is similar. We may check that 
\begin{itemize}
\item If $\ell_1,\ell_2 \le n-2$, by Lemma \ref{lemma: 2nd order}, in order that the second-order derivative $\frac{\partial^2 \Fs_M(\xib_{n,m})}{\partial G_{k_1,\ell_1}^{(p_1 q_1)}\partial G_{k_2,\ell_2}^{(p_2 q_2)}}$ is nonzero, there exist distinct $j_1$ and $j_2$ such that
\begin{gather}
\label{eq: cond1}
t_{n-1} < s_{j_1} \le t_{n}, \qquad t_{\ell_1 -1} < s_{j_1 -1} < t_{\ell_1 + 1} \le t_{n-1}, \\
\label{eq: cond2}
t_{n-1} < s_{j_2} \le t_{n}, \qquad t_{\ell_2 -1} < s_{j_2 -1} < t_{\ell_2 + 1} \le t_{n-1}.
\end{gather}
However, the conditions \eqref{eq: cond1} and \eqref{eq: cond2} contradict each other because the point $t_{n-1}$ can only locate between one pair of adjacent points in the sequence $\vec{\sb}$.
Thus $ \frac{\partial^2 F_1(\xib_{n,m})}{\partial G_{n,\ell_1}^{(p_1 q_1)}\partial G_{n,\ell_2}^{(p_2 q_2)}}$ is always zero. 

\item If $\ell_1 = n-1$ and $m+1 \le \ell_2 \le n-2$ (same for $\ell_2 = n-1$ and $m+1 \le \ell_1 \le n-2$), the corresponding conditions are
\begin{gather*}
t_{n-1} < s_{j_1} \le t_{n}, \qquad t_{n-2} < s_{j_1 -1} < t_n, \\
t_{n-1} < s_{j_2} \le t_{n}, \qquad t_{\ell_2 -1} < s_{j_2 -1} < t_{\ell_2 + 1} \le t_{n-1}.
\end{gather*}
Such $j_1$ and $j_2$ can be found only if there is at least one point in $\vec{\sb}$ between $t_{n-1}$ and $t_n$. Therefore when $M > 1$
\begin{equation} \label{eq: m gt 1}
\begin{split}
\frac{\partial^2 \Fs_M(\xib_{n,m})}{\partial G_{n,n-1}^{(p_1 q_1)}\partial G_{n,\ell_2}^{(p_2 q_2)}} &=
\sum_{r=1}^{M-1} \int_{t_n > s_{M} > \cdots > s_{M-r+1} > t_{n-1}} \int_{t_{\ell_2 + 1} > s_{M-r} > t_{\ell_2 -1} }\int_{s_{M-r} > \cdots > s_1 > t_m} \\
& \qquad (-1)^{\#\{\vec{\sb} < t\}} \frac{\partial}{\partial \Xi_{n,n-1}^{(p_1 q_1)}} \left( W_s I_h \Xi(t_n,s_M) W_s \cdots  I_h \Xi(s_{M-r+2},s_{M-r+1}) \right) \times \\
& \qquad \frac{\partial}{\partial \Xi_{n,\ell_2}^{(p_2 q_2)}} \Big( W_s I_h \Xi(s_{M-r+1},s_{M-r}) W_s \cdots I_h \Xi(s_1,t_m)\Big) \Ls(t_n,\vec{\sb})  \ \dd \vec{\sb}.
\end{split}
\end{equation}
When $M = 1$, the derivative is zero. The magnitude of the above sum can be observed from the sizes of the integral domains. The leading-order term is provided by $r = 2$, which gives
\begin{equation}\label{eq:2ed deriv same leg}
\frac{\partial^2 F_1 (\xib_{n,m})}{\partial G_{n,\ell_1}^{(p_1 q_1)}\partial G_{n,\ell_2}^{(p_2 q_2)}}  \sim O(h^2).
\end{equation}

\item If $\ell_1 = n-1$ and $\ell_2 = m$ (same for $\ell_2 = n-1$ and $\ell_1 = m$), the analysis for $M > 1$ is the same as \eqref{eq: m gt 1}. When $M = 1$, we have
\begin{equation}\label{eq:analysis O(h) M=1}
\begin{split}
\frac{\partial^2 \Fs_M(\xib_{n,m})}{\partial G_{n,n-1}^{(p_1 q_1)}\partial G_{n,\ell_2}^{(p_2 q_2)}} &=
\int_{t_{n-1}}^{t_n} (-1)^{\#\{s_1 < t\}} \frac{\partial}{\partial \Xi_{n,n-1}^{(p_1 q_1)}} \left( W_s I_h \Xi(t_n,s_1) \right) \frac{\partial}{\partial \Xi_{n,\ell_2}^{(p_2 q_2)}} \Big( W_s I_h \Xi(s_1,t_m)\Big) \Ls(t_n,s_1)  \ \dd s_1 \\
& \sim O(h).
\end{split}
\end{equation}
Therefore the derivative \eqref{eq: 2nd order derivative} also has magnitude $O(h)$.

\item If $\ell_1 = \ell_2 = n-1$, we need to find distinct $j_1$ and $j_2$ such that
\begin{gather*}
t_{n-1} < s_{j_1} \le t_{n}, \qquad t_{n-2} < s_{j_1 -1} < t_n, \\
t_{n-1} < s_{j_2} \le t_{n}, \qquad t_{n-2} < s_{j_2 -1} < t_{n}.
\end{gather*}
These conditions can be satisfied only if at least two points in $\vec{\sb}$ are in $(t_{n-2}, t_n)$, as also results in \eqref{eq:2ed deriv same leg}.
\end{itemize}

\subparagraph{Case 2: $k_1  =n,\ell_2 =m$ or $\ell_1 = m,k_2 =n$.}
This is the case where $(k_1,\ell_1)$ and $(k_2,\ell_2)$ are on the different sides of $\partial \Omega_{n,m}$ in Figure \ref{fig:sets}. Again we focus only one case $k_1  =n,\ell_2 =m$, and the other case is similar.

\begin{itemize}
  
\item   If $|\ell_1 - k_2 | > 1$ (same for $|\ell_2 - k_1 | > 1$ when $\ell_1 = m,k_2 = n$), we similarly propose the conditions
\begin{gather*}
t_{n-1} < s_{j_1} \le t_{n}, \qquad t_{\ell_1-1} < s_{j_1 -1} < t_{\ell_1 + 1}, \\
t_{k_2-1} < s_{j_2} < t_{k_2+1}, \qquad t_{m} \le s_{j_2 -1} < t_{m+1}.
\end{gather*}
When $M = 1$, the derivative is zero. When $M>1$, the leading-order term in $\frac{\partial^2\Fs_M(\xib_{n,m})}{\partial G_{n,\ell_1}^{(p_1 q_1)}\partial G_{k_2,m}^{(p_2 q_2)}}$ is the part of integral where we let $I_h \Xi(s_{j_1},s_{j_1 -1}) = I_h \Xi(t_n,s_{M})$ and $I_h \Xi(s_{j_2},s_{j_2 -1}) = I_h \Xi(s_1,t_m)$ respectively in \eqref{claim:non zero derivative}. As a result, we have the restriction $t_{\ell_1 -1} < s_M < t_{\ell_1 + 1},t_{k_2 -1} < s_1 < t_{k_2 + 1}$, which leads to 
\begin{equation}\label{eq:2ed deriv diff leg}
\frac{\partial^2  F_1(\xib_{n,m})}{\partial G_{n,\ell_1}^{(p_1 q_1)}\partial G_{k_2,m}^{(p_2 q_2)}}   \sim O(h^2).
\end{equation}

\item  If $|\ell_1 - k_2 | \le 1$ (same for $|\ell_2 - k_1 | \le 1$ when $\ell_1 = m,k_2 = n$), we propose exactly the same conditions for $|\ell_1 - k_2|>1$. When $M>1$, the derivative again has magnitude $O(h^2)$ by the same analysis. When $M=1$, we may obtain the result that the derivative again has magnitude $O(h)$ following a similar reasoning as \eqref{eq:analysis O(h) M=1} upon setting $\max(t_{\ell_1 -1},t_{k_2 -1}) < s_1 < \min(t_{\ell_1 +1},t_{k_2 +1})$.

\end{itemize}

\paragraph{(II) If $(k_1,\ell_1) \times (k_2,\ell_2) \in \partial \Omega_{n,m} \times  \mathring{\Omega}_{n,m}$}
In this case, we have $k_1 =n$ or $\ell_1 =m$. One can easily check that the derivative vanishes when $M=1$. For $M>1$, we may first assume $k_1 = n$. To find out the leading order term in each $\frac{\partial^2 \Fs_M(\xib_{n,m})}{\partial G_{n,\ell_1}^{(p_1 q_1)}\partial G_{k_2,\ell_2}^{(p_2 q_2)}} $, we consider the following two cases:

\begin{itemize}
\item If $|k_2 - \ell_1| \le 1$, we require $\max(t_{\ell_1 -1},t_{k_2 -1}) < s_M < \min(t_{\ell_1 +1},t_{k_2 +1})$ and $t_{\ell_2 -1} < s_{M-1} < t_{\ell_2 +1}$ so that we can set $I_h \Xi(s_{j_1},s_{j_1 -1}) = I_h \Xi(t_n,s_M)$ and $I_h \Xi(s_{j_2},s_{j_2 -1}) = I_h \Xi(s_M,s_{M-1})$ in \eqref{claim:non zero derivative}. Since in this case we need to restrict at least $s_M$ and $s_{M-1}$, we have $\frac{\partial^2  F_1(\xib_{n,m})}{\partial G_{n,\ell_1}^{(p_1 q_1)}\partial G_{k_2,\ell_2}^{(p_2 q_2)}}  \sim O(h^2)$.  

\item If $|\ell_1 - k_2| > 1$, we require $t_{\ell_1 -1 } < s_M < t_{\ell_1 + 1}$ and $t_{k_2 -1} < s_{j+1} < t_{k_2 +1},t_{\ell_2 -1} < s_j < t_{\ell_2 +1}$ for some $1 \le j\le M-1$ so that we set $I_h \Xi(s_{j_1},s_{j_1 -1}) = I_h \Xi(t_n,s_M)$ and $I_h \Xi(s_{j_2},s_{j_2 -1}) = I_h \Xi(s_{j+1},s_{j})$ in \eqref{claim:non zero derivative}. Since in this case we need to restrict at least $s_M,s_{j+1}$ and $s_j$, we have $\frac{\partial^2  F_1(\xib_{n,m})}{\partial G_{n,\ell_1}^{(p_1 q_1)}\partial G_{k_2,\ell_2}^{(p_2 q_2)}}  \sim O(h^3)$.   
\end{itemize}

Similar results can be obtained for the case when $\ell_1 =m$. So we now have the conclusion \eqref{eq:2ed deriv bdxint}.

\paragraph{(III) If $(k_1,\ell_1) \times (k_2,\ell_2) \in   \mathring{\Omega}_{n,m} \times \partial \Omega_{n,m}$} The reasoning is similar to (II).

\paragraph{(IV) If $(k_1,\ell_1)\times(k_2,\ell_2) \in  \mathring{\Omega}_{n,m} \times  \mathring{\Omega}_{n,m} $} Again, we can check that the derivative is non-zero only when $M>1$ and we may assume $k_1 \geq k_2$. We also have the following two cases:

\begin{itemize}

\item If $|k_2-\ell_1| \le 1$, we require $t_{k_1 -1}  < s_{j+2} < t_{k_1 +1}$, $\max(t_{\ell_1 -1},t_{k_2-1})  < s_{j+1} < \min(t_{\ell_1 +1},t_{k_2 + 1})$ and $t_{\ell_2 -1}  < s_{j} < t_{\ell_2 +1}$ for some $1\le j\le M-2$ so that we can set $I_h \Xi(s_{j_1},s_{j_1 -1}) = I_h \Xi(s_{j+2},s_{j+1})$ and $I_h \Xi(s_{j_2},s_{j_2 -1}) = I_h \Xi(s_{j+1},s_{j})$ in \eqref{claim:non zero derivative}. Since in this case we need to restrict at least $s_{j+2},s_{j+1}$ and $s_j$, we have $\frac{\partial^2  F_1(\xib_{n,m})}{\partial G_{k_1,\ell_1}^{(p_1 q_1)}\partial G_{k_2,\ell_2}^{(p_2 q_2)}}  \sim O(h^3)$.  

\item If $|k_2-\ell_1| > 1$, we require $t_{k_i -1}  < s_{j_i + 1} < t_{k_i +1}$ and $t_{\ell_i -1}  < s_{j_i} < t_{\ell_i +1}$ with $i=1,2$ for $1\le j_i \le M-1$ and $j_1 - j_2 > 1$ and set $I_h \Xi(s_{j_1},s_{j_1 -1}) = I_h \Xi(s_{j_1+1},s_{j_1})$ and $I_h \Xi(s_{j_2},s_{j_2 -1}) = I_h \Xi(s_{j_2 +1},s_{j_2})$ in \eqref{claim:non zero derivative}. Since in this case we need to restrict at least $s_{j_1+1},s_{j_1},s_{j_2 + 1}$ and $s_{j_2}$, we have $\frac{\partial^2  F_1(\xib_{n,m})}{\partial G_{k_1,\ell_1}^{(p_1 q_1)}\partial G_{k_2,\ell_2}^{(p_2 q_2)}}  \sim O(h^4)$.  

\end{itemize}

Similar analysis and result can be given for $k_1 < k_2$. Therefore, we arrive at \eqref{eq:2ed deriv intxint}.

\bibliographystyle{plain}
\bibliography{inchworm_reference}
\end{document}

%% file: images/fig_mesh_and_order.tex
\begin{figure}[h]
\centering

\begin{tikzpicture}
 
   \draw [<->,thick] (5.5,0)--(0,0)--(0,5.5);
   \node[right] at (5.5,0) {$\sa$};\node[above] at (0,5.5) {$s_i$};

\draw [->,line width=0.7mm] (0.5,0.5)--(0.5,0);
\draw [->,line width=0.7mm] (1,1)--(1,0);
\draw [->,line width=0.7mm] (1.5,1.5)--(1.5,0);
\draw [->,line width=0.7mm] (2,2)--(2,0);
\draw [->,line width=0.7mm] (2.5,2.5)--(2.5,0);
\draw [->,line width=0.7mm] (3,3)--(3,0);
\draw [->,line width=0.7mm] (3.5,3.5)--(3.5,0);
\draw [->,line width=0.7mm] (4,4)--(4,0);
\draw [->,line width=0.7mm] (4.5,4.5)--(4.5,0);
\draw [->,line width=0.7mm] (5,5)--(5,0);

\draw plot[only marks,mark=*, mark options={color=black, scale=1}] coordinates {(0.5,0) (1,0) (1,0.5) (1.5,0) (1.5,0.5) (1.5,1) (2,0) (2,0.5) (2,1) (2,1.5)};

\draw plot[only marks,mark=*, mark options={color=black, scale=1}] coordinates {(3,0) (3,0.5) (3,1) (3,1.5) (3,2) (3.5,0) (3.5,0.5) (3.5,1) (3.5,1.5) (3.5,2) (4,0) (4,0.5) (4,1) (4,1.5) (4,2) (4.5,0) (4.5,0.5) (4.5,1) (4.5,1.5) (4.5,2) (5,0) (5,0.5) (5,1) (5,1.5) (5,2)};

\draw plot[only marks,mark=*, mark options={color=black, scale=1}] coordinates {(3.5,3) (4,3) (4,3.5) (4.5,3) (4.5,3.5) (4.5,4) (5,3) (5,3.5) (5,4) (5,4.5)};

\draw (0,0)--(5,5);
\draw (0.5,0)--(5,4.5);
\draw (1,0)--(5,4);
\draw (1.5,0)--(5,3.5);
\draw (2,0)--(5,3);
\draw (2.5,0)--(5,2.5);
\draw (3,0)--(5,2);
\draw (3.5,0)--(5,1.5);
\draw (4,0)--(5,1);
\draw (4.5,0)--(5,0.5);
\draw (0.5,0.5)--(5,0.5);
\draw (1,1)--(5,1);
\draw (1.5,1.5)--(5,1.5);
\draw (2,2)--(5,2);
\draw (2.5,2.5)--(5,2.5);
\draw (3,3)--(5,3);
\draw (3.5,3.5)--(5,3.5);
\draw (4,4)--(5,4);
\draw (4.5,4.5)--(5,4.5);

 \draw[dotted,thick] (2.5,2.5)--(0,2.5);\node[below] at (2.5,0) {$t$}; 
 \node[below] at (5,0) {$2t$}; 
 \node[left] at (0,2.5) {$t$}; 
  \node[below left] at (0,0) {$O$};

\draw plot[only marks,mark=*, mark options={color=red, scale=1}] coordinates {(0,0) (0.5,0.5) (1,1) (1.5,1.5) (2,2) (2.5,2.5) (3,3) (3.5,3.5) (4,4) (4.5,4.5) (5,5)};

\draw plot[only marks,mark=*, mark options={color=green, scale=1}] coordinates {(2.5,0) (2.5,0.5) (2.5,1) (2.5,1.5) (2.5,2)};

\draw plot[only marks,mark=*, mark options={color=blue, scale=1}] coordinates {(5,2.5) (4.5,2.5) (4,2.5) (3.5,2.5) (3,2.5)};

  \end{tikzpicture}

\caption{The uniform mesh and the order of computation for $N=5$.}
 \label{fig:mesh and order}
\end{figure}

%% file: images/fig_sets.tex
\begin{figure}[h]
\centering

\begin{tikzpicture}
 
   \draw [<->,thick] (5,0)--(0,0)--(0,5);
   \node[right] at (5,0) {$j$};\node[above] at (0,5) {$k$};
   
   \draw (1,1)--(4.5,4.5)--(4.5,1)--(1,1);

   \draw plot[only marks,mark=x, mark options={color=red, scale=1.5}] coordinates {(1.5,1) (2,1) (2.5,1) (3,1) (3.5,1) (4,1) };

  \draw plot[only marks,mark=otimes, mark options={color=red, scale=1.2}] coordinates
{(4,1.5) (4,2) (4,2.5) (4,3) (4,3.5)};

\draw plot[only marks,mark=*, mark options={color=red, scale=1}] coordinates {(2,1.5) (2.5,1.5) (2.5,2) (3,1.5) (3,2) (3,2.5) (3.5,1.5) (3.5,2) (3.5,2.5) (3.5,3)};

   \draw plot[only marks,mark=x, mark options={color=green, scale=1.5}] coordinates {(4.5,1.5) (4.5,2) (4.5,2.5) (4.5,3) (4.5,3.5) (4.5,4)};

   \draw plot[mark=x, mark options={color=blue, scale=1.5}] coordinates {(4.5,1) };

 \draw[dotted,thick] (1,0)--(1,1);\node[below] at (1,0) {$2$}; 
 
  \draw[dotted,thick] (4,0)--(4,1);\node[below] at (4,0) {$8$}; 
 
       \draw[dotted,thick] (4.5,0)--(4.5,1); \node[below] at (4.5,0) {$9$};
   
        \draw[very thick] (2,1)--(4.5,3.5);

  \end{tikzpicture}

\caption{An example illustrating the elements contained in each set defined in this section when $n=8,m=2$. $\Omega_{n,m}$: all red nodes; $\partial \Omega_{n,m}$:  ``$\otimes$" + red ``$\times$"; $ \mathring{\Omega}_{n,m}$:  ``$\bullet$"; $\Omega^*_{n,m}$: all nodes; $\partial \Omega^*_{n,m}$: all ``$\times$", $ \mathring{\Omega}_{n,m}^{\ast}$: ``$\bullet$" + ``$\otimes$"; $\bar{\Omega}_{n,m}$: all nodes except blue ``$\times$"; $\Gamma_{n,m}(2)$: all red nodes on the thick black line; $\Gamma^*_{n,m}(2)$: all nodes on the thick black line.}
 \label{fig:sets}
\end{figure}

%% file: images/fig_deriv_case1_1.tex
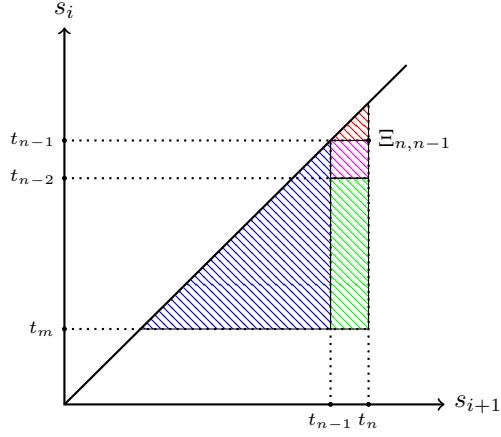
\begin{figure}[h]
\centering

\begin{tikzpicture}

\draw [<->,thick] (5,0)--(0,0)--(0,5);
\node [right] at (5,0) {$s_{i+1}$};
\node [above] at (0,5) {$s_i$};

\draw [thick] (0,0)--(4.5,4.5);
\draw [fill] (3.5,0) circle [radius = .025];\node[below] at (3.5,0)  {\scriptsize $t_{n-1}$};
\draw [fill] (4,0) circle [radius = .025];\node[below] at (4,0)  {\scriptsize $t_{n}$};
\draw [fill] (0,1) circle [radius = .025];\node[left] at (0,1)  {\scriptsize $t_{m}$};
\draw [fill] (0,3) circle [radius = .025];\node[left] at (0,3)  {\scriptsize $t_{n-2}$};
\draw [fill] (0,3.5) circle [radius = .025];\node[left] at (0,3.5)  {\scriptsize $t_{n-1}$};

\draw [dotted,thick] (3.5,0) -- (3.5,3.5); 
\draw [dotted,thick] (4,0) -- (4,4); 
\draw [dotted,thick] (0,1) -- (4,1); 
\draw [dotted,thick] (0,3) -- (4,3); 
\draw [dotted,thick] (0,3.5) -- (4,3.5); 

\draw [fill] (4,3.5) circle [radius = 0.025];
\node[right] at (4,3.5) {\small $\Xi_{n,n-1}$};
\draw [pattern = north west lines,pattern color=red] (3.5,3.5)--(4,3.5)--(4,4)--(3.5,3.5);
\draw [pattern = north west lines,pattern color=magenta] (3.5,3)--(4,3)--(4,3.5)--(3.5,3.5)--(3.5,3);
\draw [pattern = north west lines,pattern color=blue] (1,1)--(3.5,1)--(3.5,3.5)--(1,1);
\draw [pattern = north west lines,pattern color=green] (3.5,1)--(4,1)--(4,3)--(3.5,3)--(3.5,1);

\end{tikzpicture}

\caption{For each $\Is_j$, red triangle: the area where $(s_{i+1},s_i)$ pairs in $\Gs^j_1$ locate; blue triangle: the area where $(s_{i+1},s_i)$ pairs in $\Gs^j_2$ locate; pink square: the area where the value of $I_h \Xi(s_{j+1},s_j)$ is dependent on $\Xi_{n,n-1}$; green rectangle: the area where the value of $I_h \Xi(s_{j+1},s_j)$ is independent from $\Xi_{n,n-1}$.
 }
  \label{fig:deriv_case1_1}
\end{figure}

%% file: images/fig_deriv_case2_1.tex
\begin{figure}[h]
\centering
\begin{tikzpicture}

\draw [->] (1,0) -- (9,0);
\node at (9.25,0) {$t$};

\draw [fill] (2,0) circle [radius=.025];
\node at (2,0) {$($};
\node at (2,-0.4) { $t_{\ell-1}$};
\draw [fill] (3,0) circle [radius=.025];
\draw [fill] (2.5,0) circle [radius=.025];
\node at (2.5,0.4) { $s_j$};
\node at (3,-0.4) { $t_\ell$};
\draw [fill] (3,0) circle [radius=.025];
\draw [fill] (4,0) circle [radius=.025];
\node at (4,0) {$)$};
\node at (4,-0.4) { $t_{\ell+1}$};
\node at (4.75,-0.4) {$\cdots$};
\draw [fill] (5.5,0) circle [radius=.025];
\node at (5.5,0) {$($};
\node at (5.5,-0.4) { $t_{k-1}$};
\draw [fill] (6.5,0) circle [radius=.025];
\node at (6.5,-0.4) { $t_k$};
\draw [fill] (7,0) circle [radius=.025];
\node at (7,0.4) { $s_{j+1}$};
\draw [fill] (7.5,0) circle [radius=.025];
\node at (7.5,0) {$)$};
\node at (7.5,-0.4) { $t_{k+1}$};
\
\end{tikzpicture}
\caption{Locations of $s_j$ and $s_{j+1}$ in \emph{Case 1}.}
  \label{fig:deriv_case2_1}
\end{figure}
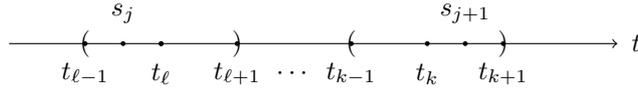

%% file: images/fig_deriv_case2_2.tex
\begin{figure}[h]
\centering

\begin{tikzpicture}

\draw [->] (1,0) -- (9,0);
\node at (9.25,0) {$t$};
\draw [fill] (2,0) circle [radius=.025];
\node at (2,0) {$($};
\node at (2,-0.4) { $t_{\ell-1}$};
\draw [fill] (4,0) circle [radius=.025];
\node at (4,0) {$($};
\node at (4,-0.4) { $t_\ell$};\node at (4,-0.8){\small $(t_{k-1})$};
\draw [fill] (6,0) circle [radius=.025];
\node at (6,0) {$)$};
\node at (6,-0.4) { $t_{\ell+1}$}; \node at (6,-0.8) {\small $(t_k)$};
\draw [fill] (8,0) circle [radius=.025];
\node at (8,0) {$)$};
\node at (8,-0.4) { $t_{k+1}$};

\draw [fill] (4.2,0) circle [radius=.025];\draw [fill] (4.4,0) circle [radius=.025];\draw [fill] (4.6,0) circle [radius=.025];\draw [fill] (4.8,0) circle [radius=.025];\draw [fill] (5,0) circle [radius=.025];\draw [fill] (5.2,0) circle [radius=.025];\draw [fill] (5.4,0) circle [radius=.025];\draw [fill] (5.6,0) circle [radius=.025];\draw [fill] (5.8,0) circle [radius=.025];

\draw [fill] (1.5,0) circle [radius=.025];\draw [fill] (8.5,0) circle [radius=.025];

\draw[decoration={brace,amplitude=10pt},decorate] (4,0)--(6,0);\node at (5,0.6) {\small $s_{u+1},s_{u+2},\dots,s_{u+v}$};

\node at (1.5,0.4) {\small $s_u$};\node at (8.8,0.4) {\small $s_{u+v+1}$};

\node at (0,0) {$\Ks^{u,v}_1$:};

\end{tikzpicture}

\begin{tikzpicture}

\draw [->] (1,0) -- (9,0);
\node at (9.25,0) {$t$};
\draw [fill] (2,0) circle [radius=.025];
\node at (2,0) {$($};
\node at (2,-0.4) { $t_{\ell-1}$};
\draw [fill] (4,0) circle [radius=.025];
\node at (4,0) {$($};
\node at (4,-0.4) { $t_\ell$};\node at (4,-0.8){\small $(t_{k-1})$};
\draw [fill] (6,0) circle [radius=.025];
\node at (6,0) {$)$};
\node at (6,-0.4) { $t_{\ell+1}$}; \node at (6,-0.8) {\small $(t_k)$};
\draw [fill] (8,0) circle [radius=.025];
\node at (8,0) {$)$};
\node at (8,-0.4) { $t_{k+1}$};

\draw [fill] (4.2,0) circle [radius=.025];\draw [fill] (4.4,0) circle [radius=.025];\draw [fill] (4.6,0) circle [radius=.025];\draw [fill] (4.8,0) circle [radius=.025];\draw [fill] (5,0) circle [radius=.025];\draw [fill] (5.2,0) circle [radius=.025];\draw [fill] (5.4,0) circle [radius=.025];\draw [fill] (5.6,0) circle [radius=.025];\draw [fill] (5.8,0) circle [radius=.025];

\draw [fill] (3,0) circle [radius=.025];\draw [fill] (8.5,0) circle [radius=.025];

\draw[decoration={brace,amplitude=10pt},decorate] (4,0)--(6,0);\node at (5,0.6) {\small $s_{u+1},s_{u+2},\dots,s_{u+v}$};

\node at (3,0.4) {\small $s_u$};\node at (8.8,0.4) {\small $s_{u+v+1}$};

\node at (0,0) {$\Ks^{u,v}_{2,L}$:};

\end{tikzpicture}

\begin{tikzpicture}
\draw [->] (1,0) -- (9,0);
\node at (9.25,0) {$t$};
\draw [fill] (2,0) circle [radius=.025];
\node at (2,0) {$($};
\node at (2,-0.4) { $t_{\ell-1}$};
\draw [fill] (4,0) circle [radius=.025];
\node at (4,0) {$($};
\node at (4,-0.4) { $t_\ell$};\node at (4,-0.8){\small $(t_{k-1})$};
\draw [fill] (6,0) circle [radius=.025];
\node at (6,0) {$)$};
\node at (6,-0.4) { $t_{\ell+1}$}; \node at (6,-0.8) {\small $(t_k)$};
\draw [fill] (8,0) circle [radius=.025];
\node at (8,0) {$)$};
\node at (8,-0.4) { $t_{k+1}$};

\draw [fill] (4.2,0) circle [radius=.025];\draw [fill] (4.4,0) circle [radius=.025];\draw [fill] (4.6,0) circle [radius=.025];\draw [fill] (4.8,0) circle [radius=.025];\draw [fill] (5,0) circle [radius=.025];\draw [fill] (5.2,0) circle [radius=.025];\draw [fill] (5.4,0) circle [radius=.025];\draw [fill] (5.6,0) circle [radius=.025];\draw [fill] (5.8,0) circle [radius=.025];

\draw [fill] (1.5,0) circle [radius=.025];\draw [fill] (7,0) circle [radius=.025];

\draw[decoration={brace,amplitude=10pt},decorate] (4,0)--(6,0);\node at (5,0.6) {\small $s_{u+1},s_{u+2},\dots,s_{u+v}$};

\node at (1.5,0.4) {\small $s_u$};\node at (7.2,0.4) {\small $s_{u+v+1}$};

\node at (0,0) {$\Ks^{u,v}_{2,R}$:};

\end{tikzpicture}

\begin{tikzpicture}
\draw [->] (1,0) -- (9,0);
\node at (9.25,0) {$t$};
\draw [fill] (2,0) circle [radius=.025];
\node at (2,0) {$($};
\node at (2,-0.4) { $t_{\ell-1}$};
\draw [fill] (4,0) circle [radius=.025];
\node at (4,0) {$($};
\node at (4,-0.4) { $t_\ell$};\node at (4,-0.8){\small $(t_{k-1})$};
\draw [fill] (6,0) circle [radius=.025];
\node at (6,0) {$)$};
\node at (6,-0.4) { $t_{\ell+1}$}; \node at (6,-0.8) {\small $(t_k)$};
\draw [fill] (8,0) circle [radius=.025];
\node at (8,0) {$)$};
\node at (8,-0.4) { $t_{k+1}$};

\draw [fill] (4.2,0) circle [radius=.025];\draw [fill] (4.4,0) circle [radius=.025];\draw [fill] (4.6,0) circle [radius=.025];\draw [fill] (4.8,0) circle [radius=.025];\draw [fill] (5,0) circle [radius=.025];\draw [fill] (5.2,0) circle [radius=.025];\draw [fill] (5.4,0) circle [radius=.025];\draw [fill] (5.6,0) circle [radius=.025];\draw [fill] (5.8,0) circle [radius=.025];

\draw [fill] (3,0) circle [radius=.025];\draw [fill] (7,0) circle [radius=.025];

\draw[decoration={brace,amplitude=10pt},decorate] (4,0)--(6,0);\node at (5,0.6) {\small $s_{u+1},s_{u+2},\dots,s_{u+v}$};

\node at (3,0.4) {\small $s_u$};\node at (7.2,0.4) {\small $s_{u+v+1}$};

\node at (0,0) {$\Ks^{u,v}_3$:};

\end{tikzpicture}

\caption{Distribution of the time sequence $\vec{\sb}$ in $\Ks^{u,v}_1$, $\Ks^{u,v}_{2,L}$, $\Ks^{u,v}_{2,R}$, $\Ks^{u,v}_3$.}
\label{fig:deriv_case2_2}
\end{figure}